\documentclass[11pt]{article}
\usepackage{amsthm, amsbsy, amsmath, amscd, mathrsfs}
\usepackage[margin=0.96in]{geometry}
\usepackage{mathtools}
\usepackage[toc, page]{appendix}
\usepackage{graphicx}
\usepackage{yfonts}
\usepackage{enumerate}
\usepackage{color,framed}  
\usepackage[colorlinks=true, pdfstartview=FitV, linkcolor=blue,citecolor=blue, urlcolor=blue]{hyperref}
\usepackage[color=yellow]{todonotes}
\usepackage{txfonts}

\let\oldbibliography\thebibliography
\renewcommand{\thebibliography}[1]{%
\oldbibliography{#1}%
\setlength{\itemsep}{0pt}%
}

\newtheorem{theorem}{Theorem}[section]
\newtheorem{lemma}[theorem]{Lemma}

\newtheorem{corollary}[theorem]{Corollary}
\newtheorem{remark}[theorem]{Remark}       
\theoremstyle{definition}
\newtheorem{definition}[theorem]{Definition}
\newtheorem{assumption}[theorem]{Assumption}
\newtheorem{example}[theorem]{Example} 

\numberwithin{equation}{section}

\DeclareMathAlphabet\mathbfcal{OMS}{cmsy}{b}{n}
\makeatletter
\@tfor\next:=abcdefghijklmnopqrstuvwxyzABCDEFGHIJKLMNOPQRSTUVWXYZ\do{%
\def\command@factory#1{%
\expandafter\def\csname b#1\endcsname{\mathbf{#1}}
\expandafter\def\csname bb#1\endcsname{\mathbb{#1}}
\expandafter\def\csname cl#1\endcsname{\mathcal{#1}}
\expandafter\def\csname bcl#1\endcsname{\mathbfcal{#1}}
}
\expandafter\command@factory\next
}

\newcommand{\scp}[2]{{\big\langle {#1}\, , \, {#2}\big\rangle}}

\newcommand{\sym}[1]{\boldsymbol{#1}}
\newcommand{\mbb}[1]{\mathbb{#1}}
\newcommand{\mfk}[1]{\mathfrak{#1}}

\newcommand{\rmd}{\textnormal{d}}

\title{\vspace{-8mm}\textbf{\Large 
Variational principles for fluid dynamics
on rough paths}}
\author{
Dan Crisan\thanks{\footnotesize Department of Mathematics, Imperial College, London SW7 2AZ, UK,, \href{mailto:d.crisan@imperial.ac.uk}{d.crisan@imperial.ac.uk}}
\and
Darryl D. Holm\thanks{\footnotesize Department of Mathematics, Imperial College, London SW7 2AZ, UK, \href{mailto:d.holm@imperial.ac.uk }{d.holm@imperial.ac.uk}}
\and
James-Michael Leahy\thanks{\footnotesize Department of Mathematics, Imperial College, London SW7 2AZ, UK, \href{mailto:j.leahy@imperial.ac.uk}{j.leahy@imperial.ac.uk}}
\and 
Torstein Nilssen\thanks{\footnotesize Institute of Mathematics, University of Agder, Kristiansand S, Norway, \href{mailto:torstein.nilssen@uia.no}{torstein.nilssen@uia.no}}
}

\date{}

\begin{document}

\maketitle
\vspace{-5mm}
\begin{abstract}
In recent works, beginning with \cite{holm2015variational}, several stochastic geophysical fluid dynamics (SGFD) models have been derived from variational principles. In this paper, we introduce a new framework for parametrization schemes (PS) in GFD. We derive a class of rough geophysical fluid dynamics (RGFD) models as critical points of rough action functionals using the theory of controlled rough paths. These RGFD models characterize Lagrangian trajectories in fluid dynamics as geometric rough paths (GRP) on the manifold of diffeomorphic maps. We formulate three constrained variational approaches for the derivation of these models. The first is the  Clebsch formulation, in which the constraints are imposed as rough advection laws. The second is the Hamilton-Pontryagin formulation, in which the constraints are imposed as right-invariant rough vector fields. And the third is the Euler--Poincar\'e formulation, in which the variations are constrained. These constrained rough variational principles lead directly to the Lie--Poisson Hamiltonian formulation of fluid dynamics on GRP. The GRP framework preserves the geometric structure of fluid dynamics obtained by using Lie group reduction to pass from Lagrangian to Eulerian variational principles, yielding a rough formulation of the Kelvin circulation theorem. The rough formulation enhances its stochastic counterpart developed in \cite{holm2015variational}, and extended to semimartingales in \cite{street2020driv}. For example, the rough-path variational approach includes non-Markovian perturbations of the Lagrangian fluid trajectories. In particular, memory effects can be introduced  through a judicious choice of the rough path (e.g. a realization of a fractional Brownian motion).
In the particular case when the rough path is a realization of a semimartingale, we recover the SGFD  models in \cite{holm2015variational,street2020driv}.  However, by eliminating the need for stochastic variational tools, we retain a \emph{pathwise} interpretation of the Lagrangian trajectories. In contrast, the Lagrangian trajectories in the stochastic framework are described by stochastic integrals, which do not have a pathwise interpretation. Thus, the rough path formulation restores this property.  
\end{abstract}

\renewcommand{\baselinestretch}{0.9}\normalsize
\setcounter{tocdepth}{3}
\tableofcontents
\renewcommand{\baselinestretch}{1.0}\normalsize

\section{Introduction}

The present work aims to transfer the fundamental properties of deterministic fluid dynamics derived by Hamilton's principle into their formulation on geometric rough paths.  Recent work  \cite{crisan2021solution} concerning solution properties of Euler fluid dynamics on rough paths demonstrates the efficacy of this approach to produce previously unavailable results, such as the Beale-Kato-Majda blowup criterion for ideal fluid solutions on geometric rough paths. 
 
To set the scenery we discuss some aspects of the Hamilton's principle variational approach to modelling fluid dynamics behaviour using its Lie group symmetry. Hamilton's principle states that critical points $\delta S=0$ of a time integral $S=\int_{0}^{T}L\,dt$ with Lagrangian functional $L:TM\to\mathbb{R}$ determine dynamical equations on a manifold $M$.  Since its inception, Hamilton's principle has provided a systematic mathematical framework for scientific investigation. For example, Lie symmetries of Hamilton's principle encode conservation laws (i.e., Noether's theorem \cite{noether1918invarianten, kosmann2011noether, holm2009geometric}) on whose level sets the ensuing dynamics takes place. Lie symmetries of Hamilton's principle also reduce the number of dynamical degrees of freedom to equivalence classes of observables that transform under the corresponding Lie group. 

The reduced Hamilton's principle leading to the Euler-Poincar\'e equations for ideal continuum mechanics was only  developed recently in \cite{holm1998euler}. For the flow of ideal fluids in a fixed domain $M\subset\mathbb{R}^n$, Lie group symmetry reduces the number of degrees of freedom to the equivalence classes of observables that transform under pull-back by smooth invertible maps with smooth inverses  ($\phi\in{\rm Diff}(M)$, diffeomorphisms) in which the composition of functions is understood as a Lie group operation. Euler fluid dynamics is then recast as a flow map $\phi_t$ which defines a \emph{time-dependent} geodesic curve on the manifold of diffeomorphisms, cf., \cite{arnold1966principe, holm1998euler}.

 The present approach is based on the premise that Euler's fluid equations arise from Hamilton's variational principle for geodesic flow on the manifold of diffeomorphisms with respect to the metric defined by the kinetic energy of the fluid,  \cite{arnold1966principe, ebin1970groups}. The variations are constrained by the condition of right-invariance of the velocity vector field. Hamilton's principle for fluids is modified when advection by the fluid motion under the action of the diffeomorphisms carries fluid properties such as mass and heat, whose contribution to the thermodynamic equation of state affects the motion \cite{holm1998euler}. These advected fluid quantities are said to follow \emph{Lagrangian trajectories} of fluid parcels in the flow. Since the Lagrangian trajectories for Euler's ideal fluid equations are \emph{pushed forward} by time-dependent diffeomorphic maps, these trajectories may be regarded as curves parametrized by time on the manifold of smooth invertible maps (diffeomorphisms) \cite{ebin1970groups}.
 
Preserving the fundamental structure derived from Hamilton's principle in the course of more general fluid modelling is paramount.  These theoretical considerations have helped in developing Hamilton's principle modelling for stochastic continuum mechanics in \cite{holm2015variational}. In turn, this new development has recently led to new methods for stochastic data assimilation using particle-filtering in geophysical fluid dynamics (GFD) \cite{cotter2019numerically}.

The need for robust and computationally efficient Parametrization Schemes (PS) that model the effects of fast sub-grid scale physics and other unresolved processes is well understood in Geophysical Fluid Dynamics (GFD). See, for example,  \cite{ghil2019physics}, for a recent overview. Stochastic Parameterization Schemes (SPS) have the additional ability to introduce model uncertainty \cite{berner2017stochastic} naturally.  SPS have improved the probabilistic skill of the ensemble weather forecasts by increasing their reliability and reducing the error of the ensemble mean. The coming years are likely to see a further increase in the use of SPS in ensemble methods in forecasts and assimilation. This, however, will put increasing demands on the methods used to represent computational model uncertainty in the dynamical core and other components of the Earth system while maintaining overall computational efficiency \cite{leutbecher2017stochastic}.

The preservation of geometrical structure and physicality of fluid dynamics can serve as a guiding principle in designing robust PS for GFD. The PS are meant to preserve predictive power, accuracy, and computational efficiency in modeling the effects of both: (i) unresolved phenomena due to the known but unresolved rapid sub-grid scale physics, as well as (ii) uncertainty due to unknown bias in the data. Thus, in ensemble computations, PS face a daunting combination of tasks. 

In this paper, we propose a structured approach for parametrization of the rapid scales of fluid motion by using a temporally rough vector field in the framework of geometric rough paths (GRP) \cite{friz2014course}, which we call Geometric Rough Path Parametrization Schemes (GRPPS). Namely, we will develop a new class of variational principles for fluids that model resolved and unresolved motions of fluid advection as GRPPS. Critical points of our rough-path constrained variational principles are rough partial differential equations (RPDEs), whose dynamics incorporate both the resolved-scale fluid velocity and the effects of the unresolved fluctuations.

In the particular case when the generating rough path is a straight line, or, more generally, a smooth curve, the GRPPS approach introduced here reduces to a PS approach obtained through classical/deterministic variational principles (see Section \ref{sec:var_smooth_paths}). Similarly, when the generating rough path is a  realization of a Brownian motion (or, more generally, of a semimartingale process),  GRPPS specializes to a pathwise formulation of the SPS characterized through the stochastic variational principles first introduced in \cite{holm2015variational}. In other words, this work enhances the mathematical framework of \emph{Stochastic Advection by Lie Transport} (SALT), in which the Lagrangian trajectories are treated as time-dependent Stratonovich stochastic processes \cite{holm2015variational}.

Non-Markovian models include models with memory which are of interest in ocean dynamics, see, e.g.,  \cite{WoodsHole, Eos, Nasa, hottestnat, woods1981memory, lebeaupin2013ocean, primeau2006ocean}.  The variational treatment of fluid dynamics on rough paths enables the introduction of such models. This can be accomplished, for example, by realising the rough path as a fractional Brownian motion, or as a more general Gaussian process with suitably chosen time correlation. There is growing evidence that non-Markovianity improves models of the effects of fast sub-grid scales on the resolved scales (see, e.g., \cite{arnold2013stochastic, gagne2020machine, chattopadhyay2020data, levine2021framework}). It stands to reason that such models could be useful in parametrising the sub-grid scales of real fluids. Our framework, in particular, includes spatially-local, non-Markovian SPS by using rough paths (rather, for example,  state-delay terms) to model the sub-grid scale terms. Since our framework retains the core geometric structure of fluid mechanics, one can expect the solution properties of our equations would track those of  the deterministic unperturbed equations  (e.g., stability up to blow-up time \cite{crisan2021solution}). This sort of fidelity would be important in the calibration of our models, especially for those calibrations that use the modern `solver-in-the loop' estimation procedures \cite{um2020solver}.

\paragraph{The contents of this paper.}
The overall goal of the present paper is to formulate rigorously in Theorems \ref{thm:Clebsch} and \ref{thm:HamPont} variational principles for ideal fluid dynamics with advection of fluid quantities along Geometric Rough Paths. To achieve this goal, our first aim is to derive a rough version of the Lie chain rule in Theorem \ref{thm:Lie_chain} leading to the GRP version of the classical Reynolds transport formula in Corollary \ref{corr:Reyn_trans}. The Reynolds transport formula for momentum density encapsulates the force law which governs fluid motion. Mathematically, the formula describes the rate of change of the integral of the fluid momentum density over a moving control volume that is being transported by the rough (fluid) flow along a GRP. Thus, our key result is the Lie chain rule formula \eqref{eq:pull-back_tensor} in Theorem \ref{thm:Lie_chain} for the rough differential (or increment in time) of the pull-back and push-forward of a tensor-field-valued GRP by a rough flow. The Lie chain rule formula \eqref{eq:pull-back_tensor} is intuitive and natural because it follows from the extension of ordinary calculus to GRP. The formula leads to a unified, stable, and flexible framework for modelling fluids whose Lagrangian parcels move along temporally rough paths. Theorem \ref{thm:Lie_chain} for the Lie chain rule is the foundation on which the other contributions of the paper rest.

In order to derive the momentum and advection equations satisfied by the critical points of our variational principles, we required a rough version of the  fundamental lemma of the calculus of variations. A version is formulated and proved in Section \ref{sec:FLCRC}. As far as we are aware, this is a new result.

Our work is an example of the rigorous content of the Malliavin transfer principle, which says that geometric constructions involving manifold-valued curves can be extended to manifold-valued rough paths by replacing classical calculus with geometric rough-path calculus (see, e.g. \cite{schwartz1984semimartingales,emery1990two, cass2011rough, boutaib2015new,driver2018global,armstrong2020non}). More specifically, we show that deterministic geometric continuum mechanics can be extended to rough-path geometric continuum mechanics.

\bigskip

\noindent
The paper is structured as follows:
\begin{itemize}
\item
Section \ref{sec:Clebsch_var_princ} formulates the first variational principle for fluid dynamics on geometric rough paths in Theorem \ref{thm:Clebsch}, by imposing the \emph{Clebsch constraint} for the advection fluid quantities along rough paths. 
\item
Section \ref{sec:LieChainRule} formulates the Lie Chain Rule Theorem \ref{thm:Lie_chain} and the Reynolds Transport Corollary \ref{corr:Reyn_trans} for geometric rough paths. The Reynolds transport formula in the special case of one-forms yields the rough Kelvin--Noether Theorem \ref{thm:Kelvin}.  Section \ref{sec:LieChainRule} also formulates the Hamilton-Pontryagin variational principle for rough paths in Theorem \ref{thm:HamPont}, which imposes the constraint that the vector fields which generate the Lagrangian trajectories are right-invariant under diffeomorphisms whose time dependence is rough. The Clebsch and Hamilton-Pontryagin variational principles for rough paths correspond to those derived in the SALT approach in \cite{holm2015variational} and \cite{gay2018stochastic}, respectively.
Next we formulate the Euler--Poincar\'e constrained variational principle for ideal fluid motion on GRP in Theorem \ref{EP-thm}. Here, we pose an open problem regarding the construction of variations used in this principle. Finally, we develop the Lie--Poisson Hamiltonian formulation of fluid dynamics on GRP in Corollary \ref{cor-LPHamForm}.
\item Section \ref{sec:examples} provides three  examples of fluid equations on GRP. These are: i) the rough Euler equation for incompressible fluid flow; ii) the rough Camassa-Holm equation and its limiting case, the rough Burgers equation; and iii) the equations for ideal compressible adiabatic fluid dynamics on GRP. 
\item Finally, Section \ref{sec:proofs} contains the proofs of the main results formulated in Section \ref{sec:LieChainRule}. 
\end{itemize}

In addition, the paper contains five Appendices which are meant to provide notation and 
background information, including proofs of key technical results invoked in the main text, a simple example of our procedure in the setting of smooth paths, and additional history and  motivation. The first two Appendices are essential and contain together the key relationships and definitions needed in both geometric mechanics and the theory of rough paths for the present work, which as far as we know are found together nowhere else. The latter three Appendices provide additional information and motivation for the theory of rough paths. 

Appendix \ref{App-notation_and_background} defines the notation we use and summarises  the essential background and results for both rough paths and geometric mechanics that we use in the text. We choose to put this section in the appendix rather than in the main text since different classes of readers may be familiar with at least some of our notation, and might wish to see the statements of the main results presented first. The main text will  refer to sections in Appendix \ref{App-notation_and_background} as needed if we think a notation is not standard. Appendix \ref{App-auxil-results} contains proofs of selected technical results which facilitate the proofs in the main text. Appendix \ref{sec:var_smooth_paths} illustrates the variational principles we use in the example of a homogeneous incompressible fluid flow perturbed by spatially and temporally smooth noise. This example serves as a guide for introducing rough perturbations into the variational principles for more general fluid theories.  Appendix \ref{sec:motivation} provides a short history and motivation in the development of the theory of rough paths and Appendix \ref{sec:Gaussian_rough_paths} discusses the concrete example of Gaussian rough paths and provides additional references to this important class of rough paths.

\medskip
 
\noindent
\textbf{Contributions of this paper.}

This paper offers a variational framework that connects Geometric Rough Path Theory with Geophysical Fluid Dynamics, hopefully to the benefit of both fields. The geometric variational approach followed here may enhance the development of mathematical and numerical models in a range of investigations in Weather Prediction, Data Assimilation, Ocean Dynamics, Atmospheric Science, perhaps even Turbulence. For example, the model development may benefit from theoretical results (stability results, large deviation principles, splitting schemes) for random dynamical systems arising from rough partial differential equations \cite{crisan2021solution}. In turn, the new connections between GFD and geometric rough paths may become a fruitful source of open problems in mathematics.

This paper introduces a GRPPS framework that transcends the scope of either deterministic or stochastic parametrization schemes by allowing GRP with H\"older index $\alpha\in \left(\frac{1}{3}, 1\right ]$.\footnote{The analysis presented here can be extended to $\alpha\in \left(0, 1\right ]$ at the expense of more elaborate computations } The case $\alpha=1$ recovers deterministic fluid dynamics  (see, e.g., Section \ref{sec:var_smooth_paths}).  The case $\alpha=1/2-\epsilon$ $(\epsilon\ll 1)$  gives a pathwise characterization of SPS. Widening the choice of  H\"older index provides a broader scope for modelling with PS. Indeed, one may also include models which are non-Markovian (for example, by choosing the rough path as a realization of a fractional Brownian motion, or of a more general Gaussian process with suitably chosen time dependence).

\bigskip

\noindent
The GRPPS presented here possess the following fundamental properties:

\begin{itemize}
\item Being derived from Hamilton's variational principle, they preserve the geometric structure of fluid dynamics \cite{holm1998euler}. 

\item They satisfy a Kelvin circulation theorem, which is the classical essence of fluid flow. 

\item They are consistent with the modern mathematical formulation of fluid flow as geodesic flows on the manifold of smooth invertible transformations, with respect to the metric associated with the fluid's kinetic energy \cite{arnold1966principe}. 

\item 
They accommodate Pontryagin's maximum principle for control in taking a dynamical system from one state to another, especially in the presence of constraints for the state or input controls \cite{bloch2000optimal}.

\end{itemize}

 \paragraph{Open problems}
Following this work, the following problems remain open:

\begin{itemize}

\item Completeness of the constrained velocity variations in formulating the RPDEs in the Euler-Poincar\'e Theorem  \ref{EP-thm}.  

\item Well-posedness of RPDEs derived in Sec.\ref{sec:Euler-Lag_mult} and Sec. \ref{sec:Euler-spaces}.

\item Estimation of the rough path properties and calibration of the GRPPS model from observed or simulated data.

\item The development of pathwise data assimilation methods for the incorporation of data into GRPPS.

\item Uncertainty quantification and forecast analysis using GRPPS.

\end{itemize}

\subsection*{Data availability}
 Data sharing not applicable to this article as no datasets were generated or analysed during the current study.
 
\subsection*{Acknowledgments} 
All of the authors are grateful to our friends and colleagues who have generously offered their time, thoughts and encouragement in the course of this work during the time of COVID-19. 
We are particularly grateful to T. D. Drivas and S. Takao for thoughtful discussions.
DC and DH are grateful for partial support from ERC Synergy Grant 856408 - STUOD (Stochastic Transport in Upper Ocean Dynamics). JML is grateful for partial support from US AFOSR Grant FA9550-19-1-7043 - FDGRP (Fluid Dynamics of Geometric Rough Paths) awarded to DH as PI. TN is grateful for partial support from the DFG via the Research Unit FOR 2402.

\section{The Clebsch variational principle  for geometric rough paths}\label{sec:Clebsch_var_princ}

To streamline the presentation of our results and to provide a ready reference for the reader, we have assembled the notation and background of geometric mechanics and rough paths theory needed for this paper into one place --  Appendix \ref{App-notation_and_background} -- rather than dispersing it sequentially in the main text.

Let $\bZ=(Z,\bbZ)\in \bclC^{\alpha}_T(\mbb{R}^K)$ be a given geometric rough path with H\"older index $\alpha\in (\frac{1}{3}, 1]$ defined in the time interval $t\in [0,T]$. Let $\mfk{X}=\mfk{X}_{\clF_1}$ denote a function space of vector fields. Let  $\mfk{X}^{\vee}=\mfk{X}^{\vee}_{\clF_2}$ denote a function space of one-form densities  such that   the canonical pairing $\langle \cdot, \cdot \rangle_{\mathfrak{X}}: \mfk{X}^{\vee}_{C^\infty}\times \mfk{X}_{C^{\infty}} \rightarrow \bbR$  defined in \eqref{def:canonical_pairing} extends to a continuous  pairing on $\mfk{X}^{\vee}\times\mfk{X}$. For incompressible fluids, we implicitly use the constructions of the divergence-free, or both divergence-free and harmonic-free vector fields and their canonical `duals' (see Definition  \ref{def:canonical_dual_incompressible}). We write all spaces in brief notation without including Riemannian measure $\mu_g$ or other extraneous adornment for a unified treatment of the compressible and incompressible case. That is to say, for incompressible fluid flows, all vector fields (and variations of vector fields) and one-form densities are constrained. Using  Cartan's formula \eqref{eq:Cartan} and the Stokes theorem, one can  show that for all $u\in \mfk{X}_{C^{\infty}}$, the adjoint (see \eqref{def:adjoint}) of the vector-field operation $$\operatorname{ad}_{u}=-\,\pounds_{u}: \mfk{X}_{C^{\infty}}\rightarrow \mfk{X}_{C^{\infty}}$$ relative to the canonical pairing $\langle \cdot, \cdot\rangle_{\mfk{X}}$ is given by $\operatorname{ad}_{u}^*=\pounds_{u}: \mfk{X}_{\clD'}^{\vee}\rightarrow \mfk{X}_{\clD'}^{\vee}$. Thus,  for all $\alpha\otimes D\in \mfk{X}_{C^{\infty}}^{\vee}$, we have
$$
\operatorname{ad}^*_u(\alpha \otimes D)=\pounds_{u}(\alpha\otimes D)=\pounds_{u}\alpha \otimes D + \alpha \otimes \pounds_{u}D=\pounds_{u}\alpha \otimes D + \alpha \otimes (\operatorname{div}_{D}u D).
$$

Let $A$ be a direct summand of alternating form bundles and tensor bundles such that the first component of $A$ is the density bundle $\Lambda^dT^*M$. Let $A^{\vee}$ denote the canonical dual in Section \ref{sec:duals_of_bundles}.  Define $\langle\cdot, \cdot\rangle_{\mfk{A}}: \mfk{A}_{C^{\infty}}^{\vee}\times \mfk{A}_{C^{\infty}}\rightarrow \bbR$ via a sum as explained in \eqref{def:canonical_pairing}. Let  $\mfk{A}=\mfk{A}_{\clF_3}=\Gamma_{\clF_3}(A)$ and $\mfk{A}^{\vee}=\mfk{A}^{\vee}_{\clF_4}=\Gamma_{\clF_4}(A^{\vee})$ be  function spaces  such that the pairing $\langle \cdot, \cdot \rangle_{\mathfrak{A}}$ 
extends to a continuous  pairing on $\mfk{A}^{\vee}\times\mfk{A}$. For all $u\in \mfk{X}_{C^{\infty}}$,  let $\pounds^*_u:\mfk{A}_{C^{\infty}}^{\vee}\rightarrow \mfk{A}_{C^{\infty}}^{\vee}$ denote the adjoint  (see \eqref{def:adjoint}) of the Lie derivative $\pounds_u:\mfk{A}_{C^{\infty}}^{\vee}\rightarrow \mfk{A}_{C^{\infty}}^{\vee}$ defined relative to the canonical pairing $\langle \cdot, \cdot \rangle_{\mathfrak{A}}$.

\begin{definition}[Diamond operator  $(\diamond)$]\label{diamond-def}
We define the bilinear \emph{diamond operator}  $\diamond : \mfk{A}^{\vee}_{C^{\infty}}\times \mfk{A}_{C^{\infty}}\rightarrow \mathfrak{X}^{\vee}_{C^{\infty}}$ via the relation
$$
\langle\lambda \diamond a,u\rangle_{\mfk{X}} = -\,\langle \lambda,\pounds_u a\rangle_{\mfk{A}}, \quad  \forall (\lambda, u,a)\in \mfk{A}^{\vee}_{C^{\infty}}\times \mfk{X}_{C^\infty}\times \mfk{A}_{C^{\infty}}.
$$
We refer the reader to \cite{holm1998euler} and Section \ref{sec:examples} (esp. Section \ref{sec.EulAdiabatCompFuid}) for explicit computations with the diamond operator in fluid dynamics.
We assume that $\diamond$ extends to a continuous operator $\diamond : \mfk{A}^{\vee}\times \mfk{A}\rightarrow \mathfrak{X}^{\vee}$.
\end{definition}

\begin{assumption}\label{asm:Lagrangian} Let $\ell: \mfk{X}\times \mfk{A}\rightarrow \bbR$. Assume there exist  (functional derivatives) $\frac{\sym{\delta} \ell}{\sym{\delta} u}:\mfk{X} \times \mfk{A}\rightarrow\mfk{X}^{\vee}$ and $\frac{\sym{\delta} \ell}{\sym{\delta} a}:\mfk{X}\times \mfk{A}\rightarrow \mfk{A}^{\vee}$ such that for all $(u,a)\in\mfk{X}\times \mfk{A}$ and $(\sym{\delta} u, \sym{\delta} a)\in \mfk{X}_{C^\infty}\times \mfk{A}_{C^{\infty}}$:
\begin{enumerate}[(i)]
\item  
$$
\frac{d}{d\epsilon}\bigg|_{\epsilon=0} \ell(u+\epsilon\sym{\delta} u,a+
\epsilon \sym{\delta} a) =: \left\langle \frac{\sym{\delta} \ell }{\sym{\delta} u}(u,a), \sym{\delta} u  \right\rangle_{\mathfrak{X}}+ \left\langle \frac{\sym{\delta} \ell }{\sym{\delta} a}(u,a), \sym{\delta} a  \right\rangle_{\mfk{A}};
$$
\item  for any sequence $\{(u^n,a^n)\}_{n\in \bbN}\subset \mfk{X}_{C^\infty}\times \mfk{A}_{C^{\infty}}$  such that  $(u^n,a^n)\rightarrow (u,a)$ as $n\rightarrow \infty$ in $\mfk{X}\times \mfk{A}$;
$$
\lim_{n\rightarrow \infty}\left\langle \frac{\sym{\delta} \ell }{\sym{\delta} u}(u^n,a^n), \sym{\delta} u  \right\rangle_{\mathfrak{X}}=\left\langle \frac{\sym{\delta} \ell }{\sym{\delta} u}(u,a),\sym{\delta} u\right\rangle_{\mathfrak{X}} \;   \hbox{ and } \;
\lim_{n\rightarrow \infty}\left\langle \frac{\sym{\delta} \ell }{\sym{\delta} a}(u^n,a^n), \sym{\delta} a  \right\rangle_{\mfk{A}}=\left\langle \frac{\sym{\delta} \ell }{\sym{\delta} a}(u,a),\sym{\delta} a\right\rangle_{\mfk{A}}.
$$
\item the mapping $\frac{\sym{\delta} \ell}{\sym{\delta} u}(\cdot, a) : \mfk{X}\rightarrow \mfk{X}^{\vee}$ is an isomorphism.
\end{enumerate}
\end{assumption}

Let $\xi \in \mfk{X}_{\clD'}^K$ denote a fixed collection of  vector fields.\footnote{It is worth noting that one can make $\xi$ time dependent and  depend on other quantities as in \cite{gay2018stochastic}.} 

\begin{definition} Let $\mathit{Clb}_{\bZ}$ denote the space of  all
$$
\left(u,\ba,\boldsymbol{\lambda}\right) \in C_T^{\alpha}(\mfk{X}) \times  \bclD_{Z,T}( \mfk{A}) \times \bclD_{Z,T}( \mfk{A}^{\vee})
$$
such that
\begin{enumerate}[(i)]
\item  for all $\phi \in \mfk{X}_{C^\infty}$, we have $\pounds_{u}a, \pounds_{u}\phi, \pounds_{\phi}a\in C_T^{\alpha}(\mfk{A})$, and  there exists $(\pounds_{\xi}a)'$ such that  $(\pounds_{\xi}a,(\pounds_{\xi}a)')\in \bclD_{Z,T}(\mfk{A})$;
\item  for all $\phi\in \mfk{X}_{C^\infty}$, we have $\pounds_{u}^*\lambda, \pounds_{u}^*\phi, \pounds_{\phi}^*\lambda \in  C_T^{\alpha}(\mfk{A}^{\vee})$, and there exists $(\pounds_{\xi}^*\lambda)'$ such that  $(\pounds_{\xi}^*\lambda,(\pounds_{\xi}^*\lambda)')\in  \bclD_{Z,T}( \mfk{A}^{\vee})$.
\end{enumerate} 
\end{definition}
\begin{remark}
Let $\tilde{\mfk{X}}$,$ \tilde{\mfk{A}}$, and $\tilde{\mfk{A}}^{\vee}$ be function spaces (see Section \ref{sec:duals_of_bundles})  such that 
$
\tilde{\mfk{X}}\hookrightarrow \mfk{X},\; \tilde{\mfk{A}} \hookrightarrow \mfk{A},$ and $ \tilde{\mfk{A}}^{\vee}\hookrightarrow \mfk{A}^{\vee},
$
and
$
\pounds\in \clL(\tilde{\mfk{X}}\times \tilde{\mfk{A}},  \mfk{A})
$ and $\pounds^*\in \clL(\tilde{\mfk{X}}\times \tilde{\mfk{A}}^{\vee},  \mfk{A}^{\vee}).
$
If
$(u,\ba,\boldsymbol{\lambda})\in C_T(\tilde{\mfk{X}}) \times \bar{ \bclD}_{\bZ,T}( \tilde{\mfk{A}}) \times \bclD_{Z,T}( \tilde{\mfk{A}^{\vee}})$ and  $\xi \in \tilde{\mfk{X}}^K$, then (i) and (ii) hold.
\end{remark}

\paragraph{Clebsch variational principle.}
The Clebsch action functional $S^{\mathit{Clb}_{\bZ}} :\mathit{Clb}_{\bZ}\rightarrow \bbR$ is defined by%
\begin{equation}\label{def:Clebsch-action}
S^{\mathit{Clb}_{\bZ}}(u, \ba, \boldsymbol{\lambda})= \int_0^T \ell(u_t, a_t)\rmd t + \langle \lambda_t,   \rmd \ba_t+ \pounds_{\rmd x_t}a_t\rangle_{\mfk{A}},
\end{equation}
where 
$$
\langle \lambda_t,   \pounds_{\rmd x_t}a_t\rangle_{\mfk{A}}:=  \langle \lambda_t, \pounds_{u_t}a_t\rangle_{\mfk{A}}\rmd t +\langle \lambda_t, \pounds_{\xi} a_t  \rangle_{\mfk{A}}\rmd \bZ_t, \quad \rmd x_t:=u_t \rmd t + \xi \rmd \bZ_t.
$$

A \emph{variation} of  $(u,\ba,\boldsymbol{\lambda} )\in \mathit{Clb}_{\bZ}$ is a  curve $\{(u^{\epsilon},\ba^{\epsilon},\boldsymbol{\lambda}^{\epsilon})\}_{ \epsilon \in (-1,1)}\subset \mathit{Clb}_{\bZ}$
of the form 
$$
(u^{\epsilon},\ba^{\epsilon},\boldsymbol{\lambda}^{\epsilon})=(u + \epsilon \sym{\delta} u,\ba + \epsilon \sym{\delta} a, \boldsymbol{\lambda} + \epsilon \sym{\delta} \lambda),
$$
for arbitrarily chosen $(\sym{\delta} u, \sym{\delta} a,\sym{\delta} \lambda) \in C^{\infty}_T(\mfk{X}_{C^\infty}\times  \mfk{A}_{C^{\infty}}\times \mfk{A}^{\vee}_{C^{\infty}})$ such that $\sym{\delta} a$  vanishes at $t=0$ and $t=T$. We say $(u,\ba,\boldsymbol{\lambda} )\in \mathit{Clb}_{\bZ}$ is a critical point of the action functional $S^{\mathit{Clb}_{\bZ}}$, if for all  variations one has
$$
\frac{d}{d\epsilon}\bigg|_{\epsilon = 0}S^{\mathit{Clb}_{\bZ}}(u^{\epsilon}, \ba^{\epsilon}, \boldsymbol{\lambda}^{\epsilon})=0.
$$

By virtue of the controlled rough path calculus and, in particular, Lemmas \ref{lem:product_and_chain} and \ref{lem:fund_calc_var}, we obtain the following Clebsch variational principle.%
\footnote{For more details about the history and applications of the Clebsch variational principle in fluid dynamics, see Appendix  \ref{sec:var_smooth_paths}.}

\begin{theorem}[Clebsch variational principle on geometric rough paths]\label{thm:Clebsch} A curve $(u,\ba,\boldsymbol{\lambda} )\in \mathit{Clb}_{\bZ}$  is a critical point of  $S^{\mathit{Clb}_{\bZ}}$ in \eqref{def:Clebsch-action} iff  for all $t\in [0,T]$, the following equations hold.
\begin{equation}\label{eq:EPeq}
\begin{aligned}
m_t &+  \int_0^t\pounds_{u_s} m_s \rmd s +   \int_0^t\pounds_{\xi} m_s\rmd \bZ_s \overset{\mfk{X}^{\vee}}{=} m_0+\int_0^t\frac{\sym{\delta} \ell}{\sym{\delta} a}(u_s,a_s) \diamond a_s \rmd s, \quad m  = \frac{\sym{\delta} \ell}{\sym{\delta} u}(u,a)=\lambda \diamond a,\\
a_t &+ \int_0^t\pounds_{u_s}a_s\rmd s + \int_0^t\pounds_{\xi} a_s \rmd \bZ_s  \overset{\mfk{A}}{=} a_0, \\
\lambda_t&\overset{\mfk{A}^{\vee}}{=} \lambda_0+\int_0^t\left(\pounds_{u_t}^*\lambda_s +\frac{ \sym{\delta} \ell}{\sym{\delta} a} (u_s, a_s) \right)\rmd s + \int_0^t\pounds_{\xi}^* \lambda_s \rmd \bZ_s.
\end{aligned}
\end{equation}
\end{theorem}
\begin{proof}
See Section \ref{proof:Clebsch}.
\end{proof}
\begin{remark}
By Remark \ref{rem:int_of_control_control}, the integral $\int_0^T  \langle \lambda _t,  \rmd \ba_t\rangle_{\mfk{A}}$ in the Clebsch action functional in \eqref{def:Clebsch-action} is well-defined. Indeed, the extra structure provided by the Gubinelli derivative in the  controlled rough path space (see Appendix subsection \ref{Gubinelli-deriv}) allows one to construct this integration. 
\end{remark}
\begin{remark}\label{rem:Clebsch_contraint}
The Lagrange multiplier $\lambda$ enforces the constraint that `$a$' satisfies
$$
a_t + \int_0^t\pounds_{u_s}a_s\rmd s + \int_0^t\pounds_{\xi} a_s \rmd \bZ_s  \overset{\mfk{A}}{=} a_0, \; \; \forall t\in [0,T].
$$
That is, the quantity `$a$' is (formally)  advected by the integral curves of the vector field $\rmd x_t=u\rmd t+\xi \rmd \bZ_t$. The Lie chain rule (Theorem \ref{thm:Lie_chain}) and Hamilton-Pontryagin variational principle in Section \ref{sec:HP_var_princ} explain the nature of this differential notation (see Remarks \ref{rem:RPDE_Inverse} and \ref{rem:Ham_advection}) which we will use freely. It follows that  $\ba = (a, -\pounds_{\xi}a) \in\bclD_{Z,T}( \mfk{A})$ and $(\pounds_{\xi}a, (\pounds_{\xi}a)')\in\bclD_{Z,T}( \mfk{A})$, where  $(\pounds_{\xi}a)'=-(\pounds_{\xi_k}\pounds_{\xi_l}a)_{1\le k,l\le K}$.  For more information about  rough partial differential equations (RPDEs) and their solutions, we refer the reader to \cite{friz2014course, bailleul2017unbounded, deya2019priori, hofmanova2019navier, hofmanova2019rough}. We mention that to prove well-posedness and, in particular, to show that `$a$' is controlled, one must obtain \emph{a priori} estimates of a remainder term which  contains third-order Lie derivatives, see equation \eqref{eq:Lie_derivative_local}.
\end{remark}

\begin{remark}
Incorporating additional constraints  into the  action functional is straightforward. For example, it is possible to enforce incompressibility via  Lagrange multipliers instead of through constraints on spaces as discussed at the beginning of this section and in Section \ref{sec:Hodge}. Naturally, additional terms appear on the right-hand-side of the equation for momentum, $m$, corresponding to the pressure terms (rough and smooth in time).  We will explain  in the examples in Section \ref{sec:examples} how one can impose the incompressibility constraint, either by using Lagrangian multipliers, or by  constraining the space of vector fields and its dual.

The most commonly solved Euler equation for incompressible homogeneous flow of an ideal fluid with transport-type noise is, in addition, harmonic-free \cite{crisan2019solution, brzezniak2016existence, brzezniak2019existence, crisan2019well}. Indeed, in most papers, the authors prove well-posedness of a transport vorticity equation on the torus $\bbT^d$, $d=2,3$ with $u$ recovered via the Biot-Savart law. By the Hodge decomposition theorem (see Section \ref{sec:Hodge}), if the underlying equation for the fluid velocity $u$ does not preserve mean-freeness (i.e., harmonic-freeness), then $u$ cannot be recovered directly from the vorticity equation by the Biot-Savart law. As a result of the perturbative nature of our theory, our equations do not, in general, preserve harmonic-freeness at the level of velocity. By imposing constraints on spaces (i.e., projections), we can easily impose that $u$ is both divergence and harmonic-free and derive the corresponding momentum equation with enough `free-variables' to impose the divergence-free and harmonic-free constraints. In particular, we shall explain how the  pressure and constant harmonic terms naturally decompose into a smooth and rough part, and how they can be recovered from $u$ as was done in, for example, \cite{mikulevicius2004stochastic, mikulevicius2005global} and  \cite{hofmanova2019navier, hofmanova2019rough}).
\end{remark}

\section{The Lie chain rule for geometric rough paths and its applications}\label{sec:LieChainRule}

For an incompressible ideal fluid evolving on a compact oriented Riemannian manifold $(M,g)$ with associated volume-form $\mu_g$, the Lagrangian flow map $\eta:[0,T]\rightarrow \operatorname{Diff}_{\mu_g}$ may be regarded as a curve in the group $G:=\operatorname{Diff}_{\mu_g}$ of volume-preserving diffeomorphisms on $M$ endowed with some appropriate topology, initiated from the identity  $\eta_{0} ={\rm id}$ and parametrized by time, $t\in [0,T]$. In his seminal paper \cite{arnold1966principe}, V. I. Arnold showed that the configuration space for incompressible hydrodynamics is the space of volume preserving diffeomorphisms and that Euler's equation for the Eulerian velocity field $u$ (i.e., $\dot{\eta}_t=u_t\circ \eta_t$) is equivalent to the path $\eta$ being a critical point of the kinetic energy action functional. Otherwise said, Euler's equation can be recast as the geodesic equation on the diffeomorphism group endowed with the right-invariant weak $L^2$-metric. 

However, various geometric-analytic challenges arise if one wishes to make this viewpoint constructive and solve the geodesic equation as an ODE (and show there is no derivative loss). The crux of the matter is that composition from the right is not smooth if one wants to endow $G$ with a Banach topology and work with a standard functional analytic tool-set \cite{ebin1970groups}. The variational principles developed in this paper can be seen as extensions of the geodesic principle in \cite{arnold1966principe} or, more generally, the overarching EPDiff theory \cite{holm1998euler}. 

\subsection{Lie chain rule and  Reynolds transport theorem}
 Theorem \ref{thm:rough_diffeo}  can be extended via a coordinate chart  or approximate flow argument (see, e.g.,  \cite{weidner2018geometric, driver2018global, bailleul2019roughmanifold}) to obtain the following theorem concerning smooth rough flows on the closed manifold $M$. We assume smoothness in the  spatial variable and compactness of our manifolds for simplicity. More relaxed conditions can be found in, for example, in \cite{weidner2018geometric}.

\begin{theorem}[Rough flow properties]\label{thm:rough_diffeo_main} There exists a unique continuous map 
$$
\operatorname{Flow} : C^{\alpha}_T(\mfk{X}_{C^\infty})\times C^{\infty}_T(\mfk{X}_{C^\infty}^K)\times  \bclC_{g,T}(\bbR^K)\rightarrow C^{\alpha}_{2,T}(\operatorname{Diff}_{C^{\infty}})
$$
such that $\eta_{ts}=\operatorname{Flow}(u,\xi,\bZ)_{ts}$, $(s,t)\in [0,T]^2$, satisfies the following properties:
\begin{enumerate}[(i)]
\item for all $(s,\theta,t)\in [0,T]^3$, $\eta_{tt}=\operatorname{Id}$ and  $$\eta_{ t\theta }\circ \eta_{\theta s}=\eta_{ts};$$
\item for all $(s,t)\in \Delta$ and $f\in C^\infty$,
\begin{equation}\label{eq:pullback_flow}
\eta_{ts}^*f=f+\int_0^t \eta_{rs}^*u_r[f]\rmd r+\int_0^t \eta_{rs}^*\xi_r[f]\rmd \bZ_r,
\end{equation}
and
\begin{equation}\label{eq:push-forward_flow}
\eta_{ts*}f=f-\int_0^t u_r[\eta_{rs*}f]\rmd r-\int_0^t\xi_r[ \eta_{rs*}f]\rmd \bZ_r.
\end{equation}
\end{enumerate}
\end{theorem}

\begin{remark}
Let us recall that $\eta^*_{ts}$ and $\eta_{ts*}$ denote the pullback and push-forward, respectively (see \eqref{def:push-forward_pullback}).
Item (ii) in \eqref{eq:pullback_flow} means that for all $X\in M$, the quantity $\eta_{\cdot s}X$ is  the  unique solution of the RDE
\begin{equation}\label{eq:roughVF}
\rmd \eta_{ts}X=u_t(\eta_{ts}X)\rmd t + \xi_t(\eta_{ts}X)\rmd \bZ_t, \; \; t\in (s,T], \;\; \eta_{ss}X=X.
\end{equation}
for all $s\in [0,T]$.
\end{remark}

We refer to the following theorem as the \emph{Lie chain rule}. A stochastic version (i.e., Brownian case) of this theorem was proved in \cite{de2020implications}[Theorem 3.1].

\begin{theorem}[Rough Lie chain rule]\label{thm:Lie_chain}
For given $\tau_0 \in \clT^{lk}_{C^{\infty}}$,  $\pi \in C_T(\clT^{lk}_{C^{\infty}})$, and $\boldsymbol{\gamma}=(\gamma,\gamma') \in \bclD_{Z,T}((\clT^{lk}_{C^{\infty}})^K)$, let
$$
\tau _t = \tau_0 + \int_0^t \pi_r \rmd r + \int_0^t \gamma_r\rmd \bZ_r, \;\;t\in [0,T].
$$
Then for all  $(s,t)\in \Delta_T$, 
\begin{equation}\label{eq:pull-back_tensor}
\eta^{*}_{ts} \tau_t= \tau_s + \int_s^t \eta^{*}_{rs} \left( \pi_r + \pounds_{u_r} \tau_r\right)\rmd r  +  \int_s^t \eta^{*}_{rs} \left( \gamma_r + \pounds_{\xi_r} \tau_r\right)\rmd \bZ_r,
\end{equation}
and
\begin{equation}\label{eq:push_forward_tensor}
\eta_{ts*}\tau_t= \tau_s + \int_s^t  \left( \eta_{rs*}\pi_r - \pounds_{u_r} (\eta_{rs*}\tau_r)\right)\rmd r  +  \int_s^t \left(  \eta_{rs*} \gamma_r - \pounds_{\xi_r}( \eta_{rs*}  \tau_r)\right)\rmd \bZ_r,
\end{equation}
where the time-dependent vector fields $u$ and $\xi$ are given in equation \eqref{eq:roughVF}.
\end{theorem}
\begin{proof}
	See Section \ref{sec:proof_Lie_chain}.
\end{proof}
\begin{remark}\label{rem:RPDE_Inverse}
By \eqref{eq:push_forward_tensor}, for an arbitrary  $\tau_0\in \clT^{r,s}_{C^{\infty}}$, it follows that  $\tau_{\cdot}=\eta_{\cdot 0\,*}\tau_0$ is a classical solution  of
$$
\tau_t+ \int_0^t \pounds_{u_r}\tau_r \rmd r + \int_0^t \pounds_{\xi_r} \tau_r \textnormal{d}\bZ_r=\tau_0.
$$
Notice that if  we introduce the notation
$\rmd x_t:=\rmd \eta_{t0} \circ \eta_{t0}^{-1}:=u_t\rmd t + \xi_t \rmd \bZ_t,$
then we may write
$$
\tau_t+ \int_0^t \pounds_{\rmd x_r}\tau_r =\tau_0,
$$ 
which  generalizes the dynamic definition of the  Lie-derivative to the rough case.
\end{remark}

The following  corollary is an extension of the Reynolds transport theorem. It is an immediate  application of the definition of the integral on manifolds (see, e.g., Sec.\ 8.1 and 8.2 of \cite{abraham2012manifolds}), the global change of variables formula,  the Lie chain rule (Theorem \ref{thm:Lie_chain}) and the rough Fubini theorem (Lemma \ref{lem:Fubini}. We will use this formula next in the case $k=1$ for the proof of the Kelvin circulation theorem (see Section \ref{sec:Kelvin_circ_thm}).

\begin{corollary}[Rough Reynolds transport theorem]\label{corr:Reyn_trans} For given $\alpha_0 \in \Omega^k_{C^{\infty}}$,  $\pi \in C_T(\Omega^k_{C^{\infty}})$, and $\boldsymbol{\gamma}=(\gamma,\gamma') \in \bclD_{Z,T}((\Omega^k_{C^{\infty}})^K)$, let
$$
\alpha _t = \alpha_0 + \int_0^t \pi_r \rmd r + \int_0^t \gamma_r\rmd \bZ_r, \;\;t\in [0,T].
$$ 
Then for all $k$-dimensional smooth submanifolds $\Gamma$ embedded in $M$ and $(s,t)\in \Delta_T$, we have
\begin{equation*}
\int_{\eta_{ts}(\Gamma)}\alpha_t=\int_{\Gamma}\alpha_s + \int_s^t\int_{\eta_{rs}(\Gamma)}  \left( \pi_r + \pounds_{u_r} \tau_r\right)\rmd r  + \int_s^t\int_{ \eta_{rs}(\Gamma)}  \left( \gamma_r + \pounds_{\xi_r} \tau_r\right)\rmd \bZ_r,
\end{equation*}
where $\eta_{ts}(\Gamma)$ denotes the image of  \,$\Gamma$ under the action of the flow $\eta$.
\end{corollary}

\subsection{Kelvin's circulation theorem}\label{sec:Kelvin_circ_thm}
Assume  that for all $t\in [0,T]$,
\begin{align*}
m_t&\overset{\mfk{X}_{C^{\infty}}^{\vee}}{=}m_0+\int_0^t\left(\frac{\sym{\delta} \ell}{\sym{\delta} a}(u_s,a_s)\diamond a_s-\pounds_{u_s}m_s\right)\rmd s-\int_0^t\pounds_{\xi}m_s \rmd \bZ_s, \quad m  = \frac{\sym{\delta} \ell}{\sym{\delta} u}(u,a),\\
D_t&+\int_0^t \pounds_{u_s}D_s\rmd r+\int_0^t\pounds_{\xi}D_s \rmd \bZ_s\overset{\operatorname{Dens}_{C^{\infty}}}{=}D_0,
\end{align*}
where all the paths and integrands are assumed to be smooth. By virtue of  Theorem \ref{thm:rough_diffeo_main}, there exists a  flow of diffeomorphisms   $\eta=\eta_{\,\cdot\, 0} \in C_T^{\alpha}(\operatorname{Diff}_{C^{\infty}})$  such that 
$$
\rmd \eta_tX = u_t( \eta_tX)\rmd  t + \xi(\eta_tX)\rmd \bZ_t,\;\; t\in (0,T], \quad \eta_0X=X\in M.
$$

We obtain the following rough version of the Kelvin-Noether theorem in \cite{holm1998euler} as an application of the Reynolds transport theorem in Corollary \ref{corr:Reyn_trans}, 

\begin{theorem}[Rough Kelvin-Noether Theorem] \label{thm:Kelvin}
Let $\gamma$ denote a compact embedded one-dimensional smooth submanifold of $M$ and denote $\gamma_t=\eta_t(\gamma)$ for all $t\in [0,T]$. If $D_0$ is non-vanishing, then
$$
\oint_{\gamma_t}\frac{1}{D_t}\frac{\sym{\delta} \ell}{\sym{\delta} u}(u_t,a_t)=\oint_{\gamma_0}\frac{1}{D_0}\frac{\sym{\delta} \ell}{\sym{\delta} u}(u_0,a_0)+ \int_0^t\oint_{\gamma_s}\frac{1}{D_s}\frac{\sym{\delta} \ell}{\sym{\delta} a}(u_s,a_s)\diamond a_s \rmd s.
$$
\end{theorem}
\begin{remark}
Formula \eqref{eq:one_over_density} explains that  $\frac{1}{\mu} : \mfk{X}^{\vee}_{C^{\infty}}\rightarrow \Omega^1_{C^{\infty}}$  is defined by $m= \alpha\otimes \nu \mapsto \frac{m}{\mu}=\alpha\frac{d\nu}{d\mu}.$ 
\end{remark}
\begin{proof}
See Section \ref{proof:Kelvin}.
\end{proof}

\subsection{The Hamilton-Pontryagin variational principle for geometric rough paths}\label{sec:HP_var_princ}

In this section, in addition to the assumptions in Section \ref{sec:Clebsch_var_princ}, we require that $\mfk{A}=\mfk{A}_{C^\infty}$, $\mfk{X}=\mfk{X}_{C^\infty}$, $\xi\in \{\mfk{X}_{C^\infty}\}^K$,  and $\bZ\in \bclC_{g,T}^{\alpha}(\bbR^K)$, $\alpha\in \left(\frac{1}{3},\frac{1}{2}\right]$, is \emph{truly rough} as in Definition \ref{def:truly_rough}. We define the space of rough diffeomorphisms by $$\operatorname{Diff}_{\bZ,T,C^\infty}=\operatorname{Flow}(C^{\alpha}_T(\mfk{X}_{C^{\infty}}), C^{\infty}_T(\mfk{X}_{C^\infty}^K),\bZ)_{\,\cdot\, 0}.$$ 

For given $\eta=\operatorname{Flow}(v, \sigma,\bZ)_{\cdot 0}\in \operatorname{Diff}_{\bZ,T,C^\infty}$ and $\lambda \in \bclD_{Z,T}( \mfk{X}^{\vee})$, we let 
\begin{equation}\label{def:g_ginv_int}
\int_0^T\langle \lambda_t,\rmd \eta_t\circ \eta^{-1}_t\rangle_{\mfk{X}}:=  \int_0^T \langle \lambda_t,v_t\rangle_{\mfk{X}}\rmd t + \int_0^T \langle \lambda_t,\sigma_t\rangle_{\mfk{X}}\rmd \bZ_t.
\end{equation}

\begin{definition} Let $\mathit{HP}_{\bZ}$ denote the space of  
$$
(u,\eta,\boldsymbol{\lambda}) \in C_T^{\alpha}(\mfk{X}_{C^\infty}) \times \operatorname{Diff}_{\bZ,T,C^\infty} \times \bclD_{Z,T}( \mfk{X}^{\vee})
$$
such that for all $\phi\in \mfk{X}_{C^\infty}$, 
$\pounds_{u}\lambda, \pounds_{u}\phi, \pounds_{\phi}\lambda \in  C_T^{\alpha}(\mfk{A}^{\vee})$ , and there exists $(\pounds_{\xi}\lambda)'$ such that  $(\pounds_{\xi}\lambda,(\pounds_{\xi}\lambda)')\in  \bclD_{Z,T}( \mfk{X}^{\vee})$.
\end{definition}

For a given  $a_0\in \mfk{A}_{C^{\infty}}$,   the Hamilton-Pontryagin action integral $S_{a_0}^{\mathit{HP}_{\bZ}} :\mathit{HP}_{\bZ}\rightarrow \bbR$  is defined by
\begin{equation}\label{HP-action-def}
S^{\mathit{HP}_{\bZ}}_{a_0}(u, \eta, \boldsymbol{\lambda})= \int_0^T \ell(u_t, \eta_{t*}a_0)\rmd t+ \langle \lambda_t,\rmd \eta_t\circ \eta^{-1}_t-u_t \rmd t - \xi \rmd \bZ_t\rangle_{\mfk{X}}.
\end{equation}
\begin{remark}\label{rem:Ham_advection}
By Theorem \ref{thm:truly_rough}, the Lagrange multiplier $\lambda$ in \eqref{HP-action-def} enforces 
$$
d \eta_tX = u_t(\eta_tX)\rmd t +\xi(\eta_tX)\rmd \bZ_t, \;\; t\in (0,T],\quad \eta_0X=X\in M.
$$
The \emph{true roughness} of the path $\bZ$ defined in Definition \ref{def:truly_rough} and satisfying Theorem \ref{thm:truly_rough} is required to ensure that \eqref{def:g_ginv_int} is well-specified  and to conclude that $v\equiv u $ and $\sigma \equiv \xi$ in the proof of Theorem \ref{thm:HamPont} (i.e., after taking variations). In contrast, we did not impose true roughness (see  Remark \ref{rem:Clebsch_contraint}) of the path for the Clebsch variational principle in Theorem \ref{thm:Clebsch} owing to the nature of the constraint and Lemma \ref{lem:fund_calc_var}.

By the Lie chain rule (Theorem \ref{thm:Lie_chain}), we find that $a_t=\eta_{t*}a_0$ satisfies
$$
a_t+\int_0^t \pounds_{\rmd x_s}a_s=a_0, \quad \textnormal{ where } \rmd x_t= u_t \rmd t + \xi \rmd \bZ_t,
$$
where the notation for $\rmd x_t$ is explained in Remark \ref{rem:RPDE_Inverse}.
That is, the quantity  $a$ is advected by the  flow $\eta \in \operatorname{Diff}_{\bZ,T,C^\infty}$.
This advection equation is used directly as the constraint in the Clebsch variational principal in Theorem \ref{thm:Clebsch}.
\end{remark}

\begin{definition}
A variation of  $(u,\eta,\boldsymbol{\lambda} )\in \mathit{HP}_{\bZ}$ is a  curve $\{(u^{\epsilon},\eta^{\epsilon},\boldsymbol{\lambda}^{\epsilon})\}_{ \epsilon \in (-1,1)}\subset  \mathit{HP}_{\bZ}$ of the form 
$$
(u^{\epsilon},\eta^{\epsilon},\boldsymbol{\lambda}^{\epsilon})=(u + \epsilon \sym{\delta} u,\psi^{\epsilon}\circ \eta,  \boldsymbol{\lambda} + \epsilon \sym{\delta} \lambda),
$$
where $\psi \in C^{\infty}([-1,1]\times [0,T];  \operatorname{Diff}_{C^{\infty}})$  is defined to be the flow (in the $t$-variable) given by
$$
\partial_t\psi^{\epsilon}_tX=\epsilon\partial_t\sym{\delta}w_t(\psi^{\epsilon}_tX), \quad \psi^{\epsilon}_0X=X\in M,
$$
for arbitrarily chosen $(\sym{\delta} u, \sym{\delta} w, \sym{\delta} \lambda) \in C^{\infty}_T(\mfk{X}_{C^\infty}\times  \mfk{A}_{C^{\infty}}\times \mfk{A}^{\vee}_{C^{\infty}})$ such that $\sym{\delta} w$ vanishes at $t=0$ and $t=T$.
\end{definition}
\begin{remark}[Variation $\eta^{\epsilon}$]
The type of variation we use for the rough diffeomorphism is common in the geometric mechanics community (see, e.g., Lemma 3.1 of \cite{arnaudon2014stochastic}). 
Notice that for all $t\in [0,T]$ and $f\in C^{\infty}$, 
$$
\psi^{\epsilon*}_t f  =f +\epsilon \int_0^t\psi^{\epsilon*}_r[\partial_t\sym{\delta}w_r f]\rmd r.
$$
Applying Theorem \ref{thm:Lie_chain} and using the natural property of the Lie derivative leads to
$$
\eta_t^{\epsilon*}f=f+\int_0^t\eta_r^{\epsilon*}\left(\pounds_{v_r^{\epsilon}}f+ \epsilon\pounds_{\partial_t\sym{\delta}w_r}f\right)\rmd r+ \int_0^t \eta_r^{\epsilon*}\pounds_{\sigma^{\epsilon}_r}f\rmd \bZ_r,
$$
where $v^{\epsilon}_t=\psi^{\epsilon}_{t*}v$ and $\sigma^{\epsilon}_t= \psi^{\epsilon}_{t*}\sigma$. Thus, for a given $\eta=\operatorname{Flow}(v, \sigma,\bZ)_{\cdot 0}$, it follows that 
$$
d \eta_t^{\epsilon}X = \left(v_t^{\epsilon}(\eta_t^{\epsilon}X)+\epsilon\partial_t\sym{\delta}w_r(\eta_t^{\epsilon}X)\right)\rmd t + \sigma^{\epsilon}_t(\eta_t^{\epsilon}X)\rmd \bZ_t, \quad \eta_0^{\epsilon}X=X\in M,
$$
and hence $$\eta^{\epsilon}=\operatorname{Flow}\left(v^{\epsilon}+\epsilon\partial_t\sym{\delta}w, \sigma^{\epsilon},\bZ\right)_{\cdot 0}\in \operatorname{Diff}_{\bZ,T,C^\infty}.$$
\end{remark}

The  proof of the following theorem is given in Section \ref{proof:Ham_Pont}.

\begin{theorem}[Hamilton-Pontryagin variational principle]\label{thm:HamPont} A curve $(u,\eta,\boldsymbol{\lambda} )\in \mathit{HP}_{\bZ}$  is a critical point of  $S^{\mathit{HP}_{\bZ}}$ if and only if for all $[0,T]$, 
\begin{equation*}
\begin{aligned}
m_t &+  \int_0^t\pounds_{u_s} m_s \rmd s +   \int_0^t\pounds_{\xi} m_s\rmd \bZ_s \overset{\mfk{X}^{\vee}}{=} m_0+\int_0^t\frac{\sym{\delta} \ell}{\sym{\delta} a}(u_s,a_s) \diamond a_s \rmd s, \quad m  = \frac{\sym{\delta} \ell}{\sym{\delta} u}(u,a)=\lambda,\\
a_t &+ \int_0^t\pounds_{u_s}a_s\rmd s + \int_0^t\pounds_{\xi} a_s \rmd \bZ_s  \overset{\mfk{A}_{C^\infty}}{=} a_0, \quad a_t=\eta_{t*}a_0,\\
\rmd \eta_tX&=u_t(\eta_tX)\rmd t + \xi(\eta_tX)\rmd \bZ_t,\;\;t\in (0,T], \;\; \eta_0X=X\in M.
\end{aligned}
\end{equation*}
\end{theorem}
\begin{remark}
The corresponding Hamilton-Pontryagin principle was derived for  SALT  in \cite{gay2018stochastic}.
\end{remark}
\begin{proof}
See Section \ref{proof:Ham_Pont}.
\end{proof}
\begin{remark}[Incompressible homogeneous Euler]
The rough incompressible homogeneous (unit density) Euler equations arise from  the choice of the `kinetic energy' Lagrangian $\ell: \dot{\mfk{X}}_{\mu_g}\rightarrow \bbR_+$ defined by
$$
\ell(u)=\int_{M} g(u,u) \mu_g,
$$
where $(M,g)$ is an oriented Riemannian manifold with corresponding  volume form $\mu_g$. We refer to Sections \ref{sec:Euler-spaces} and \ref{sec:Euler-spaces} for more details.
Letting $\dot{\mfk{X}}^{\vee}_{\mu_g}$ denote  the space of one-form densities modulo exact and harmonic forms (see Definition \ref{def:canonical_dual_incompressible}), we find
$$
m= \lambda=\frac{\sym{\delta}\ell}{\sym{\delta} u}=[u^{\flat}\otimes \mu_g]\in \dot{\mfk{X}}^{\vee}_{\mu_g},
$$
and hence that $(u,\eta,\sym{\lambda})$ is a critical point  of  $\mathit{HP}_{\bZ}$  iff 
\begin{equation*}
\begin{aligned}
\rmd [u^{\flat}_t \otimes \mu_g]&+  \pounds_{u_t}[  u^{\flat}_t\otimes \mu_g]\rmd t +   \pounds_{\xi}  [u^{\flat}_t\otimes \mu_g]\rmd \bZ_t \overset{\dot{\mfk{X}}^{\vee}_{\mu_g}}{=}\frac{1}{2}[\bd g(u_t,u_t)\otimes \mu_g]\rmd t,
\\
\rmd \eta_tX&=u_t(\eta_tX)\rmd t + \xi(\eta_tX)\rmd \bZ_t,\;\;t\in (0,T], \;\; \eta_0X=X\in M.
\end{aligned}
\end{equation*}
The first equation is equivalent to 
$$
\begin{cases}
\rmd u^{\flat}_t  + \pounds_{u_t} u^{\flat}_t \rmd t \pounds_{\xi}u^{\flat}_t \rmd \bZ_t = \frac{1}{2}\bd g(u_t,u_t)-\bd \rmd \bp_t - \rmd \bc_t ,\\
\bd^* u^{\flat}=\operatorname{div} u=0 ,\\
H(u^{\flat})=0 ,
\end{cases}
$$
where  $\rmd \bp_t = p_t \rmd t +q_t \rmd \bZ_t$ and    $\rmd \bc_t = c_t \rmd t + \tilde{c}_t \rmd \bZ_t$  are the Lagrangian multipliers corresponding to the divergence and harmonic-free constraints. It follows that $\omega = \bd u^{\flat}$ satisfies 
$$
\rmd \omega _t  + \pounds_{u_t} \omega_t \rmd t + \pounds_{\xi}\omega_t \rmd \bZ_t =0.
$$
In \cite{crisan2021solution}, we  extend the work of  \cite{hofmanova2019rough}, which studied the viscous case on the torus $M=\bbT^d$, to show that for an  initial-velocity $u_0\in \mfk{X}_{W^{m}_2}$ with $m>\frac{d}{2}+1$, there exists a unique  maximal Cauchy development $u\in C([0,T^*);\mfk{X}_{W^{m}_2})\cap C^{\alpha}([0,T_{\textnormal{max}};\mfk{X}_{W^{m-3}_2})$ . Moreover, we show that if $T_{\textnormal{max}}<\infty$, then 
$$
\int_0^{T^*}|\omega_t|_{\Omega^2_{L^\infty}}\mu_g=+\infty;
$$
that is a Beale-Kato-Majda blowup criterion holds. In dimension two, identifying $\omega$ with a scalar $\tilde{\omega}=\star \omega\in \Omega^0$, we find that  $|\tilde{\omega}_t|_{L^p}=|\tilde{\omega}_0|_{L^p}$ for all $p$ so that $T^*=+\infty$.   

Therefore, taking the initial data $u_0\in \mfk{X}_{C^\infty}$ to be smooth, we obtain  a  solution $u\in C^{\alpha}_T(\mfk{X}_{C^{\infty}})$ on any interval $[0,T]$ with $T<T_{\textnormal{max}}$, and hence we may  construct the  flow $\eta=\operatorname{Flow}(u,\xi,\bZ)\in \operatorname{Diff}_{\bZ,T,C^\infty}$. Consequently, we obtain a critical point $(u,\eta, \sym{\lambda})$ of  $\mathit{HP}_{\bZ}$ for any $T<T_{\textnormal{max}}$ with $\sym{\lambda}=(\lambda, \lambda')=([u^{\flat}\otimes \mu_g], \pounds_{\xi}[u^{\flat}\otimes \mu_g]).$
\end{remark}

\subsection{An Euler--Poincar\'e variational principle for geometric rough paths}\label{sec:EP_var_princ}
In this  section, we assume all stated quantities exist and are smooth, and  work formally (see Remark \ref{rem:EulerPoincare}). 

\begin{theorem}[Euler--Poincar\'e variational principle] \label{EP-thm}
Consider a path $\eta=\operatorname{Flow}(u,\xi,\bZ) \in \operatorname{Diff}_{\bZ,T,C^\infty}$ .  The following are equivalent:
\begin{enumerate}[(i)]
\item
The constrained variational principle
$$
\delta \int _{0} ^{T}  \ell(u_t ,a_t) dt = 0
$$
holds on $C_T^{\alpha}(\mfk{X}_{C^\infty}) \times C_T^{\alpha}(\mfk{A}_{C^\infty})$ using variations of the form
\begin{equation} \label{epvariations}
\sym{\delta} u \rmd t= \partial_t \sym{\delta}w\rmd t - \operatorname{ad}_{\rmd x_t} \sym{\delta}\sym{w} \qquad \textnormal{and} \quad
\sym{\delta} a =  -\,\pounds_{\sym{\delta}\sym{w}} a ,
\end{equation}
for arbitrarily chosen  $\sym{\delta}w \in C_T^{\infty}(\mathfrak{X}_{C^\infty})$ which vanishes at $t=0$ and $t=T$, where $\rmd x_t = u_t \rmd t + \xi \rmd \bZ_t.$ 
\item 
The \textbf{ Euler--Poincar\'{e}} equations on geometric rough paths hold: that is, for all $t\in [0,T]$,
\begin{align*}
m_t &+  \int_0^t\pounds_{u_s} m_s \rmd s +   \int_0^t\pounds_{\xi} m_s\rmd \bZ_s \overset{\mfk{X}^{\vee}_{C^\infty}}{=} m_0+\int_0^t\frac{\sym{\delta} \ell}{\sym{\delta} a}(u_s,a_s) \diamond a_s \rmd s, \quad m  = \frac{\sym{\delta} \ell}{\sym{\delta} u}(u,a),\\
a_t &+ \int_0^t\pounds_{u_s}a_s\rmd s + \int_0^t\pounds_{\xi} a_s \rmd \bZ_s  \overset{\mfk{A}_{C^\infty}}{=} a_0, \quad a_t=\eta_{t*}a_0.
\end{align*}
\end{enumerate}
\end{theorem}
\begin{proof}
See Section \ref{proof:EP-thm}.
\end{proof}
\begin{remark}
Recalling that for all $u\in \mfk{X}^\infty$ the adjoint of $\operatorname{ad}_{u}=-\pounds_{u}: \mfk{X}_{C^\infty}\rightarrow \mfk{X}_{C^\infty}$ is $\operatorname{ad}^*_u=\pounds_{u}: \mfk{X}_{C^\infty}^{\vee}\rightarrow \mfk{X}_{C^\infty}^{\vee}$, we have \begin{equation}\label{eulerpoincare}
\begin{aligned}
\rmd m_t
&+ \operatorname{ad}_{{\rm d}x_t}^{\ast} m_t
= \frac{\sym{\delta} \ell}{\sym{\delta} a_t} \diamond a_t\\
\rmd a_t &+ \pounds_{\rmd x_t}a_t=0.
\end{aligned}
\end{equation}
\end{remark}
\begin{remark}\label{rem:EulerPoincare}
This is not strictly a variational principle in the same sense as the standard Hamilton's principle.
It is akin to the classic La\-gran\-ge d'Al\-em\-bert principle for dynamics with nonholonomic constraints, because the variationsof $\sym{\delta} u$ and $\sym{\delta} a$ in \eqref{epvariations} are restricted in terms of $\sym{\delta} w$. These restrictions are discussed next.

Let $\eta = \operatorname{Flow}(u,\xi, \bZ)_{\cdot, 0}$.  Assume that for all $\sym{\delta} w\in \mfk{X}_{C^\infty}$ and $\sym{\delta}w_0=\sym{\delta}w_T=0$ we can construct  a variation $\{\eta^{\epsilon}\}_{\epsilon\in [-1,1]}$ such that 
$$
\rmd \eta^{\epsilon}_tX = u^{\epsilon}_t(\eta_t^{\epsilon} X) \rmd t +  \xi(\eta^{\epsilon}_t X)   \rmd \bZ_t,\; \; t\in (0,T], \quad \eta_0^{\epsilon}X=X\in M, 
$$
and for all $t\in [0,T]$,
\begin{enumerate}[(i)]
\item $$
\frac{\partial}{\partial \epsilon}\frac{\partial}{\partial t }  \eta_t^{\epsilon}\big|_{\epsilon=0}=\frac{\partial}{\partial t } \frac{\partial}{\partial \epsilon} \eta_t^{\epsilon}\big|_{\epsilon=0};
$$ 
\item $$
 \sym{\delta}w_t=\left(\frac{\partial } {\partial \epsilon}\big|_{\epsilon=0}\eta_t^{\epsilon}\right) \circ \eta_t^{-1} \quad \Leftrightarrow \quad \frac{\partial } {\partial \epsilon}\big|_{\epsilon=0}\eta_t^{\epsilon}X= \sym{\delta}w_t(\eta_tX).
 $$
\end{enumerate}
Define
$
\sym{\delta} u_t:=\frac{\partial}{\partial \epsilon}\big|_{\epsilon=0} u^{\epsilon}_t .
$
Then 
$$
\rmd  \frac{\partial}{\partial \epsilon} \eta_t^{\epsilon}\big|_{\epsilon=0}=\rmd  (\sym{\delta}w_t\circ \eta_t)= (\partial_t\sym{\delta}w_t)\circ \eta_t \rmd t +  (T\sym{\delta}w_t\circ  \eta_t) (u_t\circ \eta_t \rmd t + \xi \circ \eta_t \rmd \bZ_t)
$$
and
$$
\frac{\partial}{\partial \epsilon}\rmd  \eta_t^{\epsilon}\big|_{\epsilon=0}=\frac{\partial}{\partial \epsilon} (u^{\epsilon}\circ \eta^{\epsilon}_t \rmd t+ \xi \circ \eta^{\epsilon}_t\rmd \bZ_t)\big|_{\epsilon=0}=\left(\sym{\delta} u_t \circ \eta_t +  (Tu_t \circ \eta_t) (\sym{\delta} w_t\circ \eta_t)\right)\rmd t+ (T\xi \circ \eta_t) (\sym{\delta} w_t\circ \eta_t) \rmd \bZ_t
$$
Using the equality of mixed derivatives, we find
\begin{align*}
\sym{\delta} u_t \circ \eta_t \rmd t &=\left((\partial_t\sym{\delta}w_t)\circ \eta_t + [\sym{\delta}w_t, u_t]\circ \eta_t \right)\rmd t + [\sym{\delta}w_t, \xi]\circ \eta_t\bZ_t\\ &=\left((\partial_t\sym{\delta}w_t)\circ \eta_t -\operatorname{ad}_{u_t} \sym{\delta}w_t\circ \eta_t \right)\rmd t -\operatorname{ad}_{\xi} \sym{\delta}w_t\circ \eta_t\bZ_t.
\end{align*}
It follows from $a^{\epsilon}_t= \eta^{\epsilon}_{t*}a_0$ that $\sym{\delta} a_t= \frac{\partial}{\partial \epsilon}\big|_{\epsilon=0}\eta^{\epsilon}_{t*}a_0 =-\pounds_{\sym{\delta} w}a_t.$ 
Two issues now arise: i) we do  not have a proof that such variations  exist as we did for the Hamilton-Pontryagin variational principle ii) it is not clear how to deduce 
$$
\sym{\delta} u_t \rmd t =\left(\partial_t\sym{\delta}w_t -\operatorname{ad}_{u_t} \sym{\delta}w_t \right)\rmd t -\operatorname{ad}_{\xi} \sym{\delta}w_t\rmd \bZ_t.
$$
We shall leave the clarification of the Euler-Poincar\'e variations as an open problem. 
\end{remark}

\subsection{A Lie--Poisson bracket for Hamiltonian dynamics on geometric rough paths}\label{sec:LP_bracket}

\begin{definition} 
We define $h: \bclD_{Z,T}(\mfk{X}^{\vee})\oplus \bclD_{Z,T}(\mfk{A})\rightarrow \bclD_{Z,T}(\bbR)$ by 
$$
h_t(m,a) :=\int_0^t \left(\langle m_s, u_s\rangle_{\mfk{X}}- \ell(u_s,a_s) \right)\rmd s  + \int_0^t \langle m_s, \xi\rangle_{\mfk{X}} \rmd \bZ_s, \quad (m,a)\in \bclD_{Z,T}(\mfk{X}^{\vee})\oplus \bclD_{Z,T}(\mfk{A}), \;\; t\in [0,T],
$$
where  $u$ denotes the inverse of  $\frac{\sym{\delta} \ell}{\delta u}(\cdot, a): \mfk{X}\rightarrow \mfk{X}^{\vee}$ applied to $m$; that is, $m=\frac{\sym{\delta} \ell}{\delta u}(u, a)$. 
\end{definition}
The rough Hamiltonian $h_t(m,a)$ is the sum of the deterministic Hamiltonian defined by Legendre transformation associated with the Lagrangian $\ell$ and given by
$
H(m,a)= \langle m, u\rangle_{\mfk{X}}- \ell(u,a)
$,
plus $G(m)=\langle m, \xi_k\rangle_{\mfk{X}}$, so that 
$$
h_t(m,a) =\int_0^{t} H(m_s,a_s)\rmd s  + \int_0^{t} G(m_s) \rmd \bZ_s.
$$

Let us take variations of $h(m,a)$ in $m$ and $a$.
For arbitrary $\sym{\delta}m\in \mfk{X}^{\vee}_{C^\infty}$ and $\sym{\delta}a\in \mfk{A}_{C^\infty}$, we find
\begin{align*}
\sym{\delta} h_t &= \int_0^t \left(\langle m_s, \frac{\sym{\delta}u}{\sym{\delta} m}\rangle_{\mfk{X}}+\langle \sym{\delta }m_s, u_s\rangle_{\mfk{X}}- \langle \frac{\sym{\delta}\ell}{\sym{\delta} u}(u_s,a_s), \frac{\sym{\delta}u}{\sym{\delta} m}\rangle_{\mfk{X}}-\langle \frac{\sym{\delta}\ell}{\sym{\delta} a}(u_s,a_s),\sym{\delta}a\rangle_{\mfk{A}} \right)\rmd s + \int_0^t \langle \sym{\delta }m_s, \xi\rangle_{\mfk{X}} \rmd \bZ_s\\
&=\int_0^t \langle \sym{\delta }m_s, \rmd x_s\rangle_{\mfk{X}} - \int_0^t\langle \frac{\sym{\delta}\ell}{\sym{\delta} a}(u_s,a_s),\sym{\delta}a\rangle_{\mfk{A}}\rmd s,\quad \forall t\in [0,T],
\end{align*}
where in the second equality we have used  $m:=\frac{\sym{\delta} \ell}{\delta u}$ and set 
$
\rmd x_t := u_t \rmd t + \xi \rmd \bZ_t.
$
Thus,
$$
\rmd \frac{\sym{\delta} h_t}{\sym{\delta} m}(m,a)= \rmd x_t
\quad \textnormal{and} \quad 
\rmd \frac{\sym{\delta} h}{\sym{\delta} a}(m,a)=-\, \frac{\sym{\delta}\ell}{\sym{\delta} a}(u,a)\rmd t
,$$
which is to say
$$
\frac{\sym{\delta} h_t}{\sym{\delta} m}(m,a)= \int_0^tu_s \rmd s +\int_0^t \xi \rmd \bZ_s=\int_0^t \frac{\sym{\delta} H}{\sym{\delta} m}(m_s,a_s)\rmd s +\int_0^t \frac{\sym{\delta} G}{\sym{\delta} m}(m_s)\rmd s
$$
$$
\frac{\sym{\delta} h_t}{\sym{\delta} a}(m,a)=-\, \frac{\sym{\delta}\ell}{\sym{\delta} a}(u_t,a_t)=\frac{\sym{\delta}H}{\sym{\delta} a}(m_t,a_t).
$$

\begin{corollary}[Lie--Poisson Hamiltonian form]\label{cor-LPHamForm}
The Euler--Poincar\'e equations in \eqref{eulerpoincare} can be written 
in Lie--Poisson bracket form as
$$
\begin{bmatrix}
 m_t \\  a_t
\end{bmatrix}
\overset{\mfk{X}^{\vee}\oplus \mfk{A}}{=}\begin{bmatrix}
m_0 \\  a_0
\end{bmatrix}
- 
\int_0^t\begin{bmatrix}
\operatorname{ad}_{\Box}^{\ast}  m_s  & \Box \diamond  a_s
\\
\pounds_{\Box} a_s & 0
\end{bmatrix}
\begin{bmatrix}
{\rmd \sym{\delta} h_s}/{\sym{\delta} m}(m,a)  \\  \rmd {\sym{\delta} h_s }/{\sym{\delta} a}(m,a) 
\end{bmatrix},
$$
in which the boxes ($\Box$) represent substitution in the  operator.
For  arbitrary $f: \mfk{X}^{\vee}\times \mfk{A}\rightarrow \bbR$ such that $\frac{\sym{\delta} f}{\sym{\delta} m}(m,a)\in C_T(\mfk{X})$ and $\frac{\sym{\delta} f}{\sym{\delta} a}(m,a)\in \bclD_{Z,T}(\mfk{A}^{\vee})$ exist (see Lemma \ref{lem:product_and_chain}), we have
\begin{align*} 
  f(m_t,a_t)
&
=f(m_0,a_0)+\int_0^t\langle \rmd m_s, \frac{\sym{\delta} f}{\sym{\delta} m}(m_s,a_s)\rangle_{\mfk{X}}+ \int_0^t\langle  \rmd a_s,\frac{\sym{\delta} f}{\sym{\delta} a}(m_s,a_s)\rangle_{\mfk{A}^{\vee}}\\
&
=
- 
\int_0^t\left\langle \begin{bmatrix}
\operatorname{ad}_{\Box}^{\ast}  m_s  & \Box \diamond  a_s
\\
\pounds_{\Box} a_s & 0
\end{bmatrix}
\begin{bmatrix}
{\rmd \delta h}/{\sym{\delta} m}(m_s,a_s) \\  \rmd {\delta  h}/{\sym{\delta} a(m_s,a_s)} 
\end{bmatrix}
, \begin{bmatrix}
{\sym{\delta} f}/{\sym{\delta} m(m_s,a_s)} \\ {\sym{\delta} f}/{\sym{\delta} a(m_s,a_s)}
\end{bmatrix}\right\rangle_{\mfk{X}\oplus \mfk{A}^{\vee}}\\
&=:\int_0^t\Big\{f,\rmd h_s\Big\}(m_s,a_s),
\end{align*}
in which the last equality adopts the notation for the semidirect-product Lie--Poisson bracket given in \cite{holm1998euler}. In differential notation, we find
$$
\rmd f(m_t,a_t) = \{f, \rmd h_t\}(m_t,a_t)=\{f,H\}(m_t,a_t)\rmd t + \{f,G\}(m_t,a_t)\rmd \bZ_t.
$$
\end{corollary}
\begin{remark}
Stochastic Hamilton equations were introduced along parallel lines with the deterministic canonical theory in \cite{bismut1982mecanique}. These results were later extended to include reduction by symmetry in \cite{lazaro2007stochastic}. Reduction by symmetry of expected-value stochastic variational principles for Euler–Poincaré equations was developed in \cite{arnaudon2014stochastic} and \cite{chen2015constrained}. Stochastic variational principles were also used in constructing stochastic variational integrators in Bou-Rabee and Owhadi \cite{bou2009stochastic}. 

\end{remark}

\section{Examples}\label{sec:examples}
\subsection{Rough incompressible Euler equation via Lagrange multipliers}\label{sec:Euler-Lag_mult}

Let $(M,g)$ denote a smooth, compact, connected, oriented  $d$-dimensional Riemannian manifold without boundary.  Denote by $\mu_g\in \operatorname{Dens}_{C^{\infty}}$  the associated volume form,  which is given in local coordinates by 
$$
\mu_g=\sqrt{|\operatorname{det}[g_{ij}]|}\; dx^1\wedge \cdots\wedge  dx^d.
$$
Let $A=\Lambda^dT^*M\oplus \Lambda^0T^*M$ and $A^{\vee}= \Lambda^0T^*M\oplus \Lambda^dT^*M$.  Denote the advected variables by $\ba = (\bD, \sym{\rho})\in \mfk{A}=\operatorname{Dens}_{\clF_3}\oplus \Omega^0_{\clF_3}$ and the associated Lagrangian multipliers by $\sym{\lambda}=(\sym{f},\sym{\beta})\in\mfk{A}^{\vee}= \Omega^0_{\clF_4}\oplus \operatorname{Dens}_{\clF_4}$. In this example, we will explain how to impose incompressibility  through Lagrangian constraints. In the following example, we will explain how to impose incompressibility through projections and spatial constraints. Towards this end, we  introduce an additional Lagrangian multiplier $\sym{\pi} \in \bclD_{Z,T}(\clF_3\cap \clF_4)$ to enforce incompressibility.
We consider the Clebsch action functional
$$
S^{\mathit{Clb}_{\bZ}}(u, \ba, \sym{\lambda}, \sym{\pi})
= \int_0^T \ell(u_t, a_t)\rmd t  
+\scp{\rmd \sym{\pi}_t}{D_t- \rho_t\mu_g}_{\Omega^d}+\scp{f_t}{\rmd D_t + \pounds_{dx_t}D_t}_{\Omega^d}  
+ \scp{\beta_t}{\rmd \rho_t + \pounds_{dx_t}\rho_t }_{\Omega^0},
$$
where  $\rmd x_t = u_t \rmd t + \xi \rmd \bZ_t$ and the Lagrangian $\ell: \mfk{X}\times \mfk{A}\rightarrow \bbR$ is defined by
$$
\ell(u,a)=\frac{1}{2}\int_{M} g(u,u)D=\frac{1}{2}\langle u^{\flat}\otimes D,u\rangle_{\mfk{X}},
$$
where the $\flat$ operation is defined in Section \ref{sec:Hodge}.
We take variations of $(u, \ba, \sym{\lambda})$ as defined  in Section \ref{sec:Clebsch_var_princ}.  A variation of $\sym{\pi}$  is defined to be $\sym{\pi}^{\epsilon}=\sym{\pi}+\epsilon \sym{\delta} \pi$ for $\sym{\delta} \pi \in C^{\infty}_T(C^\infty)$ such that $\sym{\delta}\pi_0=\sym{\delta}\pi_T=0$.

It follows that all $(u,a)\in \mfk{X}\times \mfk{A}$ ,
$$
m=\frac{\sym{\delta} \ell}{\sym{\delta} u}(u,a )= u^{\flat}\otimes D\in \mfk{X}^{\vee}, \qquad u=\frac{\sharp}{D}m,
\qquad \textnormal{and} \quad \frac{\sym{\delta} \ell}{\delta a}(u,a)=\left(\frac{1}{2}g(u,u),\;\;0\right)\in \mfk{A}^{\vee},
$$
where the diffeomorphism $\frac{\sharp}{D}$ is defined in Section \ref{sec:duals_of_bundles}.
Let us now compute the relevant diamond terms (see Definition \ref{diamond-def}). For all $(h, \nu) \in C^{\infty}\times  \operatorname{Dens}_{C^\infty}$ and $u\in \mfk{X}_{C^{\infty}}$, we have 
$$
-\langle h, \pounds_{u}\nu \rangle_{\Omega^d}=-\int_{M} h \pounds_{u}\nu=-\int_{M} h \bd \bi_{u}\nu=\int_{M}\nu\bi_{u}\bd h =\langle \pounds_{u}h,\nu\rangle_{\Omega^d}=\langle \bd h \otimes \nu,u\rangle_{\mfk{X}},
$$
which implies $h\diamond \nu =\bd g \otimes \nu$. Moreover, since
$$
-\langle \nu, \pounds_{u}h\rangle_{\Omega^0}=-\int_{M}\nu i_{u}\bd h=-\langle \bd h \otimes \nu , u\rangle_{\mfk{X}} \quad \Rightarrow \quad \nu \diamond h = -\bd h \otimes \nu.
$$

With minor modifications of the Proof of Theorem \ref{thm:Clebsch}, we find that $(u, \ba, \sym{\lambda}, \sym{\pi})$ is a critical point of  $S^{\mathit{Clb}_{\bZ}} = 0$, if and only if for all $t\in [0,T]$:
\begin{align*}
 m_t+\int_0^t\pounds_{dx_s} m_s \overset{\mfk{X}^{\vee}}{=}m_0+\int_0^t\frac{1}{2}g(u_s,u_s)\diamond D_s\rmd s &+\int_0^t \rmd  \sym{\pi}_s\diamond D_s - \int_0^t\rmd  \sym{\pi}_s \mu_g\diamond \rho_s,\;\; m_t=f_t\diamond D_t + \beta_t \diamond \rho_t, \\
D_t+\int_0^t \pounds_{dx_s}D_s\overset{\operatorname{Dens}}{=}D_0,  \quad & \quad 
 \rho_t +\int_0^t \pounds_{dx_s}\rho_s\overset{\Omega^0}{=}\rho_0, \quad D_t=\rho_t\mu_g,\\
 f_t+\int_0^t \pounds_{dx_s}f_s \overset{\Omega^0}{=}  f_0 
 + \int_0^t \sym{\pi}_s \rmd s + \frac{1}{2}\int_0^tg(u_s,u_s) \rmd s \quad & \quad  \beta_t+\int_0^t \pounds_{dx_s}\beta_t \overset{\operatorname{Dens}}
 {=}\beta_0 - \int_0^t \sym{\pi}_s \mu_g \rmd s .
\end{align*}
Substituting   $ D=\rho\mu_g$ into the equation for $D$ and applying the diffeomorphism $\frac{1}{\mu_g}$ (see \ref{eq:one_over_density}), we find
$$
\rho_t  + \int_0^t \left(\pounds_{dx_s} \rho_s + \operatorname{div}_{\mu_g}dx_s\right)=\rho_0.
$$
Since $\rho$ is advected,  we obtain for all $t\in [0,T]$, 
$$
\int_0^t \operatorname{div}_{\mu_g}u_s \rmd s + \int_0^t \operatorname{div}_{\mu_g}\xi \rmd \bZ_s =0.
$$
In order to conclude that $\operatorname{div}_{\mu_g}u\equiv 0$, we need to assume either $ \operatorname{div}_{\mu_g}\xi\equiv 0$, or  true roughness of $\bZ$. In the next example, where $u$ is  constrained to be divergence-free,  we do not require additional assumptions to conclude that $u$ is divergence-free since we will impose this directly. 

Substituting $m=u^{\flat}\otimes D$ into the momentum equation and recalling that $D$ is advected,  that the  Lie derivative is a derivation, and that the product rule (Lemma \ref{lem:product_and_chain}) holds, we find
$$
\rmd u_t^{\flat}\otimes  \rho_t\mu_g+\pounds_{dx_t} u^{\flat}_t\otimes  \rho_t\mu_g+= \frac{1}{2}\bd g(u_t,u_t)\otimes   \rho_t\mu_g+ \bd \rmd \sym{\pi}_t\otimes   \rho_t\mu_g+  \bd \rho_t\otimes   \rmd \sym{\pi}_t \mu_g.
$$
Assuming $\rho$ is non-vanishing and applying the diffeomorphism $\frac{1}{D}$ (see \ref{eq:one_over_density}) yields 
\begin{equation}\label{eq:euler_lagrange_mult}
\rmd u^{\flat}_t+\pounds_{dx_t} u^{\flat}_t\overset{\Omega^1}{=}\frac{1}{2}\bd g(u_t,u_t) + \bd \rmd \sym{\pi}_t +\frac{1}{\rho_t} \bd \rho_t   \rmd \sym{\pi}_t= \frac{1}{2}\bd g(u_t,u_t)  + \frac{1}{\rho_t} \bd (\rho_t \rmd \sym{\pi}_t )= \frac{1}{2}\bd g(u_t,u_t) -\frac{1}{\rho_t} \bd \rmd \bp_t,
\end{equation}
in which the pressure is identified in terms of the Lagrange multiplier $\rmd\sym{\pi}_t$ as $\rmd\bp_t:=-\rho_t\rmd\sym{\pi}_t$. We will elaborate more on this equation in the following example.

\subsection{Rough incompressible Euler equation via constraint on spaces}
\label{sec:Euler-spaces}
Let $(M,g)$  and $\mu_g$ be as in the previous example.  Let  $A=\Lambda^dT^*M$ and $A^{\vee}=\Lambda^0T^*M$ in this example, and notice that for all $D\in \mfk{A}=\operatorname{Dens}_{\clF^3}$, there exists $\rho \in \Omega^0_{\clF^3}$ such that $D=\rho \mu_g$. 

 Let $\mfk{X}_{\mu_g}=\mfk{X}_{\mu_g,\clF_1}$ denote the space of incompressible vector fields and $\mfk{X}_{\mu_g}^{\vee}=\mfk{X}_{\mu_g,\clF_2}$ denote the dual space of one-form densities modulo the kernel of the divergence-free projection as defined in Definition \ref{def:canonical_dual_incompressible}. Denote by $\langle \cdot,\cdot \rangle_{\mfk{X}_{\mu_g}}:\mfk{X}_{\mu_g}^{\vee}\times \mfk{X}_{\mu_g}\rightarrow\bbR$ the canonical pairing defined in \eqref{eq:incomp_pairing} in Definition \ref{def:canonical_dual_incompressible}. Define the Lagrangian  $\ell: \mfk{X}_{\mu_g}\times \mfk{A}\rightarrow \bbR$ by
$$
\ell(u,D)=\frac{1}{2}\int_{M}\rho g(u,u)\mu_g=\frac{1}{2}\int_{M} g(u,u)D=\frac{1}{2}\langle u^{\flat}\otimes D,u\rangle_{\mfk{X}}=\frac{1}{2}\langle [u^{\flat}\otimes D],u\rangle_{\mfk{X}_{\mu_g}},
$$
The square brackets  denote an equivalence class, where we have modded out by elements of the form $\bd f \otimes \mu_g$. It follows that for all $(u,D)\in \mfk{X}_{\mu_g}\times \mfk{A}$ ,
$$
m=\frac{\sym{\delta} \ell}{\sym{\delta} u}(u,D)= [u^{\flat}\otimes D]\in \mfk{X}_{\mu_g}^{\vee},\quad u=\frac{\sharp}{D}m,
\quad \textnormal{and} \quad \frac{\sym{\delta} \ell}{\sym{\delta} D}(u,D)=\frac{1}{2}g(u,u),
$$
where the diffeomorphism $\frac{\sharp}{D}$ is defined in Section \ref{sec:duals_of_bundles}.
Using the diamond operation computed in the previous example, we find
$$
\frac{\sym{\delta} \ell}{\sym{\delta} D}(u,D)\diamond D=\left[\frac{1}{2}\bd g(u,u)\otimes D\right]\in  \mfk{X}^{\vee}_{\mu_g}.$$ 
Clebsch critical points (Theorem \ref{thm:Clebsch}) thus satisfy
\begin{equation*}
\begin{aligned}
\rmd   m_t&+  \pounds_{u_t}  m_t \rmd t +   \pounds_{\xi}  m_t \rmd \bZ_t \overset{\mfk{X}^{\vee}_{\mu_g}}{=} [\frac{1}{2}\bd g(u_t,u_t)\otimes D_t] \rmd t,\quad m = [u^{\flat}\otimes D], \\
\rmd D_t &+ \pounds_{u_t}D_t\rmd t + \pounds_{\xi} D_t \rmd \bZ_t  \overset{\mfk{A}}{=} 0,\\
\rmd \lambda_t &\overset{\mfk{A}^{\vee}}{=}\left(\pounds_{u_t}\lambda_t+\frac{1}{2}g(u_t,u_t)\right)\rmd t + \pounds_{\xi} \lambda_t \rmd \bZ_t.
\end{aligned}
\end{equation*}
Critical points of the Hamilton-Pontryagin action functional (Theorem \ref{thm:HamPont}) also satisfy the first two equations. 
Since $D$ is Lie-advected and the Lie derivative is a derivation, using the product rule (Lemma \ref{lem:product_and_chain}), we find
$$
[\rmd u^{\flat}_t\otimes D_t] +   [ \pounds_{u_t} u^{\flat}_t\otimes D_t] \rmd t +    [\pounds_{\xi} u^{\flat}_t\otimes D_t] \rmd \bZ_t \overset{\mfk{X}^{\vee}_{\mu_g}}{=} [\frac{1}{2}\bd g(u_t,u_t)\otimes D_t] \rmd t,
$$
or equivalently 
$$
\rmd  u^{\flat}_t\otimes D_t+   P(\pounds_{u_t} u^{\flat}_t\otimes D_t)\rmd t +    P(\pounds_{\xi} u^{\flat}_t\otimes D_t) \rmd \bZ_t \overset{\mfk{X}^{\vee}_{\mu_g}}{=} P(\frac{1}{2}\bd g(u_t,u_t)\otimes D_t)\rmd t,
$$
Upon invoking the definition of   $\mfk{X}^{\vee}_{\mu_g}$ in Definition \ref{def:canonical_dual_incompressible}, we find 
$$
\rmd u^{\flat}_t\otimes D_t +   \pounds_{u_t} u^{\flat}_t\otimes D_t\rmd t +    \pounds_{\xi} u^{\flat}_t\otimes D_t\rmd \bZ_t \overset{\mfk{X}^{\vee}}{=} \frac{1}{2}\bd g(u_t,u_t)\otimes D_t \rmd t-\bd p \otimes \mu_g\rmd t-\bd q\otimes \mu_g \rmd \bZ_t,
$$
where $\bd p\in C_T^{\alpha}(\Omega^0)$ and $\bd q\in \bclD_{Z,T}((\Omega^0)^K)$.

Applying $\frac{1}{D}$ (as defined in \ref{eq:one_over_density}) and  recalling that $\operatorname{div}_{\mu_g}u_t\equiv 0$ yields
\begin{equation*}
\begin{aligned}
\rmd u^{\flat}_t &+  \pounds_{u_t}  u^{\flat}_t\rmd t +   \pounds_{\xi}  u^{\flat}_t\rmd \bZ_t\ \overset{\Omega^1}{=}\frac{1}{2}\bd g(u_t,u_t)\rmd t  -\frac{1}{\rho_t}\bd p_t \rmd t -\frac{1}{\rho_t}\bd q_t \rmd \bZ_t, \quad \bd^* u^{\flat}=0=\operatorname{div}_{\mu_g}u_t,\\
\rmd \rho_t &+ \pounds_{u_t}\rho_t\rmd t + (\pounds_{\xi} \rho_t + \operatorname{div}_{\mu_g}\xi)\rmd \bZ_t  \overset{\Omega^0}{=} 0.
\end{aligned}
\end{equation*}

\begin{remark}
Thus, one sees that the Hodge decomposition necessitates introducing a `rough' Lagrangian multiplier (i.e., pressure term) $\rmd \bp_t=-\rho_t \rmd \sym{\pi}_t$ in \eqref{eq:euler_lagrange_mult}.  That is,
$$
\rmd \bp_t= \bd p \otimes \mu_g\rmd t-\bd q\otimes \mu_g \rmd \bZ_t,
$$
where
$$
\bd p_t= -Q\left(\rho_t \left(\pounds_{u_t}u_t^{\flat} - \frac{1}{2}\bd g(u_t,u_t)\right)\right) \quad \textnormal{and} \quad \bd q_t= -Q\left(\rho_t \pounds_{\xi}u_t^{\flat} \right),
$$
and  $Q: \Omega^1\rightarrow \bd \Omega^0$ denotes the  projection \eqref{eq:Hodge_decomp_k} onto flat one-forms.
\end{remark}

The following identity is well-known
$$
\pounds_{v}v^{\flat}-\frac{1}{2}\bd g(v,v)=(\nabla_{v}v)^{\flat}, \quad \forall v\in \mfk{X}_{C^{\infty}}, 
$$
where $\nabla: \mfk{X}_{C^{\infty}}\times \mfk{X}_{C^{\infty}}\rightarrow \mfk{X}_{C^{\infty}}$ is Levi-Civita connection (see, e.g., \cite{dupre1978classical}[Section 3]).\footnote{For the convenience of the reader, we  repeat the proof. For a given $u\in \mfk{X}_{C^\infty}$, define the tensor derivation $A_u=\pounds_{u}-\nabla_u$.  It follows that $A_uf\equiv 0$ for all $f\in C^\infty$ and that $A_uv=-\nabla_v u$ by the torsion-free property of the connection. Using these properties and that $A_u$ is a derivation, for a given $\alpha\in \Omega^1_{C^\infty}$, we have 
$
\bi_{v}(A_u\alpha)=\bi_{\nabla_vu}\alpha,
$
and hence $\bi_w(A_uv^{\flat})=\bi_{\nabla_wu}v^{\flat}=g(v,\nabla_wu) $ for all $w\in \mfk{X}_{C^\infty}$. 
Therefore, 
$$
 \bi_{w}(\bd g(u,v))=\nabla_w [g(u,v)]=g(v,\nabla_w u) + g(u,\nabla_w v)=\bi_{w}(A_uv^{\flat}+A_vu^{\flat}), \quad \forall u,v,w\in\mfk{X}_{C^\infty}.
$$
where we have also used $\nabla \eta=0$. Thus, 
$
\pounds_u v^{\flat} -\nabla_uv^{\flat} + \pounds_{v}u^{\flat}-\nabla_vu^{\flat}=\bd g(u,v),
$
which gives the formula upon setting $u= v$.
}
Thus, applying the $\sharp$  operator  to the equation for $u^{\flat}$ yields 
$$
\rmd u_t +  \nabla_{u_t} u_t\rmd t +   \left(\pounds_{\xi}  u^{\flat}_t\right)^{\sharp}\rmd \bZ_t= \frac{1}{\rho_t}\nabla  p_t \rmd t + \frac{1}{\rho_t}\nabla q_t \rmd \bZ_t,
$$
where in a local coordinate chart (see \eqref{eq:Lie_derivative_local}),
$$
(\pounds_{\xi}u^{\flat})^{\sharp}=\left(\xi^j\partial_{x^j}u^k+g^{ik}\xi^ju^l\partial_{x^j}g_{li}+g^{ik}g_{lj}u^l\partial_{x^i}\xi^j\right)\partial_{x^k}.
$$
It is worth noting that for all $u\in \mfk{X}_{\mu_g,C^{\infty}}$ and  $v,w\in \mfk{X}_{C^{\infty}}$, 
$$
(w,\operatorname{ad}_uv)_{\mfk{X}_{L^2}}=\langle w^{\flat}\otimes \mu_g,\operatorname{ad}_uv\rangle_{\mfk{X}}=\langle \pounds_{u}(w^{\flat}\otimes \mu_g),v\rangle_{\mfk{X}}=\langle \pounds_{u}w^{\flat}\otimes \mu_g,v\rangle_{\mfk{X}}=( (\pounds_{u}w^{\flat})^{\sharp},v)_{\mfk{X}_{L^2}}.
$$
\begin{remark}
When  $\xi\equiv 0$, the corresponding equation is the usual deterministic incompressible non-homogeneous Euler fluid equation (see, e.g.,\cite{boyer2012mathematical}[Ch.\ VI] or  \cite{marsden1976well}).
 \end{remark}
 
In case of a homogeneous fluid $\rho\equiv 1$, we find
$$
\rmd u^{\flat}_t +  \pounds_{u_t}  u^{\flat}_t\rmd t +   \pounds_{\xi}  u^{\flat}_t\rmd \bZ_t =\frac{1}{2}\bd g(u_t,u_t)\rmd t  -\bd p_t \rmd t -\bd q_t \rmd \bZ_t,
$$
or equivalently,
$$
\rmd u_t +  \nabla_{u_t} u_t\rmd t +   \left(\pounds_{\xi}  u^{\flat}_t\right)^{\sharp}\rmd \bZ_t= -\nabla  p_t \rmd t -\nabla q_t \rmd \bZ_t.
$$
In this case, another advected quantity $a \in \oplus_{i=1}^d \Lambda^0T^*M$  and its Lagrange multiplier $\lambda \in  \oplus_{i=1}^d \Lambda^dT^*M$ should be introduced into the Clebsch constraint to avoid reduction to  potential flow (see, e.g.,  Section \ref{sec:var_smooth_paths}).


\paragraph{Vorticity dynamics}
We will now discuss vorticity dynamics in both the inhomogeneous and homogeneous case. We assume $\operatorname{div}_{\mu_g}\xi\equiv 0$ and that  all quantities are regular enough subsequently to perform each calculation.  First,  notice that   for every $f\in C^\infty$, 
$$
\rmd f(\rho_t) + \pounds_{u_t}f(\rho_t)\rmd t+  \pounds_{\xi}f(\rho_t) \rmd \bZ_t=0.
$$
That is, $f(\rho_t)$ is advected by $\rmd x_t$ (see Remark \ref{rem:RPDE_Inverse}).
Let $\omega=\bd u^{\flat}\in \Omega^2$ be the vorticity two-form.  Since  the exterior derivative $\bd$ commutes with the Lie derivative, we obtain
$$
\rmd \omega_t +  \pounds_{u_t}  \omega _t\rmd t +   \pounds_{\xi}  \omega_t\rmd \bZ_t \overset{\Omega^2}{=}-\bd \rho_t^{-1}\wedge \bd p_t \rmd t-\bd \rho_t^{-1}\wedge \bd q_t \rmd \bZ_t \quad \textnormal{and} \quad 
\rmd \bd f(\rho_t) + \pounds_{u_t}(\bd f(\rho_t))\rmd t + \pounds_{\xi} (\bd f(\rho_t) )\rmd \bZ_t  \overset{\Omega^1}{=} 0.
$$
Thus, by the produt rule (Lemma \ref{lem:product_and_chain}), we get
$$
\rmd (\omega_t\wedge \bd f(\rho_t)) + \pounds_{u_t}( (\omega_t\wedge \bd f(\rho_t)) \rmd t + \pounds_{\xi}( (\omega_t\wedge \bd f(\rho_t)) \rmd \bZ_t=- \bd \rho_t^{-1}\wedge \bd p_t\wedge \bd f(\rho_t) \rmd t- \bd \rho_t^{-1}\wedge \bd q_t \wedge \bd f(\rho_t)\rmd \bZ_t.
$$
In particular, in dimension three, using that $\bd \rho_t\wedge \bd p_t \wedge \bd \rho_t\equiv  \bd \rho_t\wedge \bd q_t \wedge \bd \rho_t\equiv 0$, we find  
\begin{align*}
\rmd (\omega_t \wedge \bd \rho_t)+\pounds_{u_t}  (\omega _t\wedge \bd \rho_t)\rmd t +   \pounds_{\xi}  (\omega_t\wedge \bd \rho_t)\rmd \bZ_t\overset{\Omega^3}{=}0.
\end{align*}
Moreover, in dimension three, applying Stokes theorem, we get 
$$
\int_{M} \omega_t \wedge \bd f(\rho_t)=\int_{M} \omega_0 \wedge \bd f(\rho_0).
$$
We also have that 
$$
\rmd ( \omega_t f(\rho_t))+ \pounds_{u_t}(\omega_t f(\rho_t) \rmd t+ \pounds_{\xi} (\omega)t f(\rho_t)) \rmd \bZ_t = - \bd G(\rho_t)\wedge \bd p_t \rmd t- \bd G(\rho_t)\wedge \bd q_t \rmd \bZ_t,
$$
where $G$ is the anti-derivative of $g(\rho)=\frac{f(\rho)}{\rho^2}$.  Thus,  in dimension two,  we obtain
$$
\int_{M} \omega_t f(\rho_t) = \int_{M}\omega_0 f(\rho_0).
$$

In the homogeneous setting, the vorticity equation is given by
$$
\rmd \omega_t +  \pounds_{u_t}  \omega_t\rmd t +   \pounds_{\xi}  \omega_t\rmd \bZ_t =0.
$$
Using the Hodge-star operator $\star: \Omega^2 \rightarrow \Omega^{d-2}$ and setting $\tilde{\omega}=\star\omega \in \Omega^0$ in dimension two and $\tilde{\omega}=\sharp \star \omega\in \dot{\mfk{X}}_{\mu_g}$  in dimension three, we find 
$$
\partial_t \tilde{\omega}_t +\left((u_t\cdot \nabla)\tilde{\omega}_t -\mathbf{1}_{d=3}(\tilde{\omega}_t\cdot \nabla)u_t\right)\rmd t + \left((\xi\cdot \nabla)\tilde{\omega}_t -\mathbf{1}_{d=3}(\tilde{\omega}_t\cdot \nabla)\xi\right)\rmd t= 0.
$$
Here, we have used that $\sharp \star$ and the Lie derivative commute (see, e.g., Section A.6 of \cite{ besse2017geometric}.)  

In three dimensions, the helicity,  defined as $$\Lambda(\tilde{\omega}_t) =\int_{M} u^{\flat}_t\wedge \omega_t$$
measures the linkage of field lines of the divergence-free vector field $\tilde{\omega}$ \cite{arnold1999topological}. It follows that 
$$
\rmd  (u^{\flat}\wedge \omega) + \pounds_{u_t}(u^{\flat}\wedge \omega)\rmd t +  \pounds_{\xi}(u^{\flat}\wedge \omega)\rmd \bZ_t =-\bd \tilde{p}_t\wedge \omega_t\rmd t - \bd q \wedge \omega_t \rmd \bZ_t,
$$
where $\tilde{p}=p - \frac{1}{2}\bd g(u,u)$.  Thus, we find
$$
\Lambda(\tilde{\omega}_t)= \int_{M} u^{\flat}_t\wedge \omega_t =\int_{M} u^{\flat}_0\wedge \omega_0=\Lambda(\tilde{\omega}_0).
$$
Therefore,  the linkage number of the vorticity vector field $\Lambda(\tilde{\omega})$ is preserved by the 3D Euler fluid equations.

 In dimension two,  for any smooth  $f\in C^\infty$, 
$$
\rmd  f(\tilde{\omega}_t) + \pounds_{u_t}f(\tilde{\omega}_t)\rmd t +\pounds_{\xi}f(\tilde{\omega}_t)\rmd \bZ_t =0,
$$
and hence
$$
\int_{M} f(\tilde{\omega}_t) \mu_g=\int_{M}f(\tilde{\omega}_0)\mu_g.
$$
Letting $f(x)=x^2$, we obtain
$$
\int_{M} \tilde{\omega}_t^2 \mu_g=\int_{M}\tilde{\omega}_0^2\mu_g,
$$
which implies that in dimension two enstrophy is conserved.

\paragraph{Divergence and harmonic-free} It is worth noting that in the homogeneous setting, $u$ can only be recovered directly from $\omega$ if  $\clH^1_{\Delta}=\emptyset$ via the Biot-Savart law (see Section \ref{sec:Hodge}). Otherwise one needs to keep track of the harmonic constant. Nevertheless, one can  repeat the above analysis with the harmonic and divergence-free spaces $\dot{\mfk{X}}_{\mu_g}$ and $\dot{\mfk{X}}^{\vee}_{\mu_g}$ (see Definition \ref{def:canonical_dual_incompressible}) to derive 
$$
\rmd u^{\flat}_t +  \dot{P}\pounds_{u_t}  u^{\flat}_t\rmd t +   \dot{P}\pounds_{\xi}  u^{\flat}_t\rmd \bZ_t =\frac{1}{2}\dot{P}\bd g(u_t,u_t)\rmd t,
$$
and hence
$$
\rmd u^{\flat}_t +  \pounds_{u_t}  u^{\flat}_t\rmd t +   \pounds_{\xi}  u^{\flat}_t\rmd \bZ_t =\frac{1}{2}\bd g(u_t,u_t)\rmd t+(c_t -\bd p_t) \rmd t +(\tilde{c}_t-\bd q_t) \rmd \bZ_t,
$$
where  $c=H(\pounds_{u}  u^{\flat})\in C^\alpha_T(\clH^1_{\Delta})$  and    $\tilde{c}=H(\pounds_{\xi}  u^{\flat})\in \bclD_{Z,T}((\clH^1_{\Delta})^K)$ and $H$ is the projection onto Harmonic one-forms. Here, $u$ is constrained to be both divergence free and harmonic free. For this equation, $u^{\flat}$ (and $u$) can be recovered directly from $\omega$ via the Biot-Savart operator. This equation has been studied  in \cite{crisan2019solution, brzezniak2016existence, brzezniak2019existence, crisan2019well}. See, also, the discussion in \cite{flandoli2019high}.

\subsection{Rough Camassa-Holm equation and Burgers equation}
Let $M=\bbS$ be the flat one-dimensional  torus (i.e., the circle). Denote the standard normalized volume form by $\mu\in \operatorname{Dens}_{\,C^\infty}$  and coordinates by $x$. In this example, we take $A=\emptyset=A^{\vee}$. We  define  the Lagrangian  $\ell: \mfk{X}\rightarrow \bbR$  by
\begin{align*}
\ell(u)=\frac{1}{2}\int_{\bbS} (|u|^2+\alpha^2|\nabla_x u|^2)\mu =\frac{1}{2}\langle (\Lambda^2 u)^{\flat}\otimes \mu ,u\rangle_{\mfk{X}}, \quad\hbox{where}\quad \Lambda^{2}:=1-\alpha^2\nabla_x^2.
\end{align*}
It follows that 
 $$m=\frac{\sym{\delta} \ell}{\sym{\delta} u}(u)=(\Lambda^{2}u)^{\flat}\otimes \mu \in \mfk{X}^{\vee} \quad\hbox{and}\quad  u=\Lambda^{-2}\left(\frac{\sharp}{\mu} m\right).$$
If $(u,\eta)$ is a critical point of the Hamilton-Pontryagin action functional (Theorem \ref{thm:HamPont}), then 
 $$
 \rmd m_t + \pounds_{u_t}m_t \rmd t + \pounds_{\xi}m_t\rmd \bZ_t=0,
 $$
which we may interpret as Lie transport of the momentum in the Camassa-Holm equation along rough paths.
 
Since the  Lie derivative is a derivation and we have the explicit formula \eqref{eq:Lie_derivative_local} for one-forms, we find 
$$
\pounds_{v}m =\left(\pounds_{v}(\Lambda^{2} u)^{\flat} +  (\Lambda^{2} u)^{\flat}\operatorname{div}_{\mu} v\right)\otimes \mu=\left(v\nabla_x(\Lambda^{2} u) + 2 (\Lambda^{2} u) \nabla_x v\right)^{\flat}\otimes \mu,\; \; \forall v\in \mfk{X}.
$$
Thus, identifying $u$ with a scalar-valued function, we get
 $$
\rmd u_t + \Lambda^{-2}\left(u_t\partial_x(\Lambda^2u_t)+2(\Lambda^2u_t)\partial_xu_t\right) \rmd t + \Lambda^{-2}\left(\xi \partial_x(\Lambda^2u_t)+2(\Lambda^2u_t)\partial_x\xi\right) \rmd \bZ_t=0.
$$
After some simplification, we find 
$$
\rmd u_t + \left(u_t\partial_x u_t + \partial_x\Lambda^{-2}\bigg(\Big( |u_t|^2+\frac{\alpha^2}{2} |\partial_x u_t|^2\Big)\bigg) \right)\rmd t +\bigg(\xi \partial_x u_t +  \Lambda^{-2}\left(2 u_t \partial_x \xi + \alpha^2\partial_{x}^2\xi \partial_x u_t\right)\bigg) \rmd \bZ_t=0,
$$
written as a nonlocal Cauchy problem with the pseudo-differential operator $\Lambda^{-2}$. Indeed, notice that if one substitutes $\xi$ with  $u$ in the $\rmd\bZ_t$-term, then one obtains the same  operator in $\rmd t$-term. 

Now, if  $\alpha=0$, we obtain the Burgers equation on rough paths,
$$
\rmd u_t + 3u_t\partial_xu_t \rmd t + \left(\xi \partial_xu_t+2u_t\partial_x\xi\right) \rmd \bZ_t=0.
$$
The properties of the stochastic Burgers equation with stochastic transport noise have been investigated, e.g., in \cite{alonso2019burgers}, and the properties of the Burgers equation with rough transport noise have been studied in \cite{hocquet2019generalized}.

\subsection{Rough Euler  equations for adiabatic compressible flows}\label{sec.EulAdiabatCompFuid}

Let $A=\Lambda^dT^*M\oplus \Lambda^0T^*M$ and $A^{\vee}= \Lambda^0T^*M\oplus \Lambda^dT^*M$.   Denote the advected variables by $\ba = (\bD, \bs)\in \mfk{A}_{\clF_3}=\operatorname{Dens}_{\clF_3}\oplus \Omega^0_{\clF_3}$ and the associated Lagrangian multipliers by $\sym{\lambda}=(\sym{f},\sym{\beta})\in \mfk{A}_{\clF_4}^{\vee}=\Omega^0_{\clF_4}\oplus \operatorname{Dens}_{\clF_4}$.  Let $\rho \in \Omega^0_{\clF_3}$ be such that $D=\rho \mu_g$.  The advected variables comprise the thermodynamic evolution variables mass/volume, $\rho$  and the entropy/mass, $s$.
The internal energy/mass, $e(\rho, s)$, obeys the First Law of Thermodynamics, given by
$$
de(\rho,s) = \frac{p}{\rho^2}d\rho + T ds,
$$
with pressure $p(\rho, s)$ and temperature $T(\rho, s)$.

Define the Lagrangian  $\ell: \mfk{X}_{\mu_g}\times \mfk{A}\rightarrow \bbR$ by
$$
\ell(u, a) = \int_{M} \left( \frac{1}{2} g(u,u)-e(\rho,s) \right)D.
$$
It follows that  
$$
m=\frac{\sym{\delta} \ell}{\sym{\delta} u}(u,a)= u^{\flat}\otimes D\in \mfk{X}^{\vee} \quad \textnormal{and} \quad \frac{\sym{\delta} \ell}{\sym{\delta} a}(u,a)=\left(\frac{1}{2}g(u,u)-h(p,s),\;- TD\right),
$$
where  $h(p,s)= e(\rho,s)+ p/\rho$ is the specific enthalpy/mass, which satisfies
$$
dh(p,s) = \frac{1}{\rho}dp + Tds.
$$
Applying the calculations with the diamond operation $(\diamond)$ in Section \ref{sec:Euler-Lag_mult} yields
$$
\frac{\sym{\delta} \ell}{\sym{\delta} a}(u,a)\diamond a=\left(\left(\frac{1}{2}\bd g(u,u)-\bd h(p,s)\right)\otimes D,\;T\bd s\otimes D\right).
$$

Critical points of the Clebsch and Hamilton-Pontryagin action functionals  satisfy
\begin{equation*}
\begin{aligned}
\rmd   m_t+  \pounds_{dx_t}  m_t \rmd t  \overset{\mfk{X}^{\vee}}{=} \left(\frac{1}{2}\bd g(u_t,u_t) - \bd h(p_t,s_t)\right)\otimes D_t\rmd t &+ T_t\bd s_t \otimes  D_t \rmd t,  \\
\rmd D_t + \pounds_{dx_t}D_t\rmd t  \overset{\operatorname{Dens}}{=} 0,\quad 
\& \quad 
\rmd s_t &+ \pounds_{dx_t}s_t  \overset{\operatorname{\Omega^0}}{=} 0.
\end{aligned}
\end{equation*}
Since $D$ is Lie-advected and the Lie derivative is a derivation, using the product rule (Lemma \ref{lem:product_and_chain}) and applying the diffeomorphism $\frac{1}{D}$ yields
$$
\rmd u_t^{\flat}+ \pounds_{{\rm d}x_t} u_t^\flat =\left( \bd \,\bigg(  \frac{1}{2}g(u_t,u_t)  -\,h(p_t,s_t) \bigg) + T_t\bd s_t\right)\rmd t
=:\left( \frac{1}{2}\bd g(u_t,u_t) -\, \frac{1}{\rho_t}\bd p_t \right)\rmd t
\,.
$$
Restricting to dimension three and working with enough regularity to perform the subsequent calculations implies three advected quantities,
\begin{align}\label{eq:3AdEul-advect}
(\rmd + \pounds_{{\rm d}x_t})D_t= 0
\,,\qquad (\rmd + \pounds_{{\rm d}x_t})s_t = 0
\,,\qquad
(\rmd + \pounds_{{\rm d}x_t})( \bd u_t^\flat\wedge \bd s_t ) \overset{\Omega^3}{=} 0\,.
\end{align}

Let $\omega= \bd u^{\flat}$ denote the vorticity two-form and let $\tilde{\omega}$ denote the corresponding divergence-free vector field. From the three quantities in \eqref{eq:3AdEul-advect}, one may construct the following advected scalar quantity known as the \emph{potential vorticity}
$$
(\rmd + \pounds_{{\rm d}x_t})\Omega_t\overset{\Omega^0}{ =} 0,
\quad\hbox{where}\quad
\Omega_t :=D_t^{-1} \sym{\omega}_t\wedge \bd s_t = \rho_t^{-1}\tilde{\omega}\cdot\nabla s_t.
$$
Consequently, the following functional is conserved for the adiabatic compressible Euler equations on GRPs
$$
C_\Phi := \int_M \Phi(\Omega_t, s_t)D,
$$
for any smooth function $\Phi: \bbR^2\rightarrow \bbR$.

\section{Proof of main results}\label{sec:proofs}

\subsection{Proof of the Clebsch variational principle Theorem \ref{thm:Clebsch}}\label{proof:Clebsch}
\begin{proof}
It is worth noting that this proof closely mirrors the proof in \cite{holm2015variational} for stochastic variational principles. Nonetheless, we repeat the proof for the convenience of the reader.

If   $(u,a,\lambda )\in \mathit{Clb}_{\bZ}$ is a critical point of the action functional, 
then it satisfies
$$
0=\frac{d}{d\epsilon}\bigg|_{\epsilon=0}S^{\mathit{Clb}_{\bZ}}(u^{\epsilon},\ba^{\epsilon},\boldsymbol{\lambda}^{\epsilon})=I(\sym{\delta} u)+II(\sym{\delta}a)+III(\sym{\delta}\lambda),
$$
where 
\begin{align*}
I(\sym{\delta} u)&=\int_0^T \left\langle \frac{ \sym{\delta} \ell}{\sym{\delta} u}(u_t,a_t) -  \lambda_t \diamond a_t, \sym{\delta} u_t \right\rangle_{\mathfrak{X}} \rmd t\\
II(\sym{\delta}a)&=\int_0^T \langle \lambda_t,  \partial_t \sym{\delta} a_t\rangle_{\mfk{A}}\rmd t+ \int_0^T \langle \frac{\sym{\delta} \ell}{\sym{\delta}a}(u_t,a_t)+\pounds_{u_t}^T \lambda_t,\sym{\delta}a_t\rangle_{\mfk{A}}\rmd t +   \int_0^T \langle \pounds_{\xi}^*\lambda_t,  \sym{\delta} a_t \rangle_{\mfk{A}}  \rmd \bZ_t\\
III(\sym{\delta}\lambda)&=\int_0^T \langle \sym{\delta} \lambda _t, \rmd \ba_t \rangle_{\mfk{A}}+ \int_0^T\langle \sym{\delta}\lambda_t, \pounds_{u_t}a_t\rangle_{\mfk{A}}\rmd t + \int_0^T\langle \sym{\delta}\lambda_t, \pounds_{\xi} a_t  \rangle_{\mfk{A}}\rmd \bZ_t.
\end{align*}
Here, we have used the definition of the diamond operator 
and \eqref{asm:Lagrangian} to exchange the order of derivative in $\epsilon$ and the time-integral for the Lagrangian terms. Since we may always take $\sym{\delta} u\equiv 0$, $\sym{\delta} a\equiv 0$, and $\sym{\delta}\lambda\equiv 0$, we  conclude that $I(\sym{\delta} u)=0$, $II(\sym{\delta} a)=0$, and $III(\sym{\delta} \lambda)=0$ for all smooth $(\sym{\delta} u, \sym{\delta} \lambda, \sym{\delta} a)$ that vanish at $t=0$ and $t=T$. Splitting the variations in time and space and applying the fundamental lemma of calculus of variations in Lemmas \ref{lem:fund_calc_var}, and \ref{lem:product_and_chain}, we find $m=\frac{\sym{\delta} \ell}{\sym{\delta} u}(u,a)\overset{\mfk{X}^{\vee}}{=} \lambda \diamond a$ and that $a,\lambda$ solve the equations given in \eqref{eq:EPeq}.
Upon applying Lemma \ref{lem:product_and_chain} with the continuous bilinear pairing $\diamond: \mfk{X}\times \mfk{A}\rightarrow \mfk{X}^{\vee}$ , we obtain
$$
m_t= \lambda_0 \diamond a_0 + \int_0^t \rmd \lambda_r\diamond a_r + \int_0^t \lambda_r \diamond \rmd a_r.
$$

Subtracting $\int_0^t\frac{\sym{\delta} \ell}{\sym{\delta} a}(u_r,a_r)\rmd r$ from both sides of the above, testing against a smooth $\phi \in \mfk{X}_{C^{\infty}}$ and working in differential notation yields 
\begin{align*}
\langle \rmd m &- \frac{\sym{\delta} \ell}{\sym{\delta} a} \diamond a \rmd t , \phi \rangle_{\mathfrak{X}}    = \langle  (\rmd\lambda- \frac{\sym{\delta} \ell}{\sym{\delta} a} \rmd t)\diamond a +\lambda \diamond \rmd a   , \phi \rangle_{\mathfrak{X}}\smallskip 
\\&=  \langle  \pounds_{u}^* \lambda   \diamond a , \phi \rangle_{\mathfrak{X}}\rmd t + \langle \pounds_{\xi}^* \lambda    \diamond a , \phi \rangle_{\mathfrak{X}} \rmd \bZ - \langle \lambda \diamond  \pounds_{u}a , \phi \rangle_{\mathfrak{X}} \rmd t - \langle \lambda \diamond \pounds_{\xi} a  , \phi \rangle_{\mathfrak{X}} \rmd \bZ\;\; (\textnormal{by eqn.\ for } \lambda \;\;\&\;\;a) \smallskip
\\&=  - \langle   \pounds_{u}^*\lambda  ,  \pounds_{\phi}a   \rangle_{\mfk{A}}\rmd t - \langle \pounds_{\xi}^* \lambda , \pounds_{\phi} a \rangle_{\mfk{A}} \rmd \bZ + \langle \lambda, \pounds_{\phi} \pounds_{u}a \rangle_{\mfk{A}} \rmd t + \langle \lambda , \pounds_{\phi} \pounds_{\xi} a \rangle_{\mfk{A}} \rmd \bZ\;\; \quad(\textnormal{by def.\ of } \diamond) \smallskip
\\&=  - \langle    \lambda  ,  (\pounds_{u} \pounds_{\phi} -  \pounds_{\phi} \pounds_{u})a   \rangle_{\mfk{A}} \rmd t - \langle    \lambda  ,  (\pounds_{\xi} \pounds_{\phi}   -  \pounds_{\phi} \pounds_{\xi} )a    \rangle_{\mfk{A}} \rmd \bZ   
\smallskip
\\&=  - \langle    \lambda  ,  \pounds_{[u,\phi]}a   \rangle_{\mfk{A}} \rmd t - \langle    \lambda  ,  \pounds_{[\xi,\phi]}a    \rangle_{\mfk{A}} \rmd \bZ  \;\;\quad(\textnormal{by prop.\ of Lie derivative}) 
\smallskip
\\&=   \langle    \lambda \diamond a ,  [u, \phi]    \rangle_{\mathfrak{X}} \rmd t + \langle    \lambda \diamond a   ,  [\xi,\phi]   \rangle_{\mathfrak{X}} \rmd \bZ \;\;\quad(\textnormal{by def.\ of } \diamond) 
\smallskip
\\&= - \langle    \lambda \diamond a ,  \operatorname{ad}_{u}\phi     \rangle_{\mathfrak{X}} \rmd t - \langle    \lambda \diamond a   ,  \operatorname{ad}_{\xi}\phi   \rangle_{\mathfrak{X}} \rmd \bZ\;\;\quad(\textnormal{by def.\ of } \operatorname{ad}_{u})  
\smallskip
\\&=  -  \langle     \pounds_{u} m ,   \phi    \rangle_{\mathfrak{X}} \rmd t - \langle     \pounds_{\xi} m ,   \phi    \rangle_{\mathfrak{X}} \rmd \bZ\;\;\quad(\textnormal{b/c } \operatorname{ad}^*_vm=\pounds_{v}m). 
\end{align*}
Consequently,
$$\rmd m_t +    \pounds_{u_t} m_t \rmd t +   \pounds_{\xi} m_t \rmd \bZ_t \overset{\mathfrak{X}_{\vee}}{=} \frac{\sym{\delta} \ell}{\sym{\delta} a} \diamond a_t \rmd t. $$
The converse can be obtained by reversing the above proof. 
\end{proof}

\subsection{Proof of the rough Lie chain rule  Theorem \ref{thm:Lie_chain} }\label{sec:proof_Lie_chain}
\begin{proof}
We will prove the statement in the following steps. 
\begin{enumerate}
\item We prove the formula for scalar functions by working in a local chart;
\item We  prove the formula for vectors by reducing to step 1 and using the product formula;
\item We  prove the formula for one-forms from steps 1 and  2 and the product formula;
\item Using steps 3 and 4, we apply an induction argument to prove the general formula.
\end{enumerate}

For simplicity, we will drop the $\rmd t$-terms and time-dependence on $\xi$, and  assume $s=0$. 
That is, we consider the $C^{\infty}$-flow $\eta\in C_T^{\alpha}(\operatorname{Diff}_{C^\infty})$ satisfying
$$
d\eta_{t}X=\xi(\eta_{t}X)\rmd \bZ_t, \;\; \eta_0X=X\in M.
$$
Since we are working in $C^{\infty}$, we will simply write $\bclD_{\bZ}$ for the controlled spaces. 

\textbf{Step 1.}
Assume that $f\in C_T^{\alpha}(C^\infty)$ has the decomposition 
$$
f_t = f_0+ \int_0^t\pi_r\rmd \bZ_r, \quad t\in [0,T].
$$
We aim to show that 
\begin{equation}\label{eq:scalar_Lie_chain}
\eta_t^*f_t =f_0+\int_0^t \eta_r^*\left(\pi_r+\xi[f_r]\right)\rmd \bZ_r
\end{equation}
and
$$
\eta_{t*}f_t =f_0+\int_0^t \left(\eta_{r*}\pi_r-\xi[\eta_{r*}f_r]\right)\rmd \bZ_r,
$$
where both integrals are understood in the sense of controlled calculus. We focus only on the pull-back formula \eqref{eq:scalar_Lie_chain} as the equation for the push-forward can be shown in a similar way. 

Towards this end, let us fix a coordinate chart $(U,\phi)$ with coordinates denoted by $x$. Let $\rho := x \circ \eta$, $\Xi^i := \xi [ x^i]$, $F_t = f_t \circ x^{-1}$ and $b(t,\cdot) = \pi_t ( x^{-1}(\cdot))$.  
We will now show that 
\begin{equation} \label{F rho equation}
F_t(\rho_t)  = F_0 + \int_0^t \left( \Xi^i(\rho_r) \partial_{x^i} F_r(\rho_r)  + b(r, \rho_r) \right)\rmd \bZ_r ,
\end{equation}
which is  \eqref{eq:scalar_Lie_chain} written in local coordinates. Since $(\phi,U)$ was arbitrary, proving \eqref{F rho equation} completes step 1.

To see this, it will be convenient to spell out the expansion in terms of scalars. We identify $\Xi(\cdot)$ as an operator on $\mathcal{L}(\bbR^K, \bbR^d)$ acting on $Z$ with $\Xi_k(\cdot) \delta Z^k_{st}$ and write the Davie's expansions of  $\rho$ and $F$:
\begin{equation} \label{rho expansion}
\delta \rho_{st} = \Xi_k (\rho_s) \delta Z^k_{st} +  \partial_{x^i} \Xi_k(\rho_s) \Xi^i_l(\rho_s) \bbZ^{lk}_{st} + \rho_{st}^{\natural}
\end{equation}
and 
\begin{equation} \label{F expansion}
\delta F_{st}(\cdot) = b_k (s,\cdot) \delta Z^k_{st} + b_{k,l}'(s,\cdot)\bbZ^{lk}_{st} + b_{st}^{\natural}( \cdot),
\end{equation}
where $|\rho_{st}^{\natural} | \lesssim |t-s|^{3 \alpha}$ and $|b^{\natural}_{st}|_{C^\infty} \lesssim |t-s|^{3 \alpha}$.

To prove \eqref{F rho equation}, we Taylor expand  $F_s$, $b_k(s,\cdot)$, and $b_{k,l}'(s, \cdot)$, and use \eqref{F expansion} to write
\begin{align*}
F_t(\rho_t) - F_s(\rho_s) & = F_t( \rho_t) - F_s(\rho_t)  + F_s(\rho_t) - F_s(\rho_s) \\
& = b_k(s,\rho_t) \delta Z^k_{st} + b_{k,l}'(s,\rho_t) \bbZ^{lk}_{st} + b_{st}^{\natural}(\rho_t)  + \partial_{x^i} F_s (\rho_s) \delta \rho_{st}^i + \frac12 \partial_{x^j} \partial_{x^i} F_s (\rho_s) \delta \rho_{st}^i \delta \rho_{st}^j + o( | \delta \rho_{st}|^{3}) \\
& = ( b_k(s,\rho_s) + \partial_{x^i} b_k(s,\rho_s) \delta \rho^i_{st} + o(|\delta \rho_{st}|^2)  )  \delta Z^k_{st} + ( b_{k,l}'(s,\rho_s) + o(|\delta \rho_{st}|))  \bbZ^{lk}_{st} + b_{st}^{\natural}(\rho_t) \\
&  + \partial_{x^i} F_s (\rho_s) \Xi^i_k(\rho_s) \delta Z^k_{st} + \partial_{x^i} F_s(\rho_s)  \partial_{x^j} \Xi^i_k(\rho_s)\Xi_l^j(\rho_s) \bbZ^{lk}_{st} +  \partial_{x^i} F_s(\rho_s) \rho_{st}^{i, \natural} \\
& + \frac12 \partial_{x^j} \partial_{x^i} F_s (\rho_s)  \Xi_{l}^i \Xi^j_k \delta Z^k_{st}  \delta Z^l_{st} + o( | \delta \rho_{st}|^{3}).
\end{align*}
 Since  $\bZ$ is geometric (i.e.,  $\bbZ_{st}^{lk} + \bbZ_{st}^{kl} = Z_{st}^l Z_{st}^k$),  plugging in the expansion \eqref{rho expansion}, we find
\begin{align*}
F_t(\rho_t) - F_s(\rho_s) & = \left( b_k(s,\rho_s) + \partial_{x^i} F_s (\rho_s) \Xi_k^i(\rho_s) \right)  \delta Z^k_{st} + \left(  \partial_{x^i} b_k(s,\rho_s) \Xi^i_l(\rho_s) + \partial_{x^i} b_l(s,\rho_s) \Xi^i_k(\rho_s)\right) \bbZ_{st}^{lk}\\
& \quad + \left( \partial_{x^j} \partial_{x^i} F_s (\rho_s)  \Xi_{l}^i \Xi^j_k  +  \partial_{x^i} F_s(\rho_s)  \partial_{x^j} \Xi^i_k(\rho_s)\Xi_l^j(\rho_s)  \right) \bbZ_{st}^{lk} + o(|t-s|^{3 \alpha}) .
\end{align*}
Straightforward, but tedious, computations show that the local expansion 
\begin{align*}
\Psi_{st}  & : = \left( b_k(s,\rho_s) + \partial_{x^i} F_s (\rho_s) \Xi_k^i(\rho_s) \right)  \delta Z^k_{st}  \\
& \qquad + \left(  \partial_{x^i} b_k(s,\rho_s) \Xi^i_l(\rho_s) + \partial_{x^i} b_l(s,\rho_s) \Xi^i_k(\rho_s) +\partial_{x^j} \partial_{x^i} F_s (\rho_s)  \Xi_{l}^i \Xi^j_k  +  \partial_{x^i} F_s(\rho_s)  \partial_{x^j} \Xi^i_k(\rho_s)\Xi_l^j(\rho_s)  \right) \bbZ_{st}^{lk}
\end{align*}
satisfies $|\delta_2 \Psi_{s\theta t}| \lesssim |t-s|^{3 \alpha}$. By the uniqueness in Lemma \ref{lem:sewing}, we get \eqref{F rho equation}.

\bigskip

\textbf{Step 2.}
Let $V\in C_T^{\alpha}(\mfk{X}_{C^\infty})$ be such that
$$
V_t=V_0+\int_0^t\pi_r\rmd \bZ_r,\quad t\in [0,T],
$$
for some $(\pi, \pi') \in \clD_Z$. For any $f\in C^{\infty},$ we have
$$
(\eta_t^*V_t)[f]=(\eta_t^*V_t)[\eta_t^*\eta_{t*}f]=\eta_t^*(V_t[\eta_{t*}]).
$$
Recall that $\eta_t=\eta_{t*}\in C^{\alpha}_T(C^{\infty})$ satisfies \eqref{eq:push-forward_flow} with $s=0$. Applying Lemma \ref{lem:product_and_chain} with the continuous bilinear map $B: \mfk{X}_{C^{\infty}}\times C^{\infty}\rightarrow  C^{\infty}$, we find
$$
V_t[\eta_t] = V_0[f] +\int_0^t \left(\pi_r[g_r] - V_r[\xi[g_r]]\right)\rmd \bZ_r  .
$$
Moreover, making use of step 1 with $f_t=V_t[\eta_t]\in C^{\infty}$,  we obtain
\begin{align*}
(\eta_t^*V_t)[f]=\eta_t^*(V_t[\eta_t])&=V_0[f]+\int_0^t \eta_r^*\left(\pi_r[g_r]+\xi[V_r[g_r]]-V_r[\xi[g_r]]\right)\rmd \bZ_r\\
&=V_0[f]+\int_0^t \eta_r^*\left(\pi_r[g_r]+[\xi,V_r][g_r]\right)\rmd \bZ_r\\
&=V_0[f]+\int_0^t \left(\eta_r^*\left(\pi_r+[\xi,V_r]\right)\right)[f]\rmd \bZ_r.
\end{align*}
Because $f$ was arbitrarily chosen, we conclude that $(\eta^* V, \eta^* \pounds_{\xi} V + \eta^* \pi) \in \bclD_{\bZ}$ and 
$$
\eta_t^*V_t=V_0+\int_0^t \eta_r^*\left(\pi_r+\pounds_{\xi}V_r\right)\rmd \bZ_r.
$$
Noting that 
$$
(\eta_{t*}V_t)[f]=(\eta_{t*}V_t)[\eta_{t*}\eta_{t}^{*}f]=\eta_{t*}(V_t[\eta_t^{*}f]),
$$
and following a similar proof, we find that $(\eta_* V,  - \pounds_{\xi} \eta_* V + \eta_* \pi) \in \bclD_{\bZ}$ and 
\begin{equation} \label{lie chain rule push forward vector}
\eta_{t*}V_t=V_0+\int_0^t \left(\eta_{r*}\pi_r-\pounds_{\xi}(\eta_{r*}V_r)\right)\rmd \bZ_r.
\end{equation}

\bigskip

\textbf{Step 3.}
Assume that $\alpha\in C^{\alpha}_T(\Omega^1_{C^\infty})$ has the decomposition
$$
\alpha_t=\alpha_0+\int_0^t\pi_r\rmd \bZ_r.
$$
Fix an arbitrary vector $V\in \mfk{X}_{C^\infty}$ independent of $t$. Using \eqref{lie chain rule push forward vector} and 
Lemma \ref{lem:product_and_chain}, we get
$$
\alpha_t(\eta_{t*}V)=\alpha_0(V)+\int_0^t\left(\pi_r(\eta_{r*}V)-\alpha_r([\xi,\eta_{r*}V])\right)\rmd \bZ_r.
$$
Applying \eqref{eq:scalar_Lie_chain}, we obtain 
$$
(\eta_t^*\alpha_t)(V)=\eta_t^*\left(\alpha_t(\eta_{t*}V)\right)=\alpha_0(V)+\int_0^t\eta_r^*\left(\pi_r(\eta_{r*}V)-\alpha_r([\xi,\eta_{r*}V])+[\xi(\alpha_r(\eta_{r*}V)]\right)\rmd \bZ_r.
$$
The derivation property of the Lie derivative implies
$$
\eta_t^* \xi[\alpha_t(\eta_{t*}V) ]=(\eta_t^*(\pounds_{\xi}\alpha_t))(V)+ \eta_t^*\left(\alpha_t([\xi,\eta_{t*}V])\right),
$$
Noting that
$
\eta_t^*\left(\pi_t(\eta_{t*}V)\right)=(\eta_t^*\pi_t)(V)
$ 
and that $V$ was arbitrary, we  obtain
$$
\eta_t^*\alpha_t=\alpha_0+\int_0^t\eta_r^*\left(\pi_r+\pounds_{\xi}\alpha_r\right)\rmd \bZ_r.
$$
Following a similar argument, we get 
$$
\eta_{t*}\alpha_t = \alpha_0+\int_0^t \eta_{r*}\pi_r - \pounds_{\xi}\left(\eta_{r*}\alpha_r\right)\rmd \bZ_r,
$$
which completes step 3. 

\bigskip

\textbf{Step 4.}
Let us show how to extend to $\clT^{lk}_{C^\infty}$. Let $V_1, \dots, V_{k}\in \mfk{X}_{C^\infty}$, and $\alpha_1, \dots, \alpha_l \in\Omega^1_{C^\infty}$. We recall that 
$$
(\eta_t^* \tau_t)(\alpha_1, \dots, \alpha_l, X_1, \dots, X_k) = \eta_t^*\big(  \tau_t( \eta_{t*} \alpha_1, \dots, \eta_{t*}\alpha_l,\eta_{t*} X_1, \dots, \eta_{t*}X_k) \big)
$$
Using induction in the product formula, we obtain
\begin{align*}
\tau_t( \eta_{t*} \alpha_1, \dots, \eta_{t*} \alpha_{l}, \eta_{t*} V_1,  & \dots, \eta_{t*} V_{k})   = \tau_0(  \alpha_1, \dots, \alpha_l, V_1, \dots,  V_{k})  \\
& + \int_0^t  \big[ \gamma_r(  \eta_{r*} \alpha_1, \dots, \eta_{r*} \alpha_{l},\eta_{r*} V_1, \dots, \eta_{r*} V_{k})  \\
& - \sum_{j=1}^{l} \tau_r (  \eta_{r*} \alpha_1, \dots, \pounds_{\xi} \eta_{r*} \alpha_j, \dots, \eta_{r*} \alpha_{l}, \eta_{r*} V_1, \dots ,   \eta_{r*} V_{k} )  \\
& - \sum_{j=1}^{k} \tau_r (  \eta_{r*} \alpha_1, \dots, \eta_{r*} \alpha_{l}, \eta_{r*} V_1, \dots , \pounds_{\xi} \eta_{r*} V_j, \dots,   \eta_{r*} V_{k} ) \big] d \bZ_r
\end{align*}
which from \eqref{eq:scalar_Lie_chain} yields
\begin{align*}
\eta_t^* \tau_t(\alpha_1, \dots, \alpha_l, V_1, \dots, V_{k}) & = \tau_0( \alpha_1, \dots, \alpha_l, V_1, \dots,  V_{k}) \\
& + \int_0^t  \big[ \eta_r^*(  \gamma_r( \eta_{r*}\alpha_1, \dots, \eta_{r*}\alpha_l,\eta_{r*} V_1, \dots, \eta_{r*} V_{k}) ) \\
&-  \sum_{j=1}^{k} \tau_r (\eta_{r*}\alpha_1, \dots, \pounds_{\xi} \eta_{r*} \alpha_j, \dots,  \eta_{r*}\alpha_l, \eta_{r*} V_1, \dots  , \eta_{r*} V_{k} )  \\
& - \sum_{j=1}^{k} \tau_r (\eta_{r*}\alpha_1, \dots, \eta_{r*}\alpha_l, \eta_{r*} V_1, \dots , \pounds_{\xi} \eta_{r*} V_j, \dots,   \eta_{r*} V_{k} ) \big] d \bZ_r \\
& + \int_0^t \eta_r^*( \xi[ \tau_r( \eta_{r*} V_1, \dots, \eta_{r*} V_{k})] )\rmd \bZ_r.
\end{align*}
By the derivation property of the Lie derivative, we get
\begin{align*}
\xi[ \tau_r(  \eta_{r*}\alpha_1, \dots, \eta_{r*}\alpha_l, \eta_{r*} V_1, \dots, \eta_{r*} V_{k})] &   = (\pounds_{\xi} \tau_r)(\eta_{r*}\alpha_1, \dots, \eta_{r*}\alpha_l,\eta_{r*} V_1, \dots, \eta_{r*} V_{k}) \\
&  + \sum_{j=1}^{l} \tau_r (  \eta_{r*} \alpha_1, \dots, \pounds_{\xi} \eta_{r*} \alpha_j, \dots, \eta_{r*} \alpha_{l}, \eta_{r*} V_1, \dots ,   \eta_{r*} V_{k} )  \\
& +   \sum_{j=1}^{k} \tau_r( \eta_{r*}\alpha_1, \dots, \eta_{r*}\alpha_l, \eta_{r*} V_1, \dots , \pounds_{\xi} \eta_{r*} V_j, \dots,   \eta_{r*} V_{k} ),
\end{align*}
and hence
\begin{align*}
\eta_t^* \tau_t(\alpha_1, \dots, \alpha_l, V_1, \dots, V_{k}) & = \tau_0(  V_1, \dots,  V_{k}) + \int_0^t  \big[ \eta_r^*  \gamma_r( \eta_{r*}\alpha_1, \dots, \eta_{r*}\alpha_l,\eta_{r*} V_1, \dots, \eta_{r*} V_{k}) \\ 
& + \eta_r^* (\pounds_{\xi} \tau_r)(\eta_{r*}\alpha_1, \dots, \eta_{r*}\alpha_l,\eta_{r*} V_1, \dots, \eta_{r*} V_{k})  \big] d \bZ_r \\
& = \tau_0( \alpha_1, \dots, \alpha_l, V_1, \dots,  V_{k}) \\
& + \int_0^t  (\eta_r^*\gamma_r)( \alpha_1, \dots, \alpha_l, V_1, \dots,  V_{k}) + (\eta_r^* \pounds_{\xi} \tau_r)( \alpha_1, \dots, \alpha_l,V_1, \dots,  V_{k})   d \bZ_r .
\end{align*}
Since $\alpha_1, \dots, \alpha_l$ and $V_1, \dots, V_k$ were arbitrary, the result follows. 
\end{proof}

\subsection{Proof of the rough Kelvin--Noether Theorem \ref{thm:Kelvin}}\label{proof:Kelvin}

\begin{proof}
Let $\mu\in  \operatorname{Dens}_{C^{\infty}}$ be an arbitrary  non-vanishing density and set  $\rho=\frac{dD}{d\mu}\in C^{\infty}$ so that $D=\rho \mu$. Recall that for all $w\in \mfk{X}_{C^{\infty}}$,  $\pounds_{w}D=(\pounds_w\rho + \operatorname{div}_{\mu}w)\mu$. It follows that  for all $t\in [0,T]$,
$$
\rho_t = \rho_0-\int_0^t (\pounds_{u_r}\rho_r+\rho_r\operatorname{div}_{\mu}u_r )\rmd r -\int_0^t (\pounds_{\xi}\rho_r+\rho_r\operatorname{div}_{\mu}\xi )\rmd \bZ_r.
$$
Using the Lemma \ref{lem:product_and_chain}  and the identity $\pounds_{w}\frac{1}{\rho}=-\frac{1}{\rho^2}\pounds_{w} \rho$, $w\in \mfk{X}_{C^{\infty}}$, we find
\begin{align*}
\frac{1}{\rho_t}&=\frac{1}{\rho_0} +\int_0^t\left(-\pounds_{u_r}\frac{1}{\rho_r}+\frac{1}{\rho_r}\operatorname{div}_{\mu}u_r \right)\rmd r +\int_0^t \left(-\pounds_{\xi}\frac{1}{\rho_r}+\frac{1}{\rho_r}\operatorname{div}_{\mu}\xi \right)\rmd\bZ_r.
\end{align*}
For all $m=\alpha\otimes \nu \in \mfk{X}_{C^{\infty}}^{\vee}$ and $w\in \mfk{X}_{C^{\infty}}$, we have
$$
\pounds_{w} m =\pounds_{w} \left( \alpha\frac{d\nu}{d\mu}\otimes \mu\right)=\left( \pounds_{w}\left(\alpha_i \frac{d\nu}{d\mu}\right)+(\operatorname{div}_{\mu}w)\alpha_i \frac{d\nu}{d\mu} \right)\otimes \mu=\pounds_{w}\left(\frac{m}{\mu}\right)+(\operatorname{div}_{\mu}w)\frac{m}{\mu}.
$$
Therefore, 
$$
\frac{m_t}{\mu}=\frac{m_0}{\mu}+\int_0^t\left(\frac{1}{\mu}\frac{\sym{\delta} \ell}{\sym{\delta} a}(u_r,a_r)\diamond a_t-\pounds_{u_r}\left(\frac{m_r}{\mu}\right)-(\operatorname{div}_{\mu}u_r)\frac{m_r}{\mu}\right)\rmd r-\int_s^t\left(\pounds_{\xi}\left(\frac{m_r}{\mu}\right)+ (\operatorname{div}_{\mu}\xi)\frac{m_r}{\mu} \right)\rmd \bZ_r.
$$
Applying the Lemma \ref{lem:product_and_chain} and the identity 
$
\frac{m}{D_t}=\frac{1}{\frac{dD_t}{d\nu}}\frac{m}{\mu}=\frac{1}{\rho_t}\frac{m}{\mu},
$
we arrive at 
$$
\frac{m_t}{D_t}=\frac{m_s}{D_s} +\int_s^t\frac{1}{D_r}\left(\frac{\sym{\delta} \ell}{\sym{\delta} a}(u_t,a_t)\diamond a_t-\pounds_{u_r}\left(\frac{m_r}{D_r}\right)\right)\rmd r-\int_s^t\frac{1}{D_r}\pounds_{\xi}\left(\frac{m_r}{D_r}\right) \rmd \bZ_r.
$$
We then complete the proof by applying  Corollary \ref{corr:Reyn_trans}  with $\alpha={m}/{D}$.
\end{proof}

\subsection{Proof of the rough Hamilton--Pontryagin Theorem \ref{thm:HamPont}}\label{proof:Ham_Pont}
\begin{proof}
	
If   $(u, g, \boldsymbol{\lambda})\in \mathit{HP}_{\bZ}$ is a critical point of the action functional, 
then 
$$
0=\left. \frac{d}{d\epsilon} \right|_{\epsilon=0}S^{\mathit{HP}_{\bZ}}_{a_0}(u^{\epsilon}, \eta^{\epsilon}, \boldsymbol{\lambda}^{\epsilon})=I(\sym{\delta} u)+II(\sym{\delta}w)+III(\sym{\delta}\lambda),
$$
where 
\begin{align*}
I(\sym{\delta} u)&=\int_0^T \left\langle \frac{ \sym{\delta} \ell}{\sym{\delta} u}(u_t,a_t) -  \lambda_t, \sym{\delta} u_t \right\rangle_{\mathfrak{X}} \rmd t\\
II(\sym{\delta}w)&=\left. \frac{d}{d\epsilon} \right|_{\epsilon=0}\int_0^T\ell(u_t,(\eta_t^{\epsilon})_*a_0)\rmd t+   \left. \frac{d}{d\epsilon} \right|_{\epsilon=0}\int_0^T\langle \lambda_t, \rmd \eta^{\epsilon,-1}_t\eta^{\epsilon} _t\rangle_{\mfk{X}} \\
III(\sym{\delta}\lambda)&=\int_0^T\langle \sym{\delta}\lambda_t,\rmd \eta_t\circ \eta^{-1}_t\rangle_{\mfk{X}}-\int_0^T \langle \sym{\delta}\lambda_t,u_t\rangle_{\mfk{X}}\rmd t - \int_0^T \langle \sym{\delta}\lambda_t,\xi \rangle_{\mfk{X}}\rmd \bZ_t.
\end{align*}	
By virtue of the fundamental lemma of calculus of variations, $I(\sym{\delta}u)=0$ implies $m=\frac{\sym{\delta} \ell}{\sym{\delta} u}(u,a)\overset{\mfk{X}^{\vee}}{=} \lambda$. Separating variations in time and space and applying Theorem \ref{thm:truly_rough}, from $III(\sym{\delta}\lambda)=0$, deduce that $v\equiv u$ and $\sigma \equiv \xi$.

We now focus on $II(\sym{\delta}w)=0$. 
By the equality of mixed derivatives (see, also, Lemma 3.1 in \cite{arnaudon2014stochastic}), we have
$$
\frac{\partial^2 \psi^{\epsilon}_t}{\partial t\partial \epsilon} =\frac{\partial^2 \psi^{\epsilon}_t}{\partial \epsilon \partial t} =  \partial_t\sym{\delta}w_t\circ \psi^{\epsilon}_t +  \epsilon \frac{\partial}{\partial \epsilon}\left[\partial_t\sym{\delta}w_t\circ \psi^{\epsilon}_t\right], \quad \forall  (\epsilon, t)\in [-1,1]\times [0,T].
$$
Using the above  relation and that $\psi^0_tX=X$, we find $\frac{\partial \psi^{\epsilon}}{\partial \epsilon} \big|_{\epsilon=0}=\sym{\delta}w,$ and hence
$$
\left. \frac{d}{d\epsilon} \right|_{\epsilon=0}v^{\epsilon}_t=\left. \frac{d}{d\epsilon} \right|_{\epsilon=0}(\psi^{\epsilon}_t)_*v= -[\sym{\delta}w_t, v_t]=\operatorname{ad}_{\sym{\delta}w_t}v_t.
$$
Therefore, 
$$
\left. \frac{d}{d\epsilon} \right|_{\epsilon=0}\int_0^T\langle \lambda_t, \rmd \eta^{\epsilon,-1}_t\eta^{\epsilon} _t\rangle_{\mfk{X}} =\int_0^T\langle \lambda_t, \partial_t\sym{\delta}w_t+\operatorname{ad}_{\sym{\delta}w_t}v_t\rangle_{\mfk{X}} \rmd t+ \int_0^T\langle \lambda_t, \operatorname{ad}_{\sym{\delta}w_t}\sigma_t\rangle_{\mfk{X}} \rmd \bZ_t,
$$
where we have exchanged the order of $\frac{d}{d\epsilon}$ and the rough integral using Theorem \ref{thm:rough_int}. Moreover,  
$$
\left. \frac{d}{d\epsilon} \right|_{\epsilon=0}(\eta^{\epsilon}_t)_*a_0=\left. \frac{d}{d\epsilon} \right|_{\epsilon=0}(\psi_t^{\epsilon})_*a_t=-\pounds_{\sym{\delta} w_t}a_t, \;\; \forall t\in [0,T],
$$
which implies that 
$$
\left. \frac{d}{d\epsilon} \right|_{\epsilon=0}\int_0^T\ell(u_t,(\eta_t^{\epsilon})_*a_0)\rmd t=\int_0^T \langle \frac{\sym{\delta}\ell}{\sym{\delta} a}(u_t,a_t), -\pounds_{\sym{\delta} w_t}a_t\rangle_{\mfk{A}} \rmd t= \int_0^T\langle \frac{\sym{\delta}\ell}{\sym{\delta} a}(u_t,a_t)\diamond a_t, \sym{\delta} w_t\rangle_{\mfk{A}} \rmd t.
$$
The proof is completed by splitting the variations  of $\sym{\delta}w$ in space and time and applying Lemma \ref{lem:fund_calc_var}.
\end{proof}

\subsection{Proof of the rough Euler--Poincar\'e Theorem \ref{EP-thm}}\label{proof:EP-thm}
\begin{proof}
Using the definitions of $\sym{\delta} u$ and $\sym{\delta} a$ in \eqref{epvariations}, integrating by parts, and taking the endpoint conditions $w_0=w_T=0$  into account, we find
\begin{align*}
\delta S^{\mathit{EP}_{\bZ}}&=
\int_{0}^{T}\langle 
\frac{ \sym{\delta} \ell}{\sym{\delta} u_t}(u_t,a_t)\,,
\sym{\delta} u_t \rangle_{\mfk{X}} dt+ \langle  
\frac{\sym{\delta} \ell}{\sym{\delta} a_t}(u_t,a_t), \sym{\delta} a_t\rangle_{\mfk{A}} 
\\  &=
\int_{0}^{T}  \langle 
\frac{\sym{\delta} \ell}{\sym{\delta} u_t}(u_t,a_t)\,, \partial_t \sym{\delta} w \rangle_{\mfk{X}} \rmd t+ \langle 
\frac{\sym{\delta} \ell}{\sym{\delta} u_t}(u_t,a_t)\,,  \operatorname{ad}_{\rmd x_t} \sym{\delta} w\rangle_{\mfk{X}} -
\langle  \frac{\sym{\delta} \ell }{\sym{\delta} a_t}(u_t,a_t)
\diamond a_t\,,
\sym{\delta} w \rangle_{\mfk{X}} \rmd t
\\ &=
\int_{0}^{T} \langle  -\, \rmd 
\left(\frac{\sym{\delta} \ell }{ \sym{\delta} u_t} \right) -
 \operatorname{ad}_{\rmd x_t}^{\ast}\frac{\sym{\delta} \ell}{\sym{\delta} u_t}\,, \sym{\delta} w
\rangle_{\mfk{X}} + \langle  \frac{\sym{\delta} \ell }{\sym{\delta} a_t}
\diamond a_t\,,
\sym{\delta} w \rangle_{\mfk{X}} \,dt 
\\ &=
\int_{0}^{T} \langle  - \,\rmd 
\left(\frac{ \sym{\delta} \ell }{ \sym{\delta} u_t}\right) -
 \pounds_{\rmd x_t}^{\ast}\frac{\sym{\delta} \ell}{\sym{\delta} u_t} +
\frac{\sym{\delta} \ell }{\sym{\delta} a_t} \diamond a_t\,dt \,, \sym{\delta} w
\rangle_{\mfk{X}} .
 \end{align*}
 We conclude with the corresponding momentum equation by splitting up variations in space and time and applying Lemma \ref{lem:fund_calc_var}.
In addition, the advection equation $\rmd a_t +\pounds_{{\rm d}x_t}a_t = 0$ follows from the push-forward relation $a_t=(\eta_t)_*a_0$ by the Lie chain rule in Theorem \ref{thm:Lie_chain}. These two results complete the proof of Theorem \ref{EP-thm}.
\end{proof}

\appendix

\section{Notation and required background} \label{App-notation_and_background}

\subsection{Geometric rough paths}

In this section, we will provide an overview of the theory of geometric rough paths. We invite the reader to consult Appendix \ref{sec:motivation} for a historical account motivating the use of rough paths and \cite{lyons2007differential, friz2010multidimensional, friz2014course, bailleul2014flow} for more thorough expositions.

Let $T>0$,  $\Delta_T^2=\{(s,t)\in [0,T]^2:   s\le t\}$ and $\Delta_T^3=\{(s,\theta,t)\in [0,T]^3: s\le \theta \le t\}$. Let $E$ denote an arbitrary Fr\'echet space $E$ with family of seminorms $\bclP$. Elements of family of seminorms $\bclP$ will be denoted by $p$. For a given $\alpha\in (0,1]$, let $C^{\alpha}_T(E)$ denote the  space of H\"older continuous paths; in particular, $C^1_T(E)$ is the space of Lipschitz paths. Moreover, for a given $m\in \{2,3\}$\footnote{We only need $m=2,3$ because we consider only rough paths with H\"older regularity  $\alpha\in \left(\frac{1}{3},\frac{1}{2}\right]$.} and $\alpha \in \bbR_+$, denote by $C^{\alpha}_{m,T}(E)$ the space of functions that satisfy
$$
[\Xi]_{\alpha,p} = \underset{t_1\ne t_m}{\sup_{(t_1,\cdots,t_m)\in \Delta_T^m}} \frac{ p\left(\Xi_{t_1, \ldots, t_m}\right)}{|t_m-t_1|^{\alpha}}<\infty, \;\; p\in \bclP.
$$
Define  $\delta: C^{\alpha}_T(E)\rightarrow    C^{\alpha}_{2,T}(E)$ by 
$
\delta f=_{st} := f_t - f_s
$
for $f\in C^{\alpha}_T(E)$
and   $\delta_2: C^{\alpha}_{2,T}(E)\rightarrow   C^{\alpha}_{3,T}(E) $ by 
$$
\delta_2 \Xi_{s\theta t} := \Xi_{st} - \Xi_{s\theta } - \Xi_{\theta t}, \;\; (s,\theta , t)\in \Delta_{T}^{3}, \;\; \Xi\in C^{\alpha}_{2,T}(E).
$$
It follows that $\delta_2\circ \delta: C_T(E)\rightarrow  C_{3,T}(E)$ is the zero operator. 

For a given  $\Xi\in C^{\alpha}_{2,T}(E)$, $\beta\in \bbR_+$, and $p\in \bclP$, the quantity $[\delta_2\Xi]_{\beta,p}$, defined above, may be regarded as a measure of the extent to which $\Xi$  is an increment $\delta f$ for some $f\in C^{\alpha}_T(E)$.  The following lemma, proved in \cite{hocquet2018energy}[Proposition A.1],  is referred to as the \emph{sewing lemma}. The lemma says that if $\beta>1$, one can construct a ``unique'' $f\in C_T^{\alpha}(E)$ such that $\Xi$ is close to $\delta f$ in $C^{\beta}_{2,T}(E)$ by \eqref{eq:sewing}. Denote by 

\begin{lemma}[Sewing Lemma]\label{lem:sewing}
	There exists a  unique continuous linear map  $\clI : C^{\alpha,\beta}_{2,T}(E)\rightarrow C^{\alpha}_T(E)$ satisfying $\clI \Xi_0=0$ and
	$
	[\delta \clI \Xi - \Xi]_{\beta}\lesssim_{\beta} [\delta_2 \Xi]_{\beta,p}
	$
	for all $\Xi\in C^{\alpha,\beta}_{2,T}(E)$ and $p\in \bclP$. 
	More explicitly, for a given $(s,t)\in \Delta_T^2$,
	\begin{equation}\label{eq:sewing}
	\delta(\clI \Xi)_{st} = \lim_{|\clP([s,t])|\rightarrow 0} \sum_{[t_i,t_{i+1}]\in \clP([s,t])} \Xi_{t_it_{i+1}},
	\end{equation}
	where $\clP([s,t])$ denotes a finite partition of the interval $[s,t]$, $|\clP([s,t])|$ denotes its mesh size, and the limit is understood as a limit of nets (with the directed set of partitions partially ordered by inclusion).
\end{lemma}
\begin{remark}
Notice that if $\tilde{\Xi}\in C^{\alpha,\beta}_{2,T}(E)$ and $\Xi-\tilde{\Xi}\in C^{\beta}_{2,T}(E)$, then $\clI(\Xi)=\clI(\tilde{\Xi})$.
\end{remark}

For a given Fr\'echet space $E$ and $K\in \bbN$, let $E^K$ denote the direct sum of $E$ with itself $K$-times. By virtue of the Sewing Lemma, one can construct an integral of $Y\in C^{\beta}_T(E^K)$  against $Z\in C^{\alpha}_T(\bbR^K)$   if $\alpha+\beta>1$ by letting $\Xi_{st}=Y_s \delta Z_{st}=\sum_{k=1}^KY_s^k\delta Z^k_{st}$ for all$(s,t)\in \Delta_T^2$ and defining
$$
\int_0^{t}Y_r\rmd Z_r=\clI(\Xi)_{t}, \;\; t\in [0,T].
$$
This integral construction coincides with the integral that L.C.\ Young  \cite{young1936inequality} constructed.
In particular, for  $Z\in C^{\alpha}_T(\bbR^K)$ with $\alpha\in \left(\frac{1}{2},\infty\right)$, we may define $\bbZ\in C^2_{2,T}(\bbR^{K\times K})$ by
$$
\bbZ_{st}=\int_s^t\int_s^{t_2}\rmd Z_{t_1}\otimes \rmd Z_{t_2}=\int_s^t\delta Z_{st_2}\otimes \rmd Z_{t_2}, \quad (s,t)\in \Delta_T,
$$
where we have used the $\delta$ notation defined above in the second equality.
One can easily verify that $(Z,\bbZ)\in C^{\alpha}_T(\bbR^K)\times C^{2\alpha}_T(\bbR^{K\times K})$ satisfies
\begin{equation}\label{eq:Chen_rel}
\delta_2\bbZ_{st}= \delta Z_{s\theta } \otimes \delta Z_{\theta t},  \;\;\forall (s,\theta ,t)\in \Delta_T^3
\end{equation}
and 
\begin{equation}\label{eq:geom_main}
\operatorname{Sym}(\bbZ_{st}) = \frac{1}{2}\delta Z_{st}\otimes \delta Z_{st} \;\; \forall (s,t)\in \Delta_T^2.
\end{equation}

The condition $\eqref{eq:geom_main}$  is a geometric property which encodes the usual chain and product rules, upon which our variational theory is based.
Paths  $Z\in C^{\alpha}_T(\bbR^K)$ with  $\alpha\in (\frac{1}{2},1]$ are  referred to as Young paths.  Young paths are   distinguished from rough paths $\bZ=(Z,\bbZ)\in C^{\alpha}_T(\bbR^K)\times C^{\alpha}_T(\bbR^{K\times K})$, $\alpha\in (\frac{1}{3},\frac{1}{2}]$, which are defined to be paths  such that an a priori postulated  two-parameter path $\bbZ\in C_T^{2\alpha}(\bbR^{K\times K})$ satisfies \eqref{eq:Chen_rel}. A subclass of rough paths are the geometric rough paths, for which a classical calculus can be developed. In particular, \eqref{eq:geom_main} holds.  For a few more words of motivation about rough paths see Appendix \ref{sec:motivation}.

\begin{definition}
For a given $\alpha\in \left(\frac{1}{3},\frac{1}{2}\right]$, define the set $\bclC_{g,T}^{\alpha}(\bbR^K)$ of \textit{geometric $K$-dimensional $\alpha$-H\"older rough paths on the interval $[0,T]$} to be the closure of 
$$
\left\{(Z,\bbZ)\in C^{1}_T(\bbR^K) \oplus  C_{2,T}^{1}(\bbR^{K\times K}): \bbZ=\iint \rmd Z\otimes \rmd Z\right\}
$$
in $C^{\alpha}_T(\bbR^K) \oplus  C_{2,T}^{\alpha}(\bbR^{K\times K})$ with respect to the metric
$$
\rho(\bZ^{(1)}, \bZ^{(2)})=[Z^{(1)}- Z^{(2)}]_{\alpha} + [\bbZ^{(1)} - \bbZ^{(2)}]_{2\alpha}.
$$
It follows that both \eqref{eq:Chen_rel} and \eqref{eq:geom_main} hold for all $\bZ=(Z,\bbZ)\in \bclC_{g,T}^{\alpha}(\bbR^K)$ by a limiting argument. For a given $\alpha\in \left(\frac{1}{2},1\right]$, we denote $\bclC_{g,T}^{\alpha}(\bbR^K)= C^{\alpha}_T(\bbR^K)$. 
\end{definition}
\begin{remark}
To have a uniform notation for all $\alpha\in \left(\frac{1}{3},1\right]$, we write $\bZ=Z\in  \bclC_{g,T}^{\alpha}(\bbR^K)$ if $\alpha\in \left(\frac{1}{2},1\right].$
\end{remark}

It is possible to consider infinite-dimensional geometric rough paths, but for simplicity we restrict ourselves to finite-dimensional paths. However, we consider controlled rough paths (defined in the next section) in Fr\'echet spaces. We also remark that our theory can be extended to more irregular paths $\alpha<\frac{1}{3}$, but this requires higher-order iterated integrals and  more cumbersome notation. 

There are a large class of Gaussian processes that belong to $\bclC_{g,T}^{\alpha}(\bbR^K)$ for $\alpha\in \left(\frac{1}{3},\frac{1}{2}\right]$. We refer the reader to Appendix \ref{sec:Gaussian_rough_paths} for a slightly more in-depth discussion of Gaussian rough paths. The present discussion will be brief.

\begin{example}[Stratonovich Brownian motion]
Consider a Brownian motion $B:\Omega\times [0,T]\rightarrow \bbR^K$ on a complete probability space  $(\Omega, \clF, \bbP)$ and let $\bbB$  be the stochastic iterated integral constructed from Stratonovich integration theory. By virtue of the Kolmogorov continuity theorem,  one can find a event $\tilde{\Omega}\in \clF$ with $\bbP(\tilde{\Omega})=1$ such that $\bB(\omega)=(B(\omega),\bbB(\omega)\in \bclC^{\alpha}_{g,T}(\bbR^K)$ for any $\alpha \in \left(\frac{1}{3},\frac{1}{2}\right)$. This the Stratonovich lift of Brownian motion. Indeed, the Stratonovich integral is a limit of integrals of piecewise-linear approximations of Brownian motion.
\end{example}
\begin{example}[Gaussian rough paths]
	More broadly, a Gaussian process $Z: \Omega\times [0,T]\rightarrow \bbR^K$ can be lifted to a geometric rough path $\bZ(\omega)=(Z(\omega),\bbZ(\omega)\in \bclC^{\alpha}_{g,T}(\bbR^K)$ for $\alpha \in \left(\frac{1}{3},\tilde{\alpha}\right]$, provided the correlation in time of the process is fast enough depending on $\tilde{\alpha}$ (see Appendix \ref{sec:Gaussian_rough_paths}). In particular, fractional Brownian motion $B^H:\Omega \times [0,T]\rightarrow \bbR^K$ can be lifted to a strong geometric rough path $\bB^H(\omega)=(B^H(\omega),\bbB^H(\omega)\in \bclC^{\alpha}_{g,T}(\bbR^K)$ for all $\alpha \in \left(\frac{1}{3},\frac{1}{4H}\right)$ for all $\omega$ in some event of probability one.
\end{example}

\subsubsection{Controlled rough paths and integration}\label{Gubinelli-deriv}
Let us first describe the integration theory for paths $\bZ=(Z,\bbZ)\in \bclC^{\alpha}_{g,T}(\bbR^K)$ such that $\alpha \in \left(\frac{1}{3},\frac{1}{2}\right]$. 

\begin{definition}[Controlled rough path]
	We say a path $Y\in  C^{\alpha}_T(E)$ is \emph{controlled} by $Z$, if there exists a $Y'\in C^{\alpha}_T(E^{K})$ such that $R^Y: \Delta_T^2\rightarrow W$ defined by 
	\begin{equation}\label{eq:control_main}
	R^{Y}_{st} = \delta Y_{st} -Y_s'\delta Z_{st}=\delta Y_{st}-\sum_{k=1}^KY_s'^{k}\delta Z^k_{st}, \;\;(s,t)\in \Delta_T^2,
	\end{equation} 
	satisfies $R^{Y}\in C_{2,T}^{2\alpha}(E)$. For $\alpha\in \left(\frac{1}{3}, \frac{1}{2}\right]$, we define the linear space  $\bclD_{Z,T}(E)$ of controlled rough paths to be those pairs $\bY = (Y, Y')\in C^{\alpha}_T(E)\oplus C^{\alpha}_T(E^K)$ such that $R^Y\in C_{2,T}^{2\alpha}(E)$. The function $Y'$ is referred to as the \textit{Gubinelli derivative} \cite{gubinelli2004controlling}. The space $\bclD_{Z,T}(E)$ is a Fr\'echet space  with seminorms
	\begin{equation}\label{def:controlled_norm}
	|\bY|_{\bZ,p} =|Y_0|_p+|Y'_0|_p+[Y']_{\alpha,p} + [R^Y]_{2\alpha,p}, \quad p\in \clP.
	\end{equation}
\end{definition}

We note that any $Y\in C^{2\alpha}_T(E)$ satisfies \eqref{eq:control_main} with $Y'\equiv 0$. Moreover, $\bZ$  itself is controlled with $Z'\equiv \operatorname{id}$. However, the additional structure provided by $Y'$ is natural in the context of rough differential equations (see Remark \ref{rem:soln_RDE_controlled}). It is worth mentioning that $Y'$ is not uniquely specified unless the path is truly rough  (see Definition \ref{def:truly_rough} and Theorem \ref{thm:truly_rough} below). The integration of controlled rough paths is an immediate consequence of Lemma \ref{lem:sewing}

\begin{theorem} \label{thm:rough_int}
Let $\bZ=(Z,\bbZ)\in \bclC_{g,T}^{\alpha}(\bbR^K)$ for a given   $\alpha\in \left(\frac{1}{3}, \frac{1}{2}\right]$.  There exists a  linear continuous map $\bI_{\bZ}: \bclD_{Z,T}(E^K)\rightarrow \bclD_{Z,T}(E)$ defined by $\bI_{\bZ}(\bY)=(\clI(\Xi),Y)$, where 
$$
\Xi_{st}=Y_s \delta Z_{st}  +Y_s'\bbZ_{st}=\sum_{k=1}^KY_s^k\delta Z^k_{st} + \sum_{k=1}^KY_s'^{kl}\bbZ^{lk}, \;\;(s,t)\in \Delta_T^2,
$$
and $\clI$ is as in Lemma \ref{lem:sewing}. We write $$\int_0^{t} Y_r\rmd \bZ_r=\clI(\Xi)\in C^{\alpha}_T(E), \;\; t\in [0,T].$$
\end{theorem}

\begin{remark}[Integral of controlled path against a controlled path]\label{rem:int_of_control_control}
Let $F,G$  denote a Fr\'echet  spaces and  $B: F\times E\rightarrow G$ be continuous  and bilinear. 
For  $\bY=(Y,Y')\in \bclD_{Z,T}(E)$ and $\bX=(X,X')\in \bclD_{Z,T}(F)$,  we  define
\begin{equation}\label{def:int_cont_cont}
\int_{s}^t B(X_r,\rmd \bY_r)=(\delta \clI \Xi)_{st},\quad \textnormal{where} \quad 
\Xi_{st}=B(X_s,\delta Y_{st})+B(X'_s, Y'_s)\bbZ_{st}, \quad (s,t)\in \Delta_{T}^2,
\end{equation}
Indeed, for all $(s,\theta, t)\in \Delta_T^3$,
$$
\delta_2 \Xi_{s\theta t}=-B(R^X_{s\theta },R^Y_{\theta t}) -B(R^X_{s\theta },Y'_{\theta})\delta Z_{\theta t}-B(X'_s,R^Y_{\theta t})\delta Z_{s\theta} -B(X'_s,\delta Y'_{s\theta })\delta Z_{s\theta }\otimes \delta Z_{\theta t}+(B(X'_s,Y'_s)-B(X'_\theta,Y'_\theta))\bbZ_{\theta t},
$$
which implies in $ \Xi\in C^{3\alpha}_{2,T}(G)$, so that we may apply Lemma \ref{lem:sewing}.  Notice that if $Y\in C_T^{2\alpha}(E)$ and $Y'\equiv 0$, then \eqref{def:int_cont_cont} agrees with the Young integral. This definition is used in the Clebsch variational principle in order to define the integral of the Lagrange multiplier against an advected quantity (see Remark \ref{rem:Clebsch_contraint}).
\end{remark}

For Young paths, the extra structure provided by the Gubinelli derivatives is not needed.

\begin{definition}[Controlled paths in the Young case]
We define $\bclD_{Z,T}(E)= C_T^{\alpha}(E)$ if $\alpha\in   \left(\frac{1}{2},1\right]$.
\end{definition}
\begin{remark}
To have a uniform notation for all $\alpha\in \left(\frac{1}{3},1\right]$, we  write   $\bY=Y\in \bclD_{Z,T}(E)$  if $\alpha\in \left(\frac{1}{2}, 1\right]$. We also remark that obviously the controlled space does not depend on $Z$ in this case.
\end{remark}

\subsubsection{The rough chain and product  rule}

Let $E$ and $F$ be Fr\'echet spaces and $C(E;F)$ denote the space of continuous maps. Let $C^1_b(E;F)$ denote the space of bounded functions $\Phi: E\rightarrow F$ such that the limit
$$
D\Phi(e)h=\lim_{\epsilon\rightarrow 0}\frac{\Phi(e+th)-\Phi(e)}{t}
$$
exists for all $e,h\in E$ and  $D\Phi: E\times E\rightarrow F$ is continuous and bounded. We  define  $C^{m}_b(E;F)$ for $m\ge 2$ analogously (see \cite{hamilton1982inverse}[Def.\ 3.1.1\; \& \;Section I.3.6]).  Let $N_{\alpha}=0$ if $\alpha=1$, $N_{\alpha}=1$ if $\alpha\in \left(\frac{1}{2},1\right)$ and $N_{\alpha}=2$ if $\alpha\in \left(\frac{1}{3},\frac{1}{2}\right)$. The following lemma says that controlled rough paths are stable under composition and products. Their proof can be found in Lemma 7.3 and Corollary 7.4 of \cite{friz2014course}.

\begin{lemma}\label{lem:controlled_chain_product}

\

\begin{enumerate}[(i)]
    \item If  $\bY=(Y,Y')\in \bclD_{Z,T}(E)$ and $\Phi \in C_T^{1}(C^{N_{\alpha} }_b(E;F))$, then $\Phi(\bY) = (\phi(Y), D\phi(Y)Y')\in \bclD_{Z,T}(F)$. 
    \item Let $B: F\times E\rightarrow G$ be continuous  and bilinear. If $\bX=(X,X')\in \bclD_{Z,T}(F)$ and $\bY=(Y,Y')\in \bclD_{Z,T}(E)$, then $B(\bX,\bY)=(B(X,Y),B(X',Y)+B(X,Y'))\in \bclD_{Z,T}(G)$.   
\end{enumerate}
\end{lemma}

In order to construct the integration theory given above, we have actually not needed the geometric nature of the path (i.e., \eqref{eq:geom_main}). However, to obtain an extension of the ordinary chain and product rule, we require \eqref{eq:geom_main} to hold  (see \cite{friz2014course}[Section 7.5]).

\begin{lemma}\label{lem:product_and_chain}

\

\begin{enumerate}[(i)]
\item For a given, $Y_0\in E$, $\beta\in C_T(E)$, and $(\sigma,\sigma')\in \bclD_{Z,T}(E^K)$, let 
\begin{equation}\label{eq:doubly_controlled}
Y_t=Y_0+\int_0^{t}\beta_r\rmd r+\int_0^t \sigma_r \rmd \bZ_r, \;\; t\in [0,T].
\end{equation}
If $\Phi \in C_T^{1}(C^{N_{\alpha} +1}_b(E;F))$, then  for all $t\in [0,T]$
$$
\Phi_t(Y_t)=\Phi_0(Y_0)+\int_0^{t}\left(\partial_t\Phi_r(Y_r)+D\Phi_r(Y_r)\right)\beta_r \rmd r+\int_0^{t}D\Phi_r(Y_r)\sigma_r \rmd \bZ_r.
$$	
\item For a given, $X_0\in F$, $\tilde{\beta}\in C_T(F)$, and $(\tilde{\sigma},\tilde{\sigma}')\in \bclD_{Z,T}(F^K)$, let 
$$
X_t=X_0+\int_0^{t}\tilde{\beta}_r\rmd r+\int_0^t \tilde{\sigma}_r \rmd \bZ_r, \;\; t\in [0,T].
$$
and $Y$ be as specified in (i). Let $B: F\times E\rightarrow G$ be continuous and  bilinear.  Then for all $t\in [0,T]$,
$$
B(X_t,Y_t)=B(X_0,Y_0)+\int_0^t\left(B(\tilde{\beta}_r,Y_r)+B(X_r,\beta_r)\right)\rmd r + \int_0^t\left(B(\tilde{\sigma}_r,Y_r)+B(X_r,\sigma_r)\right)\rmd \bZ_r.
$$
\end{enumerate}
\end{lemma}

In Section \ref{sec:HP_var_princ}, we need the decomposition of paths $Y$ satisfying the relation \eqref{eq:doubly_controlled} to be unique. A  decomposition of a path $Y$ satisfying \eqref{eq:doubly_controlled} is unique if the rough path $\bZ$ is \emph{truly rough} (Theorem 6.5 of \cite{friz2014course}). Examples of truly rough paths include  fractional Brownian motion $B^H$ with $H\in \left(\frac{1}{3},\frac{1}{2}\right]$.

\begin{definition}[Truly rough path]\label{def:truly_rough}
Let $\alpha \in \left(\frac{1}{3},\frac{1}{2}\right]$ and $\bZ\in \bclC_{g,T}^{\alpha}(\bbR^K)$.  We say $\bZ$ is \emph{truly rough} if for all $s$ in a dense set in $[0,T]$,   
$$
\limsup_{t\downarrow s}\frac{|\delta Z_{st} |}{|t-s|^{2\alpha}}=\infty.
$$
\end{definition}
\begin{theorem}\label{thm:truly_rough}
If $\bZ$ is  truly rough and 
$$
Y_t=Y_0+\int_0^{t}\beta_r\rmd r+\int_0^t \sigma_r \rmd \bZ_r=\tilde{Y}_0+\int_0^{t}\tilde{\beta}_r\rmd r+\int_0^t \tilde{\sigma}_r \rmd \bZ_r, \;\; \forall t\in [0,T],
$$
where $Y_0,\tilde{Y}_0\in E$, $\beta, \tilde{\beta}\in C_T(E)$, and $(\sigma,\sigma'),(\tilde{\sigma},\tilde{\sigma}')\in \bclD_{Z,T}(E^K)$, then $\beta \equiv \tilde{\beta}$ and $(\sigma, \sigma')\equiv (\tilde{\sigma}, \tilde{\sigma}')$.
\end{theorem}

\subsubsection{Solutions of  rough differential equations (RDEs)}

We will now introduce the definition of solution to an RDE. Let $V$ denote a Banach space.

\begin{definition}
	Let $u\in C_T(C_b(V;V))$  and  $ \xi \in C^{1}_T(C_b^{N_{\alpha} }(V;V)^K)$. We say $Y$ is a solution of 
	\begin{equation}\label{eq:RDE_diff}
	\textnormal{d}Y_t = u_t(Y_t)\textnormal{d}t + \xi_t(Y_t)\textnormal{d}\bZ_t, \;\; Y_0=v\in V,
	\end{equation}
	on the interval $[0,T]$, if $\bY=(Y,\xi(Y))\in \bclD_{Z,T}(V)$ and  
	\begin{equation}\label{eq:RDE_int}
	Y_t= v+\int_0^{t} u_r(Y_r)\textnormal{d}r + \int_0^{t} \xi_r(Y_r)\textnormal{d}\bZ_r, \quad \forall t\in [0,T].
	\end{equation}
\end{definition}
\begin{remark}\label{rem:soln_RDE_controlled}
	The rough integral in  \eqref{eq:RDE_int} is well-defined by virtue of Lemma \ref{lem:controlled_chain_product}.
\end{remark}
The following lemma concerns equivalent notions of solutions. Its proof is a direct  application of Theorem \ref{thm:rough_int} and \ref{lem:product_and_chain}. The first formulation is referred to as the Davie's formulation  \cite{davie2008differential} and the second  naturally extends to the manifold setting. 
\begin{lemma}\label{lem:soln_RDE}
$Y$ is a solution of \eqref{eq:RDE_diff} on the interval $[0,T]$ if and only if
\begin{enumerate}[(i)]
\item
$
Y^{\natural}_{st}:=\delta Y_{s t}-\int_{s}^{t}u_{r}(Y_{r})\textnormal{dr}-\xi_s(Y_s)(\delta Z_{st})-D\xi_s(Y_s ) \xi_s(Y_s)(\bbZ_{st}), \; (s,t)\in \Delta_T^2,
$
satisfies $Y^{\natural}\in C^{3\alpha}_{2,T}$;
\item 
$
f(Y_t)=f(v)+\int_0^t Df(Y_r)u_r(Y_r)\rmd r+\int_0^t Df(Y_r)\xi_r(Y_r)\rmd \bZ_r, \;\; \forall t\in [0,T], \;\; \forall f\in C_b^{\infty}(V;\bbR).
$
\end{enumerate}
\end{lemma}

The proof of existence and uniqueness for RDEs uses  Picard iteration in the controlled rough path topology (i.e., \ref{def:controlled_norm}). We refer the reader to, e.g., \cite{friz2014course}[Section 8.5] for a proof. Moreover, in Section \ref{sec:RDE_Euclidean} we given more details about flows on Euclidean spaces.
\begin{theorem}
	There exists a unique continuous  solution map
	\begin{align*}
	S: V \times  C_T(C_b^1(V;V))\times C_T^{1}(C_b^{N_{\alpha}+1 }(V;V)^K)\times \bclC_{g,T}(\bbR^K)&\rightarrow  \bclD_{Z,T}(V)\\
	(v,u,\xi,\bZ)&\mapsto (Y,\xi(Y)).
	\end{align*}
\end{theorem}

\subsection{Differential geometry}\label{sec:differential_geometry}

\subsubsection{Basic setting and the Lie derivative}
Let $M$ denote a smooth compact, connected, oriented  $d$-dimensional manifold without boundary.  For an arbitrarily given rank $p$ vector bundle $E$ over $M$, denote by $\Gamma_{C^\infty}(E)$ the space of smooth sections endowed with the Fr\'echet topology defined through a cover of total trivializations of $E$. Denote by $C^\infty=\Gamma_{C^\infty}(\bbR)$ the space of smooth  functions on $M$, $\mfk{X}_{C^\infty}=\Gamma_{C^\infty}(TM)$ the space of smooth vector fields on $M$, $\clT^{lk}_{C^\infty}=\Gamma_{C^\infty}(T^{lk}TM)$  the space of smooth ($l$-contravariant, $k$-covariant) tensor fields on $M$. Denote $\Omega^{k}_{C^\infty}=\Gamma_{C^\infty}(\Lambda^kT^*M)$ as  the space of smooth alternating $k$-forms on $M$. We let $ \operatorname{Dens}_{C^\infty}:=\Omega^d_{C^\infty}$. It is worth remarking that non-orientability and tensor-densities can be easily accommodated  by introducing weighted densities (see, e.g., \cite{spivak1970comprehensive, abraham2012manifolds, vanCraniac2017}). However, we avoid this extension for brevity in the presentation here. 

Denote the  wedge and tensor product  by $\wedge$ and $\otimes$, respectively.  Let  $\mathbf{d}:\Omega^k_{C^{\infty}} \rightarrow \Omega^{k+1}_{C^\infty}$ denote the exterior derivative operator and $\bi_{u}: \Omega^{k}_{C^\infty}\rightarrow \Omega^{k-1}_{C^\infty}$ denote the  interior product operator for an arbitrarily given $u\in\mathfrak{X}_{C^{\infty}}$.
For given  $F\in  \operatorname{Diff}_{C^{\infty}}$ and $\tau\in \clT^{lk}_{C^\infty}$, the push-forward and pull-back are defined by 
\begin{equation}\label{def:push-forward_pullback}
F_{*}\tau=(TF)_*\circ \tau  \circ F^{-1} \quad \textnormal{and} \quad F^*\tau=(F^{-1})_*\tau,
\end{equation}
respectively, where $TF \in C^{\infty}(TM;TM)$ is the tangent map of $F$, which extends to an isomorphism (on fibers) $(TF)_*\in C^{\infty}(T^{lk}TM;T^{lk}TM)$.

For a given time-dependent vector field $u\in C^{\infty}(\bbR\times M;TM)$, let $\eta: \bbR^2\times M\rightarrow M$ denote  the two-parameter smooth flow of diffeomorphisms generated by $u$;\footnote{The time-dependent vector field $u$ may be associated with a time-independent vector field $\bar{u}\in  C^{\infty}(\bbR\times M,T\bbR\times TM)$ on the manifold $T\bbR\times TM$ via $\bar{u}_t(p)=\{u_t(p),1_t\}\in T_t\bbR\times T_pM$ for all $(t,p)\in \bbR\times M$. Thus, the two-parameter flow may be defined in terms of  the one-parameter flow of $\bar{u}$ by  $\eta_{t-s}(x,s)=\{\eta_{ts},t\}$. It follows from $\eta_{st}\circ \eta_{ts}={\rm id}$ that for all $s\in \bbR$
	$$
	\frac{d}{dt}\eta_{st}=-T\eta_{st}\circ u_t=-(\eta_{st})_*u_t\circ \eta_{st}.
	$$
	Equivalently,  for all $f\in C^{\infty}$, $h_{\cdot}=\eta_{\cdot s*}f=(\eta_{s\cdot })^*f\in C^m$ solves the PDE
	$
	\partial_t h_t+\pounds_{u_t}h_t=0.
	$} that is, for all $(s,t)\in \bbR^2$,
$
\eta_{t\theta}\circ \eta_{\theta s}=\eta_{ts},
$ and $\eta_{\cdot s}X$ is the unique integral curve
$$
\dot{\eta}_{ts}X=u_t(\eta_{ts}X), \quad \eta_{ss}X=X\in M.
$$
For  given $u\in C^\infty(\bbR\times M; TM)$ and $t\in \bbR$,  the  Lie derivative $\pounds_{u_t}: \clT^{lk}_{C^\infty}\rightarrow \clT^{lk}_{C^\infty}$ is defined  by 
$$
\pounds_{u_t}\tau=\left. \frac{d}{d\theta} \right|_{\theta=t}(\eta_{\theta t})^*\tau \quad \Leftrightarrow  \quad \frac{d}{dt}\eta_{ts}\tau = \eta_{ts}\pounds_{u_t}\tau.
$$
If $u$ is independent of time, we define
$
\pounds_{u}\tau=\left. \frac{d}{dt} \right|_{t=0}\eta_{t}^*\tau,
$
where $\eta_{t}=\eta_{t0}$ is the corresponding one-parameter flow map.
It is well-known that  the Lie derivative (see, e.g., \cite{holm2009geometric, abraham2012manifolds}) is the unique operator on the tensor algebra $\oplus_{l,k}\clT^{lk}_{C^\infty}$ that i) commutes with tensor contractions, ii) is natural with respect to restrictions, and iii) satisfies for a given local chart $(U,\phi)$ and all $f\in C^{\infty}|_{U}$, $u,v\in \mfk{X}_{C^{\infty}}|_{U}$,
$$
\pounds_{u}f=u[f]=u^i\partial_{x^i}f=\bi_{u}\bd f \quad  \textnormal{and} \quad 
(\pounds_{u} v)^i =u^j\partial_{x^j}v^i - v^j\partial_{x^j}u^i.
$$
It follows that for all $u\in \mfk{X}_{C^\infty}|_{U}$, $\alpha \in \Omega^k_{C^{\infty}}|_{U},$ and $\tau\in \clT^{lk}_{C^{\infty}}|_U$,
\begin{equation}
\begin{aligned}\label{eq:Lie_derivative_local}
(\pounds_{u}\alpha)_{i_1\cdots i_k}&= u^j\partial_{x^j}\alpha_{i_1\cdots i_k} + \alpha_{j\cdots i_k}\partial_{x^{i_1}}u^j +\cdots +\alpha_{i_1\cdots j}\partial_{x^{i_k}}u^j\\
(\pounds_{u}\tau)^{j_1\cdots j_l}_{i_1\cdots i_k}&= u^j_t\partial_{x^j}\tau_{i_1\cdots i_k}^{j_1\cdots j_l}-
\tau_{i_1\cdots i_k}^{j_1\cdots j_l}\partial_{x^{j_1}}u^j -\cdots -\tau_{i_1\cdots i_k}^{j_1\cdots j}\partial_{x^{j_l}}u^j+ \tau_{j\cdots i_k}^{j_1\cdots j_k}\partial_{x^{i_1}}u^j +\cdots +\tau_{i_1\cdots j}^{j_1\cdots j_l}\partial_{x^{i_k}}u^j.
\end{aligned}
\end{equation}
Thus, for given $u\in C^{\infty}(\bbR\times M,TM)$, the Lie derivative $\pounds_{u_t}$ is a first-order differential operator on the bundles $\Lambda^kT^*M$ and $T^{lk}TM$. For non-vanishing $\mu \in \Omega^d_{C^\infty}$, the operator $\operatorname{div}_{\mu}: \mfk{X}_{C^\infty}\rightarrow C^\infty$ is defined  by the relation
$$
\pounds_{u}\mu=(\operatorname{div}_{\mu}u) \mu.
$$
For $u,v\in \mfk{X}_{C^\infty}$ we let $[u,v]=\pounds_{u}v$ and $\operatorname{ad}_{u}v:=-\pounds_{u}v$ and note that 
$$
\left(\pounds_{u}\pounds_{v}-\pounds_{v}\pounds_{u}\right)\tau=\pounds_{[u,v]}\tau=-\pounds_{\operatorname{ad}_{u}v}\tau, \quad \forall \tau\in \clT^{lk}_{C^{\infty}}.
$$
Moreover, for all $u\in  \mfk{X}_{C^{\infty}}$ and $ \alpha\in \ \Omega^k_{C^{\infty}}$, we have  
\begin{equation}\label{eq:Cartan}
\pounds_{u} \alpha=\mathbf{d}(\mathbf{i}_{u} \alpha) + \mathbf{i}_{u}\mathbf{d}\alpha\,,
\end{equation}
which is referred to as \emph{Cartan's formula}.

\subsubsection{Vector bundles: canonical pairings,  adjoints, and function spaces}\label{sec:duals_of_bundles}
For a given vector bundle $E$, denote by $E^*$ the dual bundle. Let $E^{\vee}=E^*\otimes\Lambda^d T^*M$ and we may extend the dual pairing between $E$ and $E^*$ to a mapping $\langle \cdot, \cdot \rangle_E : E^{\vee}\times E\rightarrow \Lambda^d T^*M$. The bundle $E^{\vee}$ is often called the functional dual bundle.  We may then define the canonical pairing  $\langle \cdot, \cdot \rangle_{\Gamma(E)}: \Gamma_{C^\infty}(E^{\vee})\times \Gamma_{C^\infty}(E)\rightarrow \bbR$ by
\begin{equation}\label{def:canonical_pairing}
\langle s', s\rangle_{\Gamma(E)} =\int_{M}\langle s', s\rangle_{E}, \qquad (s',s)\in \Gamma_{C^\infty}(E^{\vee})\times \Gamma_{C^\infty}(E).
\end{equation}
The quantity $\langle s', s\rangle_{E}\in \operatorname{Dens}_{C^\infty}$ in the integrand is a volume form and it is being integrated over the manifold $M$.
The distributional sections  of $E$ and $E^{\vee}$ are defined by  $\Gamma_{\clD'}(E):=\Gamma_{C^\infty}(E^{\vee})^*$ and $\Gamma_{\clD'}(E^{\vee}):=\Gamma_{C^\infty}(E)^*$, respectively. The canonical pairing \eqref{def:canonical_pairing} induces the following dense embeddings:
$$
\Gamma_{C^\infty}(E)\hookrightarrow \Gamma_{\clD'}(E) \;\; \textnormal{via} \;\;   s\mapsto l_{s}=\langle \cdot, s \rangle_{\Gamma(E)}\quad \textnormal{and} \quad 
\Gamma_{C^\infty}(E^{\vee})\hookrightarrow \Gamma_{\clD'}(E^{\vee}) \;\;\textnormal{via}\;\; 
\tilde{s}\mapsto l_{s'}=\langle s', \cdot \rangle_{\Gamma(E)}.
$$
The pairing and definitions of distributions are canonical in the sense that no metric or volume form are needed to define them. We extend the pairing  $\langle \cdot, \cdot\rangle_{\Gamma(E)} $ to $\Gamma_{\clD'}(E^{\vee})\times \Gamma_{C^\infty}(E)$ and  $ \Gamma_{C^\infty}(E^{\vee})\times \Gamma_{\clD'}(E)$ in the usual way. 

The adjoint of a linear differential operator $L:  \Gamma_{C^\infty}(E)\rightarrow  \Gamma_{C^\infty}(E)$, denoted by $L^*: \Gamma_{\clD'}(E^{\vee})\rightarrow \Gamma_{\clD'}(E^{\vee})$, is defined by
\begin{equation}\label{def:adjoint}
\langle L^*s',s\rangle_{\Gamma(E)}=\langle s',Ls\rangle_{\Gamma(E)}, \quad \forall (s',s)\in \Gamma_{\clD'}(E^{\vee})\times \Gamma_{C^\infty}(E).
\end{equation}
It follows that $L^*$ restricts to  $L^*: \Gamma_{C^\infty}(E^{\vee})\rightarrow \Gamma_{C^\infty}(E^{\vee})$, and we write $L=L^{**}:  \Gamma_{\clD'}(E)\rightarrow  \Gamma_{\clD'}(E)$. 

For a normal, local, and invariant Fr\'echet, Banach, or Hilbert function space $\clF$ on $\bbR^d$,\footnote{ Let $\clF$ denote a locally convex topological vector space  of functions $f:\bbR^d\rightarrow \bbR$ such that $C_c^{\infty}(\bbR^d)\hookrightarrow  \clF\hookrightarrow   \clD'(\bbR^d):=C^{\infty}_c(\bbR^d)^*$ and such that pointwise multiplication of functions in $\clF$ by functions in $C_c^{\infty}(\bbR^d)$ is a continuous operation. We say  that a function space $\clF$ is  normal if the embedding $C_c^{\infty}(\bbR^d)\hookrightarrow  \clF$ is dense,  local if $\clF=\{u\in \clD'(\bbR^d): \phi u\in \clF, \;\; \forall \phi \in C_c^{\infty}(\bbR^d)\}$, and invariant if any smooth diffeomorphism $\chi\in \operatorname{Diff}_{C^{\infty}}$ induces an topological isomorphism on $\clF$ via push-forward.} we define a Fr\'echet, Banach, or Hilbert, respectively, of sections $\Gamma_{\clF}(E)$  via a cover of total trivializations.\footnote{A total trivialization 
is a triple $(U,\phi, \psi)$ such that $(U,\phi)$ is a local chart of $M$ and $\psi: E_{U}\rightarrow U\times \bbR^{\operatorname{rank}(E)}$ is a trivialization of $E$ over $U$. Any such trivialization induces an isomorphism $h_{\phi,\psi}: \Gamma_{\clD'}(E|_{U})\rightarrow \Gamma_{\clD'}(\phi(U))$. A section $s\in \Gamma_{\clD'}(E)$ belongs to $\Gamma_{\clF}(E)$ if for every total trivialization $(U,\phi, \psi)$, $h_{\phi,\psi}(s|_{U})\in \clF(\phi(U))^{\operatorname{rank}(E)}$. } It follows that 
$$
\Gamma_{C^\infty}(E)\hookrightarrow  \Gamma_{\clF}(E)\hookrightarrow  \Gamma_{\clD'}(E),
$$ 
where the  embedding $\Gamma_{C^\infty}(E)\hookrightarrow  \Gamma_{\clF}(E)$ is dense. We refer the reader to \cite{vanCraniac2017}[Ch.\ 3] for more details. Exactly the same construction applies to obtain a function space  $\Gamma_{\clF}(E^{\vee})$:
$$
\Gamma_{C^\infty}(E^{\vee})\hookrightarrow  \Gamma_{\clF}(E^{\vee})\hookrightarrow  \Gamma_{\clD'}(E^{\vee}).
$$ 
In the present work, we assume that all  function spaces $\clF$ are normal, local, and invariant. In particular, we let $$\clF=\Gamma_{\clF}(\bbR),\;\; \mfk{X}_{\clF}=\Gamma_{\clF}(TM),\;\;\operatorname{Dens}_{\clF}^{\vee}=\clF,\;\; \clT^{lk}_{\clF}=\Gamma_{\clF}(T^{lk}TM),\;\;\Omega^{k}_{\clF}=\Gamma_{\clF}(\Lambda^kT^*M),$$  $$\clF^{\vee}=\Gamma_{\clF}(\Lambda^{d}T^*M)=\operatorname{Dens}_{\clF}, \;\; \mfk{X}_{\clF}^{\vee}=\Gamma_{\clF}(TM^{\vee}), \quad (\clT^{lk}_{\clF})^{\vee}=\Gamma_{\clF}((T^{lk}TM)^{\vee}), \quad (\Omega^{k}_{\clF})^{\vee}=\Gamma_{\clF}(\Lambda^kT^*M^{\vee}).$$

Any  strong bundle pseudo-metric  $(\cdot, \cdot)_E:E\times E\rightarrow \bbR$  induces an isomorphism $\flat : \Gamma_{\clF}(E)\rightarrow \Gamma_{\clF}(E^*)$ with inverse denoted $\sharp:   \Gamma_{\clF}(E^*)\rightarrow \Gamma_{\clF}(E)$ for an arbitrarily given function space $\clF$.  Moreover, a non-vanishing volume form $\mu \in\operatorname{Dens}_{C^\infty}$ induces an isomorphism
$\operatorname{id}\otimes \mu: \Gamma_{\clF}(E^*)\rightarrow \Gamma_{\clF}(E^{\vee})$ with inverse $\frac{1}{\mu}:\Gamma_{\clF}(E^{\vee})\rightarrow \Gamma_{\clF}(E^*)$. For every $s\in \Gamma_{\clF}(E^*)$, $$\operatorname{id}\otimes \mu(s)=s\otimes \mu.$$ To describe the inverse, note that  for  all densities  $\nu\in  \operatorname{Dens}_{C^{\infty}}$, there exists $\frac{d\nu}{d\mu}\in C^{\infty}$ such that $\nu=\frac{d\nu}{d\mu}\mu$. The inverse is induced by
\begin{equation}\label{eq:one_over_density}
\frac{1}{\mu}(s \otimes \nu)=s\frac{d\nu}{d\mu}, \quad s\otimes \mu \in \Gamma_{\clF}(E^{\vee}).
\end{equation}

Composing these isomorphisms, we obtain an isomorphism $\flat \otimes \mu: \Gamma_{\clF}(E)\rightarrow \Gamma_{\clF}(E^{\vee})$ with inverse $\frac{\sharp}{\mu}: \Gamma_{\clF}(E^{\vee})\rightarrow \Gamma_{\clF}(E)$.  In particular, we may define a pairing $(\cdot, \cdot)_{\Gamma_{L^2}(E)}: \Gamma_{L^2}(E)\times \Gamma_{L^2}(E)\rightarrow \bbR$ by 
$$
(s_1, s_2)_{\Gamma_{L^2}(E)}=\langle\flat \otimes \mu (s_1), s_2\rangle_{\Gamma(E)}=\int_{M}\langle s_1^{\flat},s_2\rangle_{E}\mu = \int_{M}(s_1,s_2)_{E}\mu, \quad s_1,s_2\in \Gamma_{C^\infty}(E),
$$
which may be extended to $\Gamma_{\clD'}(E)\times \Gamma_{C^\infty}(E)$ via the isomorphisms $(\flat\otimes \mu)^*:\Gamma_{\clD'}(E)\rightarrow  \Gamma_{\clD'}(E^{\vee})$. 

If $(\cdot,\cdot)_{E}$ is a metric, we obtain a  Hilbert structure on $\Gamma_{L^2}(E)$, the space of square-integrable equivalence classes of measurable sections. Furthermore, for every $s\in \bbR$, there exists an order $s$ elliptic operator $A$ satisfying $A: \Gamma_{W_2^{s}}(E)\rightarrow \Gamma_{W_2^{s-1}}(E)$ and $\Gamma_{W_2^{s}}(E)\cong A^{-1}\Gamma_{L^2}(E)$, where $W_2^{s}=(I-\Delta)^{-\frac{s}{2}}L^2$ denotes the Bessel-potential spaces, which provides a  Hilbert structure to $\Gamma_{W_2^{s}}(E)$ \cite{Hintz2019, melrose2008microlocal}. Moreover, if  $L: \Gamma_{C^\infty}(E)\rightarrow \Gamma_{C^\infty}(E)$, then 
$L^T_{\flat\otimes \mu}:=\frac{\sharp}{\mu}\circ L^T\circ \flat\otimes \mu :\Gamma_{C^\infty}(E)\rightarrow \Gamma_{C^\infty}(E)$ is the adjoint of $L$ relative to the pairing $( \cdot, \cdot)_{\Gamma_{L^2}(E)}$.

\subsubsection{Riemannian manifolds and the Hodge decomposition}\label{sec:Hodge}
Any Riemannian metric $g$ on $M$ gives rise to a volume form $\mu_g$ defined in a local coordinate chart $(U,\phi)$ by
$$
\mu_g=\sqrt{\operatorname{det}[g_{ij}]}\; dx^1\wedge \cdots\wedge  dx^d.
$$
Furthermore, the metric $g$ extends to bundle metrics on $T^{lk}TM$ and $\Lambda^kT^*M$ in the usual way. In particular, we obtain the diffeomorphisms $\flat, \operatorname{id}\otimes \mu_g$ and $\flat\otimes \mu_g$ discussed in the previous section. For every $k\in \{0,1,\ldots, d\}$, the metric and orientation gives rise to the inner product on $\Omega_{L^2}^k$ defined by
$$
(\alpha,\beta)_{\Omega^k_{L^2}}=\int_{M}g(\alpha,\beta)\mu_g= \int_{M}\alpha \wedge \star \beta, \quad \alpha,\beta\in \Omega_{L^2}^k,
$$
where we have used the Hodge-star diffeomorphism $\star: \Omega_{\clF}^k\rightarrow \Omega_{\clF}^{d-k}$ defined by
$$
\alpha \wedge \star \beta = g(\alpha, \beta)\mu_g, \quad  \forall \alpha,\beta\in \Omega_{\clF}^k.
$$ 
The adjoint of $\bd: \Omega^{k}_{C^\infty}\rightarrow \Omega^{k+1}_{C^\infty}$ with respect to the $L^2$-pairing is given by $\bd^* := (-1)^{dk+1}\star \mathbf{d}\star: \Omega^{k+1}_{\clD'}\rightarrow \Omega^{k}_{\clD'}$.

The Hodge decomposition plays an essential role in incompressible fluids on manifolds. We now briefly describe the decomposition and the canonical pairing we use in the incompressible case.

Let $\Delta_H = \mathbf{d}\bd^* +\bd^* \mathbf{d}: \oplus_{k=0}^{d}\Omega_{\clD'}^{k}\rightarrow \oplus_{k=0}^d\Omega_{\clD'}^{k}$ denote the Hodge Laplacian, which is formally self-adjoint and non-negative with respect to the inner product $\sum_{k=0}^d(\cdot, \cdot)_{\Omega^k}$. Let  
$$\mathcal{H}^k_{\Delta}=\big\{\alpha\in \Omega^k_{C^\infty} : \Delta_H \alpha=0\big\}=\{\alpha\in \Omega^k_{C^\infty} : \bd \alpha=\delta\alpha=0\}
$$ 
denote the finite-dimensional space of harmonic $k$-forms. It follows that harmonic $0$-forms are constant.

Let $\clF$ denote either the smooth  $\clF=C^\infty$,   the Bessel-potential  $\clF=W^{s}_p$, $s\ge 0$, $p\in [1,\infty)$, or the H\"older functions $\clF=C^{m,\alpha}$, $m\ge 0$, $\alpha\in (0,1)$. 
The Hodge decomposition of $\Omega^k_{\clF}$ is given by
\begin{equation}\label{eq:Hodge_decomp_k}
\Omega^k_{\clF}=\clH^k_{\Delta}\oplus \Delta_H G\Omega^k_{\clF}=\mathcal{H}^k_{\Delta}\oplus \bd^*\Omega^{k+1}_{\clF^{+1}}\oplus \mathbf{d}\Omega^{k-1}_{\clF^{+1}},
\end{equation}
where  $G:\Omega^k_{\clF}\rightarrow \Omega^k_{\clF^{+2}}$ satisfies $\Delta_H G \alpha =\alpha -H \alpha$, $H: \Omega^k_{\clF}\rightarrow \mathcal{H}^k_{\Delta}$ is the harmonic projection  \cite{palais1968foundations, warner2013foundations, morrey2009multiple, scott1995, mikulevicius2004stochastic}, and $\clF^{+1}$ and $\clF^{+2}$ are the  one and two-more regular  spaces (in the non-smooth case). That is, $\clF^{+1}=W^{s+1}_p$ and $\clF^{+2}=W^{s+2}_p$, and similarly for H\"older spaces. 

Letting $k=1$ in \ref{eq:Hodge_decomp_k}, applying the diffeomorphism $\sharp: \Omega^1_{\clF}\rightarrow \mfk{X}_{\clF}$, and defining 
$$
\nabla \clF^{+1}:=\sharp \mathbf{d}\clF^{+1},\quad \mfk{X}_{\clF,\mu_g}:=\sharp\clH^k_{\Delta} \oplus \sharp \bd^*\Omega^{2}_{\clF^{+1}}, \quad 
\& \quad  \dot{\mfk{X}}_{\clF,\mu_g}:=\sharp \bd^*\Omega^{2}_{\clF^{+1}},
$$
we obtain an extension of the  Helmholtz decomposition of (possibly non-smooth) vector fields to manifolds: 
\begin{equation}\label{eq:vec_Hodge_decomp}
\mfk{X}_{\clF}=\mfk{X}_{\clF,\mu_g}\oplus\nabla \clF^{+1}= (\clH^1_{\Delta})^{\sharp}\oplus \dot{\mfk{X}}_{\clF,\mu_g} \oplus \nabla \clF^{+1},
\end{equation}
which is an orthogonal decomposition with respect to the inner product $( \cdot, \cdot )_{\mfk{X}_{L^2}}: \mfk{X}_{L^2}\times \mfk{X}_{L^2}\rightarrow \bbR$ defined by
$$
(u, v )_{\mfk{X}_{L^2}} = \int_Mg(u,v)\mu_g, \quad u,v\in \mfk{X}_{L^2}.
$$
Using  $\mathbf{i}_{u} \mu_g= \star u^{\flat}$ and Cartan's formula, we find $\operatorname{div}_{\mu_g} u=-\bd^* u^{\flat}=0$ for all $u\in \mfk{X}_{\clF,\mu_g}$. Thus,   $\mfk{X}_{\clF,\mu_g}$ consists of divergence-free vector fields and $\dot{\mfk{X}}_{\clF,\mu_g}$ consists of harmonic-free and divergence-free vector fields.

Let us recall the canonical pairing  (\ref{def:canonical_pairing}) $\langle \cdot, \cdot\rangle_{\mfk{X}}: \mfk{X}_{C^\infty}^{\vee}\times \mfk{X}_{C^\infty}\rightarrow \bbR$:
$$
\langle \alpha\otimes \mu, u\rangle_{\mfk{X}}=\int_{M}\alpha(u)\mu,
$$
and  diffeomorphism $\flat\otimes \mu_g: \mfk{X}_{C^\infty}\rightarrow \mfk{X}_{C^\infty}^{\vee}$, which satisfies 
$\langle \flat\otimes\mu_g(v), u\rangle_{\mfk{X}}=(v,u)_{\mfk{X}_{L^2}}$ for all $u,v\in \mfk{X}_{C^\infty}$.  Applying the diffeomorphism $\flat\otimes \mu_g$ to \eqref{eq:vec_Hodge_decomp}, we get
$$
\mfk{X}_{C^\infty}^{\vee}= (\operatorname{id}\otimes \mu_g) \Omega^1_{C^\infty}=(\operatorname{id}\otimes \mu_g)\mathcal{H}^1_{\Delta}\oplus (\operatorname{id}\otimes \mu_g)\delta\Omega^{2}_{C^\infty}\oplus (\operatorname{id}\otimes \mu_g)\mathbf{d}C^\infty.
$$
Define the `projection' operators $P:\mfk{X}_{C^\infty}^{\vee}\rightarrow  (\operatorname{id}\otimes \mu_g)\mathcal{H}^1_{\Delta}\oplus (\operatorname{id}\otimes \mu_g)\bd^*\Omega^{2}_{C^\infty}$ and $\dot{P}:\mfk{X}_{C^\infty}^{\vee}\rightarrow  (\operatorname{id}\otimes \mu_g)\bd^*\Omega^{2}_{C^\infty}$, which  act only on the one-form  component.
Clearly, if we restrict the canonical pairing $\langle \cdot, \cdot\rangle_{\mfk{X}}$ to $\mfk{X}_{C^\infty}^{\vee}\times \mfk{X}_{C^\infty,\mu_g}$ and $\mfk{X}_{C^\infty}^{\vee}\times \dot{\mfk{X}}_{C^\infty,\mu_g}$, then the pairing is degenerate; indeed,
$$
\langle \alpha \otimes \mu, u\rangle_{\mfk{X}}=0, \quad \forall u\in \mfk{X}_{C^\infty,\mu_g} \quad \Longrightarrow \quad P(\alpha\otimes \mu)=0,
$$
$$
\langle \alpha \otimes \mu, u\rangle_{\mfk{X}}=0, \quad \forall u\in \dot{\mfk{X}}_{C^\infty,\mu_g} \quad \Longrightarrow \quad \dot{P}(\alpha\otimes \mu)=0.
$$
Notice that the kernel of $P$ is $(\operatorname{id}\otimes \mu_g)\mathbf{d}C^\infty$ and the kernel of $\dot{P}$ is $(\operatorname{id}\otimes \mu_g)\mathcal{H}^1_{\Delta}\oplus (\operatorname{id}\otimes \mu_g)\mathbf{d}C^\infty$. To restore non-degeneracy, we  mod out by the kernel; the following definition is standard \cite{arnold1999topological, khesin2008geometry, khesin2020geometric}.

\begin{definition}\label{def:canonical_dual_incompressible}
Let $
\mfk{X}_{\clF,\mu_g}^{\vee}:=\mfk{X}_{\clF}^{\vee}\big/(\operatorname{id}\otimes \mu_g)\mathbf{d}\clF^{+1}$ and  $ \dot{\mfk{X}}_{\clF,\mu_g}^{\vee}=\mfk{X}_{\clF}^{\vee}\big/(\operatorname{id}\otimes \mu_g)\mathcal{H}^1_{\Delta}\oplus (\operatorname{id}\otimes \mu_g)\mathbf{d}\clF.
$
Moreover, we define the canonical pairings $\langle \cdot, \cdot\rangle_{\mfk{X}_{\mu_g}}: \mfk{X}_{C^{\infty},\mu_g}^{\vee} \times \mfk{X}_{C^{\infty},\mu_g}\rightarrow \bbR$ and $\langle \cdot, \cdot\rangle_{\dot{\mfk{X}}_{\mu_g}}: \dot{\mfk{X}}_{C^{\infty},\mu_g}^{\vee} \otimes \dot{\mfk{X}}_{C^{\infty},\mu_g}\rightarrow \bbR$  by
\begin{equation}\label{eq:incomp_pairing}
\begin{aligned}
\langle [\alpha\otimes \mu],u\rangle_{\mfk{X}_{\mu_g}}&=\langle \alpha\otimes \mu,u\rangle_{\mfk{X}}, \;\; \forall \;([\alpha\otimes \mu],u)\in \mfk{X}_{C^{\infty},\mu_g}^{\vee}\times \mfk{X}_{C^{\infty},\mu_g}\\   \langle [\alpha\otimes \mu],v\rangle_{\dot{\mfk{X}}_{\mu_g}}&=\langle \alpha \otimes\mu ,v\rangle_{\mfk{X}},\;\; \forall \;([\beta\otimes \nu],v)\in \dot{\mfk{X}}_{C^{\infty},\mu_g}^{\vee}\times \dot{\mfk{X}}_{C^{\infty},\mu_g},
\end{aligned}
\end{equation}
where the $[\alpha\otimes \mu]$  denotes an equivalence class with representative $\alpha\otimes \mu$.  It follows that $\flat \otimes \mu_g:\mfk{X}_{\clF,\mu_g}\rightarrow   \mfk{X}_{\clF,\mu_g}^{\vee}$ and  $\flat \otimes \mu_g:\dot{\mfk{X}}_{\clF,\mu_g}\rightarrow   \dot{\mfk{X}}_{\clF,\mu_g}^{\vee}$  are diffeomorphisms.  
\end{definition}

It can easily be checked the definition is well-defined in the sense that the right-hand-sides of  \ref{eq:incomp_pairing} are independent of the representative.  Indeed, for any two given representatives  $\alpha \otimes \mu$ and $\beta \otimes \nu$  of an equivalence class of $\mfk{X}_{\clF,\mu_g}^{\vee}$, we have
$$
P(\alpha\otimes \mu) = P (\beta \otimes \nu) \quad \Longleftrightarrow \quad \alpha\otimes \mu = \beta \otimes \mu + \bd f \otimes  \mu_g \;\; \textnormal{for some} \; f\in \clF^{+1}.
$$ 
and for any two given representatives  $\alpha \otimes \mu$ and $\beta \otimes \nu$  of an equivalence class of $\dot{\mfk{X}}_{\clF,\mu_g}^{\vee}$
$$
\dot{P}(\alpha\otimes \mu) = P (\beta \otimes \nu) \quad \Longleftrightarrow \quad \alpha\otimes \mu = \beta \otimes \mu + (\bd f+c) \otimes \mu_g \;\; \textnormal{for some}  \; f\in \clF^{+1}\; \& \; c\in \mathcal{H}^1_{\Delta}.
$$ 

\section{Auxiliary results}\label{App-auxil-results}

\subsection{Rough flows on Euclidean space}\label{sec:RDE_Euclidean}
\begin{theorem}\label{thm:rough_diffeo} There exists a continuous map 
$$
\operatorname{Flow} : C^{\alpha}_T\left(\mfk{X}_{C^{\infty}_b}(\bbR^d)\right)\times C^{\infty}_T\left(\mfk{X}_{C^{\infty}_b}(\bbR^d)^K\right)\times  \bclC_{g,T}(\bbR^K)\rightarrow C^{\alpha}_{2,T}(\operatorname{Diff}_{C^{\infty}}(\bbR^d))
$$
such that the flow $\eta_{ts}=\operatorname{Flow}(u,\xi,\bZ)_{ts}$, $(s,t)\in [0,T]^2$ satisfies the following properties:
\begin{enumerate}[(i)]
\item for all $(s,\theta,t)\in [0,T]^3$, $\eta_{tt}=\operatorname{Id}$ and  $\eta_{ t\theta }\circ \eta_{\theta s}=\eta_{ts}$;
\item $Y_{\cdot}=\eta_{\cdot s}(X)\in C^{\alpha}([s,T];\bbR^d)$ is the unique solution of
$$\textnormal{d}Y_t = u_t(Y_t)\textnormal{d}t + \xi_t(Y_t)\textnormal{d}\bZ_t, \;\; Y_s=X\in \bbR^d;$$
\item $\eta$ is the unique two-parameter flow satisfying (i) and 
$$
|\eta_{ts}-\mu_{ts}|_{\infty}\le C |t-s|^{3\alpha},\;\; \forall (s,t)\in [0,T]^2,
$$
for a constant $C$, where $\mu\in  C^{\alpha}_{2,T}(\operatorname{Diff}_{C^{\infty}}(\bbR^d))$ is the $C^{\infty}$-\textit{approximate flow} given by
$$\mu_{ts}:=\exp\left(u_{s}(t-s)+\sum_{k=1}^{K}\xi_k(s) \bbZ_{st}^k+\sum_{1\le k<l\le K}[\xi_k(s),\xi_l(s)]\bbA_{st}^{kl}\right),$$ 
or equivalently  by $\mu_{ts}(X):=Y_1$ such that 
$$
\dot{Y}_{\theta}=u_{s}(Y_{\theta})(t-s)+\sum_{k=1}^{K}\xi_k(s)(Y_{\theta}) \delta Z_{st}^k+\sum_{1\le k<l\le K}[\xi_k(s),\xi_l(s)](Y_{\theta})\bbA_{st}^{kl}, \;\;\theta \le 1,\;\; Y_0=X\in \bbR^d;
$$
\item for all $f \in C^{\infty}_b(\bbR^d;\bbR)$ and $s\in [0,T]$,  $\eta=f(\eta_{ \cdot s}^{-1})\in C^{\alpha}([s,T];C^{\infty}(\bbR^d;\bbR))$ satisfies 
$$
\eta_t+\int_s^t \pounds_{u_r}g_r\textnormal{d}r+\int_s^t \pounds_{\xi_r}g_r\textnormal{d}\bZ_r=f;
$$
in $C^{\infty}_b(\bbR^d)$; that is, $(\xi[g],-\xi[\xi[g]])\in \bclD_{\bZ}([s,T];C^{\infty}_b(\bbR^d)).$
\end{enumerate}
\end{theorem}
\begin{remark}
Claims (i-iii) are a direct extension of Corollary 11.14 of \cite{friz2010multidimensional}; one can easily verify the  Davie's estimates (Corollary 11.14 of \cite{friz2010multidimensional}). We do not impose that our drift coefficient is Lipschitz in time because it is the solution of a rough partial differential equation driven by the path $\bZ$ in our framework, and hence it can only be expected to be $\alpha$-H\"older continuous. We also allow for time dependence in the vector field $\xi$ since this is used in Section \ref{sec:HP_var_princ} to take variations.  Claim (iv) is a minor extension of  Theorem 16 of  \cite{bailleul2014flow,bailleul2019rough} (or Theorem 1.27 of \cite{driver2018global}), which uses the  method of approximate flows. It is possible to weaken the required regularity  in space and time of the coefficients, but  for simplicity, we do not pursue this.

Claim (iv) is the initial-value first-order  linear transport rough partial differential equation (RPDE) for the inverse flow. We understand $g$ as the classical solution in the spatial variable and in the sense of controlled rough paths in time. In \cite{caruana2009partial}[Corollary 8], the method of characteristics solution theory for initial-value RPDEs (in the case $u\equiv 0$) is established and the solutions are  characterized as being a limit point of $g^n=f(X^n_t)$, where $X^n$ is the solution of the time-reversal along a sequence of smooth paths $\bZ^n=(Z^n,\bbZ^n)$ converging to  $\bZ$ in the rough path topology. It is not clear that one can deduce a stronger notion of solution (in the sense of controlled rough paths) from this result in a simple manner (see, also, Remark 2.10 of \cite{diehl2017stochastic}). The works \cite{diehl2017stochastic} and  \cite{bellingeri2020transport}prove the well-posedness of the final-value transport equation and its adjoint, the initial-value continuity equation, in the sense of controlled rough paths. We were not able to find  the exact result in the literature. 

Nevertheless, the solution of the RPDE can be derived using theory of unbounded rough drivers (\cite{bailleul2017unbounded,deya2019priori}), which is analogous to the energy method in deterministic PDE. Indeed, one may first derive a solution $g\in C([s,T];W^{n}_2(\bbR^d))$ under the assumption $u\in C_T(C^{m}(\bbR^d;\bbR^d))$, $\xi\in C^{m+3}(\bbR^d;\bbR^d)$, and $f\in W^{m}_2(\bbR^d)$ for any $m\in \bbN_0$ by adapting Theorem 2 of \cite{hocquet2018energy} and Section 5.2 of \cite{coghi2019rough}. Then one may obtain a solution $u\in C^{\alpha}_T(C^{\infty}(\bbR^d;\bbR))$  by applying the Sobolev embedding. Finally, one can apply the pull-back version of the Lie chain rule Theorem \ref{thm:Lie_chain} to show that $g(\eta_{\cdot s})=f$.
\end{remark}

\subsection{Rough Fubini's theorem}
Let $T>0$,  $\alpha\in \left(\frac{1}{3},1\right]$, and  $\bZ\in \bclC^{\alpha}_{g,T}(\bbR^K)$.  By virtue of the fact that rough integration is a linear continuous map, we can easily obtain a version of Fubini's theorem.  Let $(X,\clA,\mu)$ be a $\sigma$-finite measured space and $W$ be a Banach space. Denote by $L^1(X;W)$ the Banach space of  equivalent classes of Bochner integrable functions $f:X\rightarrow W$ endowed with the norm
$$
|f|_{L^1(X;W)}=\int_{X}|f|_Vd\mu, \; \; f\in L^1(X;W).
$$
Recall that for an arbitrary Banach space  $V$ and linear map $L\in \clL(W,V)$,
\begin{equation}\label{eq:bounded_op_integral}
L\int_{X}fd\mu = \int_{X}Lfd\mu, \;\;\; \forall f\in L^1(X;W).
\end{equation}
The following lemma is then a straightforward application of  \eqref{eq:bounded_op_integral}, Theorem \ref{thm:rough_int}, and
$$
L^1(X; \bclD_{Z,T}^{\alpha}(V^K))\subset \bclD_{Z,T}(L^1(X;V)^K),
$$
which itself follows from  Fatou's lemma.
\begin{lemma}[Rough Fubini] \label{lem:Fubini}
If $\bF=(F,F')\in L^1(X; \bclD_{Z,T}^{\alpha}(V^K))$, then for all $(s,t)\in \Delta_T^2$,
$$
\int_{X}\int_s^{t} F_r\rmd \bZ_rd\mu  =\int_s^{t} \int_{X} F_rd\mu \rmd \bZ_r.
$$
\end{lemma}

\subsection{Fundamental lemma of the calculus of rough variations}\label{sec:FLCRC}
Let $T>0$,  $\alpha\in \left(\frac{1}{3},1\right]$, and  $\bZ\in \bclC^{\alpha}_{g,T}(\bbR^K)$. 

\begin{lemma}\label{lem:fund_calc_var}
Assume that $\bY=(Y,Y') \in \bclD_{Z,T}(\bbR^K)$ and $\lambda \in C_T(\bbR)$ satisfy
\begin{equation} \label{tested against phi}
\int_a^b \lambda_t \dot{\phi}_t \rmd t = \int_a^b \phi_t Y_t\rmd \bZ_t
\end{equation}
for all $\phi \in C^{\infty}_T(\bbR)$ such that $\phi_{0}=\phi_T = 0$.  Then for all $(s,t) \in \Delta_T^2$,
\begin{equation} \label{equality for all times}
\delta\lambda_{st}  = \int_s^t Y_r \rmd \bZ_r.
\end{equation}
\end{lemma}
\begin{remark}
On the right-hand-side of \eqref{tested against phi}, we have  used that $(\phi,0)\in \bclD_{Z,T}(\bbR)$, and thus that   $\phi \bY = (\phi Y, \phi Y')\in \bclD_{Z,T}(\bbR^K)$ by  Lemma \ref{lem:controlled_chain_product}.
\end{remark}
\begin{proof}
\textbf{Step 1.}
We will begin by showing that equality \eqref{tested against phi} must hold for any Lipschitz $\phi\in C^1_T(\bbR)$ such that $\phi_0  = \phi_T = 0$, where $\dot{\phi}$ on the left hand side is the bounded weak derivative (which exists by Rademacher's theorem \cite{heinonen2012lectures}[Theorem 6.15]).  Consider a  mollifier on $\bbR$ defined by $\rho_n(\theta ) := n \rho(n\theta )$, $n\in \bbN$,  where $\int_{\bbR} \rho(\theta) \rmd \theta  = 1$ and  $\operatorname{supp}\rho\subset [0,T]$. Because $\phi$ vanishes at the end points, we can extend $\phi$ by zero
$$
\tilde{\phi}_t = \left\{
\begin{array}{ll}
\phi_t & t \in [0,T] \\
0 & t \notin [0,T],\\
\end{array}
\right.
$$
and note that $\tilde{\phi}\in C^1_T(\bbR)$ has the same Lipschitz constant as $\phi$. For a given  $n\in\bbN$, define 
$$
\phi^n_t : = \tilde{\phi} \ast \rho_n(t) = \int_{\bbR} \tilde{\phi}_{t-\theta } \rho_n(\theta ) d\theta = \int_a^b \phi_\theta \rho_n(t - \theta) d\theta,\; \; t\in \bbR,
$$
which is clearly bounded in $C_T(\bbR)$. 
For all $n\in \bbN$ and $s,t\in [0,T]$, we find
$$
\left| \phi^n_t - \tilde{\phi}_t \right|   
= \left| \int_{\bbR} (\tilde{\phi}_{t - \theta} - \tilde{\phi}_t) \rho_n(\theta) d\theta  \right|  \lesssim \int_{\bbR} |\theta | \rho_n(\theta) d\theta = n^{-1} \int_{\bbR} |\theta| \rho(\theta) d\theta
$$
and
$$
|\delta \phi^n_{st}| = \int_{\bbR} (\tilde{\phi}_{t-\theta} - \tilde{\phi}_{s-\theta}) \rho_n(\theta) d\theta \lesssim |t-s|  \int_{\bbR}  \rho_n(\theta) d\theta = |t-s|.
$$
By Arzela-Ascoli's theorem, $\phi^n \rightarrow \phi$ uniformly, and, in fact, in $C^{\beta}_T(\bbR)$ for all $\beta < 1$. A classical argument shows that $$\lim_{n \rightarrow \infty} \int_0^T \dot{\phi}^n_t \lambda_t \rmd t=  \int_0^T \dot{\phi}_t \lambda_t \rmd t.$$
For fixed $\bY  \in \bclD_{Z,T}(\bbR^K)$, the mapping $\psi \mapsto \psi \bY := (\psi Y, \psi Y')$ is a linear and continuous operation from $C^{\beta}_T(\bbR)$ to $\bclD_{Z,T}(\bbR^K)$ for all $\beta \geq 2 \alpha$. Moreover,
\begin{equation} \label{multiplicator on controlled}
|\psi \bY|_{\bZ} \leq |\psi|_{\beta} |\bY|_{\bZ}( |Y|_{\infty} + |Y'|_{\infty}).
\end{equation}
Thus, by the continuity of the rough path integral (Theorem \ref{thm:rough_int}), we obtain 
$$\lim_{n \rightarrow \infty} \int_a^b \phi_t^n Y_t \rmd \bZ_t= \int_a^b \phi_t Y_t \rmd \bZ_t,
$$
which  completes  step 1. 

\bigskip

\textbf{Step 2.}
We will now construct a sequence of Lipschitz functions $\{\phi^n\}_{n\in \bbN}\subset C^1_T(\bbR)$ converging to the characteristic function $ \mathbf{1}_{[s,t]}$, for $s,t\in \bbR$ such that  $0<s<t<T $, and then pass to the limit  on both sides of \eqref{tested against phi} to obtain \eqref{equality for all times}. We then extend the equality to $(s,t)\in \Delta_T^2$ by continuity.

Towards this end, for large enough  $n\in \bbN$ and $r\in [0,T]$, define  
$$
\phi^n_r = \left\{ 
\begin{array}{ll}
1 & r \in [s,t] \\
n(r - s) +1 & s \in [s - n^{-1},s] \\
n(t - r) +1 & s \in [t,t+ n^{-1}] \\
0 & \textrm{ otherwise}, \\
\end{array}
\right.
$$
so that $| \phi^n |_{\infty} =1$ and $| \dot{\phi}^n |_{\infty} = n$ where $\dot{\phi}^n$ is the weak derivative defined by
$$
\dot{\phi}^n_r = \left\{ 
\begin{array}{ll}
n & r \in [s - n^{-1},s] \\
-n & r \in [t,t+ n^{-1}] \\
0 & \textrm{ otherwise}. \\
\end{array}
\right.
$$
A classical argument shows that $$\lim_{n\rightarrow \infty}\int_0^T \lambda_r \dot{\phi}^n_r\rmd r = \delta \lambda_{st}.$$
Since the rough integral is an increment, we have
$$
\int_0^T\phi^n_r Y_r \rmd \bZ_r = \int_{s - n^{-1}}^s \phi^n_r Y_r \rmd \bZ_r + \int_s^{t}  Y_r \rmd \bZ_r +\int_{t}^{t + n^{-1}} \phi^n_r Y_r \rmd \bZ_r.
$$
If we can show that the first and last integrals converge to zero as $n \rightarrow \infty$, then we are finished. We will only show that the last term converges to zero because the argument for the  first integral is easier. Let $C$ denote a constant that is independent of $n$ and may vary from line to line. By Theorem \ref{thm:rough_int} and the fact that  $| \phi^n \bY |_{\bZ} \le C n$  by \eqref{multiplicator on controlled}, we find 
$$
\int_{t}^{t + n^{-1}} \phi^n_r \bY_r \rmd \bZ_r = \phi^n_{t} Y_{t}\delta Z_{t, t + n^{-1}} + \phi^n_{t} Y_{t}' \bbZ_{t, t + n^{-1}} + R^{n}\left(t, t + n^{-1}\right),
$$
where 
$$
|R^{n}(t, t+ n^{-1})| \le C (([Z]_{\alpha} + [\bbZ]_{2\alpha}) |\phi^n \bY| |t + n^{-1} - t|^{3 \alpha} \le C n^{1 - 3 \alpha } \rightarrow 0.
$$
as  $n\rightarrow \infty$. Moreover, 
\begin{align*}
\left| \phi^n_{t} Y_{t}\delta  Z_{t, t + n^{-1}} + \phi^n_{t} Y_{t}' \bbZ_{t, t + n^{-1}} \right| \leq  |Y|_{\infty} [Z]_{\alpha} n^{-\alpha}  +|Y'|_{\infty} [\bbZ]_{2 \alpha} n^{-2\alpha} \rightarrow 0,
\end{align*}
as  $n\rightarrow \infty$, which completes the proof.
\end{proof}

\section{The variational principle for  incompressible fluids on smooth paths}\label{sec:var_smooth_paths}

The purpose of this section is to explain the variational principles we formulate in this paper in the simplified setting of an incompressible homogeneous ideal fluid evolving on the the torus with a smooth perturbation. The beginning of the section can be read with  no knowledge of differential geometry. The rest of the section assumes some basic knowledge of differential geometry (see Section \ref{sec:differential_geometry}). 

We also explain the presence of the so-called line-element stretching term in our main equation. The presence of this term distinguishes our equations from a pure transport perturbation of the deterministic Euler equation on flat space in velocity form. In particular, we show that the stretching term arises as a direct consequence of the variational principle and not by how momentum is characterized; that is to say, our variational principle indirectly enforces a covariant formulation, which naturally leads to a Kelvin's circulation theorem, helicity conservation in dimension three, and enstrophy conservation in dimension two. 

By appealing to the Helmholtz decomposition (Hodge decomposition), we explicitly show the decomposition of the pressure terms into the unperturbed and perturbed part, which motivates the corresponding decomposition in the rough case. As a result of the presence of the  stretching term, our equations do not preserve  mean-freeness (i.e., harmonic-freeness) unless we impose an additional constraint in the variational principle. By imposing this constraint, the velocity $u$ can be recovered directly from the vorticity  $\tilde{\omega}=\nabla \times u$ via the Biot-Savart law. In vorticity form, our equations are a pure transport perturbation of the deterministic Euler equation in dimension two. The associated vorticity equation in the  Brownian setting has been studied   in the literature  with $u$ recovered directly from the vorticity $\omega$ via Biot-Savart \cite{brzezniak2016existence, crisan2019solution,crisan2019well,brzezniak2019existence}.

We consider an incompressible homogeneous fluid moving on the flat $d$-dimensional torus $\bbT^d$ with the standard  volume form $dV$. Denote by $\mfk{X}$ the space of smooth vector fields, $\mfk{X}_{dV}$   the space of smooth divergence-free vector fields and   $\dot{\mfk{X}}_{dV}$  the  space of smooth divergence and mean-free vector fields. It follows that 
$$
\mfk{X}=\mfk{X}_{dV}\oplus \nabla C^\infty =\dot{\mfk{X}}_{dV}\oplus \bbR^d \oplus \nabla C^\infty,
$$
where the decomposition is orthogonal with respect to the $L^2$-inner product. Let  $P:\mfk{X}\rightarrow\mfk{X}_{dV}$, $Q: \mfk{X}\rightarrow \nabla C^\infty$,  $\dot{P}: \mfk{X}\rightarrow \dot{\mfk{X}}_{dV}$, and $H:\mfk{X}\rightarrow \bbR^d$ denote the corresponding projections (see Section \ref{sec:Hodge} and \eqref{eq:vec_Hodge_decomp}). We recall that in dimension three, $\operatorname{curl}:\dot{\mfk{X}}_{dV}\rightarrow \dot{\mfk{X}}_{dV}$ is an isomorphism, and in dimension two, $\operatorname{curl}:\dot{\mfk{X}}_{dV}\rightarrow C^\infty$ is an isomorphism. Denote the inverse of $\operatorname{curl}$ by $\operatorname{BS}$ (for Biot-Savart).

We assume that the Eulerian velocity field $v:[0,T]\rightarrow \mfk{X}_{dV}$  of the fluid admits a decomposition into a sum of a dynamical velocity variable $u:[0,T]\rightarrow\dot{\mfk{X}}_{dV}$ and a \emph{known} model vector field $\zeta:[0,T]\rightarrow \mfk{X}_{dV}$:
\begin{equation}
v_t = u_t + \zeta_t, 
\label{eq:model-decomp}
\end{equation}
where the  vector field $\xi$ admits the specified decomposition
$$
\zeta_t(x)=\xi(x) \dot{Z}_t=\sum_{k=1}^K\xi_k(x)\dot{Z}^k_t, \quad (t,x)\in [0,T]\times \bbT^d,
$$
where $\xi \in \dot{\mfk{X}}_{dV}^K$  and $Z:[0,T]\rightarrow \bbR^K$ in this appendix is a differentiable path, as opposed to the rough paths in the main text.  

\paragraph{Review of geometric ideal incompressible fluid dynamics.}
In ideal incompressible fluid dynamics, the fluid flow is obtained as a smooth, time-dependent volume-preserving diffeomorphism   $\eta: [0,T]\times \bbT^d\rightarrow \bbT^d$ by integrating the velocity vector field 
$$
\dot{\eta}_t= v_t\circ \eta_t=u_t\circ \eta_t + \xi_t\circ \eta_t, \;\; \eta_0=\operatorname{id}.
$$
In fact, $\eta$ may be regarded as a curve in the group of volume-preserving diffeomorphisms on $M$, denote by
$G=\operatorname{Diff}_{dV}(\bbT^d)$ and endowed with some appropriate topology. The Lagrangian, or material, velocity, is the velocity of the particle labeled by $X\in \bbT^d$ at time $t$. The Lagrangian velocity is given by $U_t(X)=\dot{\eta}_tX=v_t(\eta_tX)$; that is, $U=v\circ \eta$.
The Eulerian velocity, which is the velocity of the particle currently in position $x\in \bbT^d$ at time $t$ (i.e., $x=x(X,t)=\eta_tX$), can be expressed as
$$
v_t(x)=U_t(X)=U_t(\eta_t^{-1}x) \quad \textnormal{or} \quad v_t =\dot{\eta}_t \eta_t^{-1}=T_{\eta_t}R_{\eta_t^{-1}}\dot{\eta}_t,
$$
where the notation in the right-most expression  is the right action (technically the tangent lift of the action) of the inverse map $\eta_t^{-1}$ on the tangent vector $\dot{\eta}_t\in T_{\eta_t}G$  by the inverse map $\eta_t^{-1}$.  The action by the inverse map translates the tangent vector $\dot{\eta}_t$ at $\eta_t$ back to the identity $\mfk{g}=T_{\rm id}G\cong \mfk{X}_R(G)\cong\mfk{X}_{dV}$ (the space of divergence-free vector fields).   It follows that $v_t=\dot{\eta}_t \eta_t^{-1}$ is invariant under the action of the diffeomorphisms from the right given by $\eta_t\to \eta_t h$ for any fixed diffeomorphism $h\in \operatorname{Diff}_{dV}$. This symmetry  corresponds to the well-known invariance of the Eulerian fluid velocity vector field $v_t$ under relabelling of the Lagrangian coordinates as $X\to hX$. As discussed in Section \ref{sec:Kelvin_circ_thm}, right-invariance is the key to understanding the Kelvin circulation theorem from the viewpoint of Noether's theorem.  

\paragraph{Clebsch constrained variational principle.}
In order to derive an equation for $u$, we will apply a Clebsch constrained variational principle.  For arbitrary $u:[0,T]\rightarrow \dot{\mfk{X}}_{dV}$ and $\lambda, a:[0,T]\times \bbT^d\rightarrow \bbR^d$, we define
\begin{equation}
S(u,a,\lambda)=\int_0^T\int_{\bbT^d}\left[\frac{1}{2}|u_t|^2+\sum_{q=1}^d \lambda_t^q\left(\partial_t a_t^a+(v_t\cdot\nabla)a_t^q \right) \right]dV \rmd t.
\label{Clebsch-action}
\end{equation}
The history of the Clebsch constrained variational principle $\sym{\delta} S(u,a,\lambda)=0$ goes back to \cite{clebsch1859ueber}, as reviewed for fluid dynamics, e.g., in \cite{serrin1959mathematical}. Henceforth, we will drop the summation over $q\in \{1,\ldots, d\}$.

The first term in the Clebsch action integrand in \eqref{Clebsch-action} corresponds to the kinetic energy of the unperturbed velocity $u$ in the decomposition \eqref{eq:model-decomp}, not the total velocity, $v$. The second term indirectly imposes the constraint $\dot{\eta}=v\circ \eta$ through the advection relation.  Indeed, the  method of characteristics shows for a given $a_0:\bbT^d\rightarrow \bbR$, the path $a_t=a_0(\eta_t^{-1})=\eta_{t*}a_0$ (the push-forward of $a_0$ by $\eta_t$) satisfies  the advection equation
$$
\partial_t a_t+(v_t\cdot \nabla)a_t=\partial_t a_t+(u_t\cdot \nabla)a_t +(\xi\cdot \nabla)a_t \dot{Z}_t=0.
$$

To continue, we consider  variations of the form $$u^{\epsilon}=u+\epsilon\sym{\delta}u, \quad a^{\epsilon}=a+\epsilon\sym{\delta}a, \quad \lambda^{\epsilon} = \lambda + \epsilon \sym{\delta}\lambda, \quad \epsilon \in (-1,1),$$ for arbitrarily given $\sym{\delta}u:[0,T]\rightarrow \dot{\mfk{X}}_{dV}$ and $\sym{\delta}a,\sym{\delta}\lambda :[0,T]\times \bbT^d\rightarrow \bbR^d$ such that $\sym{\delta}u, \sym{\delta}a, \sym{\delta}\lambda |_{t=0,T}\equiv 0$. Upon taking these variations of the action functional, one finds
\begin{equation}\label{eq:varS_intro}
0 =
\sym{\delta}S(u,a,\lambda)=\int_0^T\int_{\bbT^d}\big [\left(u+\lambda \nabla a\right)\cdot \sym{\delta}u+\lambda \left(\partial_t\sym{\delta}a+(v\cdot \nabla)\sym{\delta} a\right)+ \sym{\delta}\lambda\left(\partial_t a+(v\cdot \nabla) a\right)\big]dV\rmd t.
\end{equation}
Here, `$\cdot$' denotes the inner product on $\bbR^d$ relative to the standard coordinate system (i.e., flat metric $\delta_{ij}$). We note also that since $u$ and $\sym{\delta}u$ are constrained to be mean and divergence-free, we have
$$
\int_{\bbT^d}\left(u+\lambda \nabla a\right)\cdot \sym{\delta}udV=\int_{\bbT^d}\left(u+\dot{P}\lambda \nabla a\right)\cdot \sym{\delta}udV.
$$

Using integration by parts in space and time, we get that $(u,a,\lambda)$ is a critical point of $S$ if and only if 
$$
u= -\dot{P} (\lambda \cdot \nabla a),\qquad  \partial_t\lambda+(v\cdot \nabla)\lambda=0, \qquad  \partial_t a+(v\cdot \nabla)a=0.
$$ 
It follows that 
\begin{equation}\label{eq:intro_calc_diamond}
\begin{aligned}
\partial_t \dot{P}u &= -\dot{P}\partial_t \lambda \nabla  a - \dot{P}\lambda  \nabla  \partial_t a=\dot{P}((v\cdot \nabla)\lambda)\nabla a +\dot{P}\lambda \nabla((v\cdot \nabla)a)\\
&=\left(\dot{P}((v\cdot \nabla)\lambda)\nabla a +\dot{P}\lambda (v\cdot \nabla)\nabla a\right)+\dot{P}\lambda \partial_{x^j}a\nabla v^j  \\
&=-\dot{P}(v\cdot \nabla)u-\dot{P}(\nabla v)^T\cdot u.
\end{aligned}
\end{equation}
Here $((\nabla v)^T \cdot u)^i :=\delta^{ij}u^k\partial_{x^j}v^k, $ and we  have used the $\delta^{ij}$ in order to maintain the geometric index convention even though we are working on flat space.  Therefore,
\begin{equation}\label{eq:main_eqn_intro}
\partial_t u_t +\dot{P}(v_t\cdot \nabla)u_t +\dot{P}(\nabla v_t)^T\cdot u_t= 0 \quad \Leftrightarrow \quad 
\partial_t u_t +(v_t\cdot \nabla)u_t +(\nabla v_t)^T\cdot u_t= -\nabla p_t + c_t.
\end{equation}
In terms of the projections $Q$ and $H$, we find
$$
-\nabla p= Q(v_t\cdot \nabla)u_t +Q(\nabla v_t)^Tu_t=Q(u_t\cdot \nabla)u_t +\left(Q(\xi\cdot \nabla)u_t +Q(\nabla \xi)^Tu_t\right) \dot{Z}_t 
$$
$$
c_t = H(v_t\cdot \nabla)u_t +H(\nabla v_t)^T\cdot u_t=H(\nabla v_t)^Tu_t=H(\nabla \xi)^T\cdot u_t\dot{Z}_t=\int_{\bbT^d}(\nabla \xi)^T\cdot u_tdV\dot{Z}_t.
$$
We note that the pressure $p$  enables us to enforce the constraint that $u$ is incompressible and the constant (in space) $c$ enables us  to enforce that $u$ is mean-free.
Substituting in $v=u+ \xi \dot{Z}$, we find
$$
\partial_t u_t +(u_t\cdot \nabla )u_t + \left((\xi \cdot \nabla)u_t  + (\nabla \xi)^T\cdot u_t\right)\dot{Z}_t =-\nabla \tilde{p}_t+c_t, \quad \tilde{p}_t = p_t +\frac{1}{2}|u_t|^2.
$$
In dimension two and three, one can readily check an equivalent formulation in terms of the  vorticity  $\tilde{\omega} = \operatorname{Curl}u$:
\begin{equation}\label{eq:vorticity_standard}
\partial_t \tilde{\omega}_t +(v_t\cdot \nabla)\tilde{\omega}_t -\mathbf{1}_{d=3}(\tilde{\omega}_t\cdot \nabla)v_t= 0, \qquad u = \operatorname{BS}(\omega).
\end{equation}

From this point on, we assume the reader is familiar with basic differential geometry  (see Section \ref{sec:differential_geometry}. Let us  introduce an arbitrary coordinate system and denote  by $\{dx^i\}_{i=1}^d$ a global frame of $\Omega^1$. Moreover, let the musical notation $\flat: \mfk{X}\rightarrow \Omega^1$ denote the isomorphism between vector fields and one-forms. Equation \ref{eq:main_eqn_intro} can be expressed covariantly as
\begin{equation}\label{eq:velocity_one-form}
\partial_t u^{\flat}_t +\pounds_{v_t}u^{\flat}_t= -\bd \tilde{p}+c^{\flat}_t,
\end{equation}
where the Lie-derivative operator $\pounds_{v_t}$ acts on the one-form $u^{\flat}$ to produce the one-form $\pounds_{v_t}u^{\flat}$, given by
$$
\pounds_{v_t}u^{\flat}=\pounds_{v_t}(g_{ki}u^kdx^i)=\left(v_t^j\partial_{x^j}(g_{ki}u^k)+g_{kj}u^k\partial_{x^i}v_t^j\right)dx^i=\left(v_t^ju^k\partial_{x^j}g_{ki} +g_{ki}v_t^j \partial_{x^j}u^k+g_{kj}u^k\partial_{x^i}v_t^j\right)dx^i.
$$
Here $\bd \tilde{p}$ is exterior derivative of the scalar-field $\tilde{p}$.

Let $\omega = \bd u^{\flat}\in \Omega^2$ denote the vorticity two-form obtained by applying the exterior derivative operator $\bd$. Since the exterior derivative commutes with the Lie derivative, one finds
\begin{equation}\label{eq:vorticity_two-form}
\partial_t \omega_t +\pounds_{v_t}\omega_t=0.
\end{equation}
The two characterizations of the vorticity $\omega$ and $\tilde{\omega}$ satisfying, \eqref{eq:vorticity_standard} and \eqref{eq:vorticity_two-form}, respectively, are related by the Hodge-star operator $\star: \Omega^2 \rightarrow \Omega^{d-2}$. In dimension two, $\tilde{\omega}=\star\omega \in \Omega^0$, and in dimension three, $\tilde{\omega}=\sharp \star \omega\in \dot{\mfk{X}}_{dV}$. In order to obtain equation \eqref{eq:vorticity_standard} directly from \eqref{eq:vorticity_two-form}, one uses that $\sharp \star $ and the Lie derivative commute (see, e.g., \cite{besse2017geometric}[Section A.6]). 

\medskip

\textbf{Kelvin circulation theorem.}
The covariant formulation immediately implies a Kelvin circulation theorem. Let $\gamma$ denote a closed loop in $\bbT^d$. Then using Reynolds transport theorem,
$$
\frac{d}{dt}\oint_{\eta_t(\gamma)} u^{\flat}_t=\oint_{\eta_t(\gamma)}(\partial_t u^{\flat}_t + \pounds_{v_t}u^{\flat}_t)=\oint_{\eta_t(\gamma) }\bd \tilde{p} = 0,
$$
where one transforms the integration around the moving loop $\eta_t(\gamma) $ to the loop $\gamma$ in the material frame by applying the pull back $\eta_t^*$ to the integrand, then takes the time derivative, applies the dynamic definition of the Lie-derivative, transforms back and substitutes the covariant equation of fluid motion \eqref{eq:velocity_one-form}. 
\medskip

\textbf{Helicity conservation.} In three dimensions, the helicity,  defined as $$\Lambda(\tilde{\omega}) =\int_{\bbT^3} u^{\flat}\wedge \omega= \int_{\bbT^3} u^{\flat}\wedge {\bd}u^{\flat}$$
measures the linkage of field lines of the divergence-free vector field $\tilde{\omega}$ \cite{arnold1999topological}. 
Owing to \eqref{eq:velocity_one-form} and \eqref{eq:vorticity_two-form}, we have 
$$
\partial_t (u^{\flat}\wedge \omega) = -\pounds_{v_t}(u^{\flat})\wedge \omega -u^{\flat}\wedge \pounds_{v_t}\omega-\bd \tilde{p}\wedge \omega,
$$
and hence
$$
\frac{d\Lambda}{dt}(\tilde{\omega})= \frac{d}{dt}\int_{\bbT^3} u^{\flat}\wedge \omega = \int_{\bbT^3} \bd \tilde{p}\wedge {\bd}u^{\flat}
=  \int_{\bbT^3} \bd (\tilde{p}\, {\bd}u^{\flat}) = 0\,.
$$
Thus, the linkage number of the vorticity vector field $\Lambda(\tilde{\omega})$ is preserved by the 3D Euler fluid equations \eqref{eq:velocity_one-form}.

\medskip

\textbf{Enstrophy conservation in two-dimensions.}
In two dimensions,  for any $f\in C^\infty$, we find
$$
\partial_t f(\tilde{\omega}_t) + (v_t\cdot \nabla f)(\tilde{\omega}_t)=0,
$$
and hence
$$
\int_{\bbT^2} f(\tilde{\omega}_t) dV=\int_{\bbT^2}f(\tilde{\omega}_0)dV.
$$
In particular, taking $f(x)=x^2$, we find
$$
\int_{\bbT^2} |\tilde{\omega}_t|^2 dV=\int_{\bbT^2}|\tilde{\omega}_0|^2dV,
$$
which implies that in two-dimensions enstrophy is conserved.

\bigskip 

\textbf{Momentum representation.} 
The Lie derivative of the volume form $dV$ along $v$ is zero since $\pounds_{v} dV=(\operatorname{div} v)\, dV=0$. Thus, we can also write equation \eqref{eq:main_eqn_intro} as 
\begin{equation}\label{eq:one-form-density-smooth}
\partial_t m_t +\pounds_{v_t}m_t= \bd \tilde{p}\otimes dV+c_t\otimes dV,
\end{equation}
where $m_t =u^{\flat} \otimes dV\in \mfk{X}^{\vee}:=\Omega^1\otimes \operatorname{Dens}$ denotes the space of smooth  one-form densities. In Sections \ref{sec:Clebsch_var_princ}, \ref{sec:HP_var_princ}, and \ref{sec:EP_var_princ}, the momentum will be characterized as a one-form density in order to conveniently incorporate both the inhomogeneous and compressible setting and work canonically. One can always  transform  between equivalent formulations once a metric and volume form have been fixed. We will  now explain how one can directly derive the various equivalent formulations directly from the Clebsch action functional.

\medskip

\textbf{Clebsch constrained variational principle revisited.} 
Let us now explain how we can directly derive \eqref{eq:velocity_one-form} and \eqref{eq:one-form-density-smooth} from the Clebsch action functional. The first term on the RHS of \eqref{eq:varS_intro} can be understood in a coordinate-free manner either as:
\begin{enumerate}[(i)]
\item $$(u+\lambda \nabla a, \sym{\delta}u)_{\mfk{X}_{L^2}}=\int_{\bbT^d}g(u+\lambda \nabla a, \sym{\delta}u) dV, \quad \textnormal{where} \;\; (\cdot,\cdot)_{\mfk{X}_{L^2}}: \mfk{X}\times \mfk{X}\rightarrow \bbR;$$
\item $$\langle u^{\flat}+\lambda \bd a, \sym{\delta}u\rangle= \int_{\bbT^d}\bi_{\sym{\delta}u}(u^{\flat}+\lambda \bd a) dV, \quad \textnormal{where} \;\; \langle \cdot,\cdot\rangle: \Omega^1\times \mfk{X}\rightarrow \bbR,$$
\item $$\langle u^{\flat}\otimes dV+ \bd a \otimes \lambda dV, \sym{\delta}u\rangle_{\mfk{X}}=\int_{\bbT^d}\bi_{v}\left[(u^{\flat}\otimes dV+\bd a\otimes \lambda dV)\right], \;\; \textnormal{where} \quad \langle \cdot,\cdot\rangle_{\mfk{X}}: \mfk{X}^{\vee}\times \mfk{X}\rightarrow \bbR,$$
\end{enumerate}
Let us denote
$$
(i)\;\;m=u\in \mfk{X}, \quad (ii)\;\; m = u^{\flat} \in \Omega^1,\quad (iii)\;\;\; m = u^{\flat}\otimes dV\in \mfk{X}^{\vee}.
$$ 
Let $$
(\lambda, a)_{L^2} = \int_{\bbT^d}\lambda a dV, \quad \textnormal{where} \quad (\cdot, \cdot):\Omega^0\times \Omega^0\rightarrow \bbR.
$$
It follows that 
$$
(i)\;\; (\lambda, \pounds_{v} a)_{L^2}= -(\lambda \diamond a, v)_{\mfk{X}_{L^2}}, \quad (ii)\;\;
(\lambda, \pounds_{v} a)_{L^2}= -\langle \lambda \diamond a, v\rangle,\quad \textnormal{or} \quad 
(iii)\;\;( \lambda , \pounds_{v} a)_{L^2} = -\langle\lambda \diamond a, v\rangle_{\mfk{X}},
$$
where 
$$
(i)\;\;\lambda \diamond a = -\lambda \nabla a,  \quad (ii)\;\;\lambda \diamond a = -\lambda \bd a, \quad \textnormal{or} \quad (iii)\;\;\lambda  \diamond a = -\bd a\otimes \lambda dV,
$$
respectively.

A critical point of the Clebsch action $S$ in \eqref{Clebsch-action} then  satisfies $$\dot{P}m = \dot{P}(\lambda \diamond a),$$ where we use the same notation $\dot{P}$ for the corresponding projection onto `divergence and harmonic-free' parts (see Section \ref{sec:Hodge}) in all three cases. In all three cases, following a similar calculation to the one given in \eqref{eq:intro_calc_diamond}, we obtain
$$
\partial_t m_t + \dot{P}\pounds_{v_t} m_t =0.
$$
The first case (i)  agrees with  the direct calculus computation given above. In general, the main ingredients of this computation (see Section \ref{proof:Clebsch}) are 1) the definition of $\diamond$, 2) the relation for all $v,w\in \mfk{X}$ and $a\in \Omega^0$ (i.e., for all tensor fields, $a$) that
$$
\pounds_{v}\pounds_{w} a- \pounds_{w}\pounds_{v} a= \pounds_{[v,w]}a ,
$$
and 3) that 
$$
\langle m,\operatorname{ad}_{v}w \rangle = \langle \pounds_{v}m,w \rangle,
$$
for all of the above pairings. That is, $\operatorname{ad}^*_vm=\pounds_{v}m$. 
If $v$ is not divergence-free, then $\operatorname{ad}^*_vm=\pounds_{v}m$  is only true for the pairing $\langle \cdot, \cdot \rangle_{\mfk{X}}$. 

Thus, one may characterize the `momentum' $m$ in various ways if a metric and volume form are fixed. However, the pairing $\langle \cdot, \cdot\rangle_{\mfk{X}}$ is canonical in that it  does not require a metric or volume form to be fixed (see the discussion in Section \ref{sec:duals_of_bundles}), and we use this pairing above. 

As a consequence of this discussion, we see that the line-element stretching term $(\nabla v_t)^T\cdot u_t$ in equation \eqref{eq:main_eqn_intro}  does not arise because we have characterized momentum in a certain way. This term appears even if we treat $m$ as a vector and work in a fixed standard coordinate system. As derived here, the stretching term tells us that the Clebsch variational principle has produced covariant coordinate-free equations. This is simply the generalized-coordinate theorem for the covariance of variational principles, the first being the Euler-Lagrange equations in classical mechanics, which are valued for precisely this reason.

\newpage

\noindent\textbf{Hamilton-Pontryagin variational principle.} 

Another way to impose the constraint on the deterministic flow decomposition is through the \emph{Hamilton-Pontryagin variational principle}.  The Hamilton-Pontryagin  action  integral on $[0,T]$ is given by
$$
S(u,\eta,\lambda)=\int_0^T\int_{\bbT^d}\left[\frac{1}{2}|u_t|^2+ \lambda_t\cdot\left(\dot{\eta}_t\eta_t^{-1}-v_t\right) \right]dV \rmd t. 
$$
Here, $\eta:[0,T]\rightarrow \operatorname{Diff}_{dV}$ is an arbitrary. The second-term corresponds to the Lagrangian dynamical constraint $\dot{\eta}=v\circ \eta$.  A variation of $\eta$ is simply a two-parameter curve $\eta: [-1,1]\times[0,T]\rightarrow \operatorname{Diff}_{dV}$ with equality of mixed-derivatives. 

One  refers to the stationary principle $\sym{\delta} S = 0$ for the action  integral above as the \textit{Hamilton-Pontryagin} variatonal principle since the Lagrangian constraint variable $\lambda$ is the symmetry-reduced version of the adjoint variable in the Pontryagin maximum principle, as first discussed for fluids in \cite{bloch2000optimal}. To explain this analogue further, the cost may be regarded as the $L^2$-norm of the (control) $u$, the path is constrained to satisfy $\dot{\eta}_t= v_t\circ \eta_t$, the endpoints of $\eta$ are treated as fixed (i.e., $\eta_0={\rm id}$ and $\eta_T=\psi$), and one seeks to find a path that minimizes the cost. However, in general, critical points are not global minimizers  \cite{brenier1989least, brenier1999minimal}. 

\section{A few words of motivation  for the theory of rough paths}\label{sec:motivation}

Let $\{\xi_k\}_{k=1}^K\subset \mathfrak{X}_{C^{\infty}}$ be  a family of smooth vector fields on a closed manifold $M$.  Let $\alpha\in (0,1]$ and  $Z\in C^{\alpha}_T(\bbR^K)$.  Consider the ordinary differential equation
\begin{equation}\label{eq:diff_eqn}
\textnormal{d}Y_t =\sum_{k=1}^K \xi_k(Y_t)\textnormal{d}Z^k_t, \;\; Y_t|_{t=0}=Y_0.
\end{equation}
If we can solve \eqref{eq:diff_eqn}, then we  expect  for any $f\in C^{\infty}(M)$ that  $f(Y)\in C^{\alpha}_T(\bbR)$, and hence  $$\xi_k [f](Y)=\xi^i_k(Y)\partial_{x^i}f(Y)\in C^{\alpha}_T(\bbR^K).$$ If we require $2\alpha>1$, then the integral $\int_0^t \xi_k[ f](Y_s)\textnormal{d}Z^k_s$ in 
\begin{equation}\label{eq:diff_eq_int_form}
f(Y_t)=f(Y_0)+\sum_{k=1}^K\int_0^t \xi_k[f](Y_s)\rmd Z_t^k
\end{equation}
may be  defined as a Young integral \cite{young1936inequality} (see Lemma \ref{lem:sewing}), and we  expect to have stability properties of the solution in terms of the  path $Z$; that is, the  mapping $Z\in C^{\alpha}_T(\bbR^K)
\mapsto f(Y)\in C^{\alpha}_T(\bbR)$ is continuous for all $f\in C^{\infty}(M)$. However, if $2\alpha\le 1$, then Young integration is inadequate to develop a pathwise solution theory with a stability property.

A prime example of such a path is a realization of a $K$-dimensional real Brownian motion  $Z_t^k=B^k_t(\omega)$, $\omega\in \Omega$, for which it is known that on a set of probability one,  $B(\omega)\in C^{\alpha}_T(\bbR^K)$ for $\alpha<\frac{1}{2}$. Indeed, T.\ Lyons showed \cite{lyons1991non} (see, also, Prop.\ 1.1. in \cite{friz2014course}) that there exists no separable Banach space $\clB\subset C_T(\bbR^K)$ in which the sample paths of Brownian motion lie and for which the integral $\int_0^{\cdot}f_tdg_t: C^{\infty}_T(\bbR)\times C^{\infty}_T(\bbR)\rightarrow C^{\infty}_T$ extends in a continuous way to $\clB\times \clB\rightarrow C_T(\bbR^K)$. Since the  integral $\int_0^t B_s^1(\omega) \textnormal{d}B_s^2(\omega)$ is expected to be the solution of the simplest  differential equation driven by a two-dimensional Brownian motion $B(\omega)=(B^1(\omega),B^2(\omega))$, the result of T.\ Lyons indicates that the development of a pathwise theory must take into account the additional structure of the solution $Y$ and the path $Z$. 

If, however,  $K=1$ or the vector fields commute (i.e. $[\xi_{k_1},\xi_{k_2}]\equiv 0$ for all $k_1,k_2$), then  a solution theory can be developed for  continuous paths $Z\in C_T(\bbR^K)$. Indeed, H.\ Doss and H.\ Sussman \cite{doss1977liens, sussmann1978gap} showed that the solution can be defined by 
$$
Y_t=\exp\left(\sum_{k=1}^K\xi_k Z^k_t\right)Y_0,
$$
where  $\exp\left(\sum_{k=1}^K\xi_k Z^k_t\right)$ is the flow of the vector field $\sum_{k=1}^K\xi_k Z^k_t$ with $t$-fixed at  time $t=1$ (i.e., the time-one map). This is clearly a continuous function of $Z$ and satisfies the equation exactly if $Z$ is differentiable.
In \cite{sussmann1978gap}[pg.\ 21], H.\ Sussman discussed the connection of  pathwise solutions with so-called Wong-Zakai results/anomalies (see, e.g., \cite{sussmann1991limits}) and clearly indicated that: (i) extending this result to $K>1$ in the  non-commutative case  would require substantially new methods; and (ii) finding such an extension would lead to significant progress in our understanding of the  anomalies. 

The key idea for extending the pathwise theory came from T. Lyons \cite{sip1993,lyons1995interpretation,lyons1994differential} as a \emph{tour de force} which combined iterated integrals  \cite{peano1888integration, bocher1914introduction, magnus1954exponential, chen1957integration}, control theory \cite{Chen1963, fliess1981fonctionnelles, sussmann1983lie, fliess1986volterra, sussmann1987general}, system identification and filtering \cite{mitter1979multiple, boyd1984analytical, boyd1985fading}, numerical schemes \cite{butcher1972algebraic, clark1980maximum, sussmann1988product,gaines1994algebra, hairer2006geometric}, and renormalization \cite{guttinger1955products, fliess1981fonctionnelles, connes1999hopf, blanes2009magnus}.  

To describe this idea, let us assume for the moment that third-order brackets vanish (i,e.,  $[\xi_{k_1},[\xi_{k_2},\xi_{k_3}]]=0$ for all $k_1,k_2,k_3$) and that  $Z_t^k=B_t^k(\omega)$ is a realization of a $K$-dimensional Brownian motion. Consider  for all $(s,t)\in \Delta_T^2$, the time-one map
\begin{align*}
\mu_{st}(\omega)&= \exp\left(\sum_{k=1}^K\xi_k\delta B^k_{st}(\omega)+\frac{1}{2}\sum_{k,l=1}^K[\xi_l,\xi_k]\mathbb{B}_{st}^{lk}(\omega)\right)=\exp\left(\sum_{k=1}^K\xi_k\delta B^k_{st}(\omega)+\sum_{k<l}[\xi_l,\xi_k]\mathbb{A}^{lk}_{st}(\omega)\right),
\end{align*}
where the quantity $$\mathbb{B}^{lk}_{st}(\omega):=\left(\int_{s}^t\int_{s}^{t_1}\textnormal{d}B^l_{t_2}\circ \textnormal{d}B^k_{t_1}\right)(\omega)$$ is the $2\alpha$-H\"older modification of the Stratonovich integral evaluated at $\omega$ and  $\mathbb{A}^{lk}_{st}(\omega)=\frac{1}{2}\left(\mathbb{B}^{lk}_{st}(\omega)-\mathbb{B}^{kl}_{st}(\omega)\right)$. Then $Y_t(\omega):=\mu_{0t}(\omega)Y_0$  is the pathwise  solution of the SDE.
Thus, the notion of path is enhanced to include the addition of the iterated-integral $$\mathbf{B}(\omega)=(B(\omega),\bbB(\omega))\in C^{\alpha}_T(\bbR^K)\times C^{2\alpha}_{2,T}(\bbR^{K\times K}), \quad \alpha<\frac{1}{2},$$ where  $\omega$ belongs to a set  $\Omega'\in \clF$ of probability one. Of course, we are able to construct a pathwise solution because probability theory (i.e., $L^2(\Omega)$-closure) enabled us to construct the iterated integral of the path $Z_t^k=B_t^k(\omega)$ and the Kolmogorov continuity theorem allowed us to obtain a $2\alpha$-H\"older version of the iterated integral. Furthermore, the map is stable in the sense that for any $\{B^n(\omega)\}_{n\in \bbN}$ such that $\mathbf{B}^n(\omega)=(B^n(\omega),\bbB^n(\omega))\rightarrow \mathbf{B}(\omega)$, one has $\mu^n_{st}(\omega)\rightarrow \mu_{st}(\omega)$ as $n\rightarrow \infty$. It is in this sense that the $Y_t(\omega)$ is a pathwise solution. The reader familiar with Magnus expansions will notice that $\mu$ is essentially the second-order Magnus expansion and the expansion is exact because of the third-order bracket condition. 

Use of the relation $\bbB^{lk}_{s t}(\omega)+\bbB^{kl}_{s t}(\omega)= \delta B^l_{st}(\omega)\delta B^k_{st}(\omega)$ shows that for all $(s,t)\in \Delta_T$ and $f\in C^{\infty}(M)$, \begin{equation}\label{eq:chenn_fleiss}
f(Y_t(\omega))=f(Y_s(\omega))+\sum_{k=1}^K \xi_k[f](Y_s(\omega))\delta B^k_{st}(\omega) +\sum_{k,l=1}^K \xi_l[\xi_k[f]](Y_s(\omega))\mathbb{B}_{st}^{lk}(\omega)+f^{\sharp}_{st}(\omega),
\end{equation}
where  $f^{\sharp}: \Delta_T^2\rightarrow \bbR $ satisfies for a constant $C>0$
$$|f^{\sharp}_{st}(\omega)|\le C  |\xi|_{C^{3}}([B(\omega)]_{\alpha}+[\bbB(\omega)]_{2\alpha})^2|t-s|^{3\alpha}.$$
Upon defining for all $(s,t)\in \Delta_T$ and $\omega\in \Omega'$,
$$
\Xi_{st}=\xi_k[f](Y_s(\omega))\delta B^k_{st}(\omega)+\xi_l[\xi_k[f]](Y_s)\bbB_{st}^{lk}(\omega)+f^{\sharp}_{st}(\omega),
$$
and invoking $|f^{\sharp}_{st}(Y_s)(\omega)|\le C(\omega) |t-s|^{3\alpha}$ and $ \delta_2 \bbB^{lk}_{s\theta t}(\omega)=\delta B^l_{s\theta}(\omega)\delta B^k_{\theta t}(\omega)$, one can directly check that $\Xi\in C_{2,T}^{\alpha,3\alpha} (\bbR)$. Hence, one may apply Lemma \ref{lem:sewing} to construct the integral $\clI\Xi=(\int \xi[f] \rmd \mathbf{B})(\omega)$. This integral agrees with the Stratonovich integral $\left(\int_s^t \xi[f](Y_s)\circ \textnormal{d}B_s\right)$ on a set of  probability one (see Theorem \ref{thm:rough_int}).

The expansion \ref{eq:chenn_fleiss} is called the second-order Chen-Fleiss expansion in the system-identification and control literature. Here, $B$ can be interpreted as a control. Such expansions illustrate that all information of the controls impact on the system is contained in the iterated integrals of the control. The Chen-Fleiss expansion can  be obtained directly from \eqref{eq:diff_eq_int_form} by formally iterating the integral (Taylor series) with $Z=B$ and
and then evaluating at $\omega$. One  immediately recognizes the advantage of the Magnus expansion over the  Chen-Fleiss series. Namely, the Magnus expansion is an exact solution of an approximating system, while the Chen-Fleiss series is not \cite{kawski2004bases}.  Nevertheless, such expansions are of great utility in the  study of controllability  and analysis of control systems \cite{sussmann1987general}. 

From the above discussion, we have learned that for all $\omega\in \Omega'$,  $(B(\omega),\bbB(\omega))$ belongs to  the class of $(Z,\bbZ)\in C^{\alpha}_T(\bbR^K)\times C^{2\alpha}_{2,T}(\bbR^{K\times K})$ such that  for all $(s,t)\in \Delta_T$ 
\begin{equation}\label{eq:chen1}
\delta_2\bbZ^{lk}_{s\theta t}=\delta Z^l_{s\theta}\delta Z^k_{\theta t}, \quad \bbZ^{lk}_{s t}+\bbZ^{kl}_{s t}= \delta Z^l_{st} \delta Z^k_{st}.
\end{equation}
For $\alpha \in \left(\frac{1}{3},\frac{1}{2}\right]$, the closure of the set of Lipschitz paths in $C^{\alpha}_T(\bbR^K)\times C^{2\alpha}_{2,T}(\bbR^{K\times K})$ that satisfy the above properties is called the space of geometric rough paths.

The fundamental idea of T.\ Lyons is that  in the general case of $\xi$ not having vanishing Lie brackets, there is a notion of solution of equations driven by geometric rough paths $\bZ=(Z,\bbZ)$ and an accompanying well-posedness theory. There  are many equivalent notions of solution (see Lemma \ref{lem:soln_RDE}). For example, the Chen-Fleiss expansion up to level two can be used to define an intrinsic notion of solution by additionally specifying that the remainder $f^\sharp$ belongs $C^{3\alpha}_{2,T}(\bbR)$ for any $f\in C^\infty(M)$ \cite{davie2008differential}. Higher-order iterated integrals are needed  if $\alpha<\frac{1}{3}$. However, one still needs a means of constructing $\bbZ$, and probability  is the main tool used to do so. Effectively, then, the  technical ingredient necessary to develop the basic theory of rough paths  is  the sewing lemma (Lemma \ref{lem:sewing}) \cite{gubinelli2004controlling}. To wit, the sewing lemma is used to establish the existence of integrals against  $\bZ$ and to obtain bounds on `remainder'  $f^{\sharp}_{st}$. It is also possible to prove that there exists a unique two-parameter flow associated with the time-one map
$$
\mu_{st}(\omega)= \exp\left(\sum_{k=1}^K\xi_k\delta Z^k_{st}(\omega)+\frac{1}{2}\sum_{k,l=1}^K[\xi_l,\xi_k]\mathbb{Z}_{st}^{lk}(\omega)\right), \quad \forall (s,t)\in \Delta_T^2,
$$
even if the third-order Lie-brackets of $\xi$ do not vanish. The main ingredient in this approach is the  multiplicative sewing lemma, developed by I.\ Bailleul \cite{bailleul2014flow}.

A prophethetical quote of M. Fleiss \cite{fliess1981fonctionnelles}[pg. 33] translated into English reads, 
\begin{quote}
We know (cf. Schwartz \cite{schwartz1954limpossibilite}) that it is generally impossible to multiply the distributions and, in particular, that the powers $\delta^2,\delta^3,\cdots$, of the Dirac impulse are not distributions. Similarly here, we cannot represent the square of a Dirac impulse by a series of Chen. However, it is possible to propose what is called in physics a renormalization (et. G\"uttinger \cite{guttinger1955products}) based on natural combinatorial considerations.  
\end{quote}

T. Lyons showed that by postulating the existence of objects $\bbZ$ which satisfy \eqref{eq:chen1}, a  solution theory can be developed for differential equations driven by  rough paths. 
As explained above, probability is used to construct $\bbZ$. Thus, probability can be understood as a tool to renormalize through its construction of otherwise analytically ill-defined quantities $\bbZ$ -- and it is only this quantity that needs to be defined to construct a solution. M. Hairer extended the T. Lyons program by developing the theory of \emph{regularity structures} as the basis of a  solution theory for stochastic partial differential equations driven by white noise \cite{hairer2014theory} (see, also, \cite{friz2014course}). One of the key theorems in M.\ Hairer's theory  is the Reconstruction Theorem, which is a substantial generalization of the sewing lemma. As predicted by M. Fleiss \cite{fliess1981fonctionnelles} and H. Sussman \cite{sussmann1978gap}, this theory has had a transformative impact on renormalization in statistical physics, and of our understanding of stochastic differential equations (in finite and infinite dimensions) and the so-called Wong-Zakai anomalies. 

\section{Gaussian  rough paths}\label{sec:Gaussian_rough_paths}
A broad class of geometric rough paths are given by the Gaussian rough paths. Fix a complete probability space  $(\Omega, \clF, \bbP)$ supporting a $K$-dimensional Gaussian process  $\{Z_t\}_{t\le T}$ with independent components and zero mean.  Let  $R_{k}(s,t)=\bbE[Z^k_sZ^k_t]$ denote the corresponding  covariance functions and  $$R^{st}_{k,uv}=\bbE[\delta Z^k_{st} \delta Z^{k}_{uv}] =R_k(s,u)+R_k(t,v)-R_k(s,v)-R_k(t,u).$$ The existence of a rough path lift for $X$ is contingent upon sufficient rate of decay of the correlation of the increments. If for a given $q \in [1,\frac{3}{2})$, there exist a constant $C>0$ such that for  all $k\in \{1,\ldots, K\}$ and $(s,t)\in \Delta_T$,
\begin{equation}\label{ineq:Gauss_rough_cond}
\sup_{\clP([s,t]^2)}\sum_{[t_i,t_{i+1}]\times [s_i,s_{i+1}]\in \clP([s,t]^2)}|R^{t_i,t_{i+1}}_{k,s_i,s_{i+1}}|^{q}\le C |t-s|,
\end{equation}
where the supremum is taken over all finite partitions $\clP([s,t]^2)$ of the interval $[s,t]^2$, then there is a random variable $\bbZ$ and set $\bar{\Omega}\in \clF$ for which $\bbP(\bar{\Omega})=1$ and such that for all $\omega\in \bar{\Omega}$, $\bZ(\omega)=(Z(\omega),\bbZ(\omega)))\in \bclC^{\alpha}_{g,T}(\bbR^K)$ for $\alpha\in (\frac{1}{3},\frac{1}{2q})$. Furthermore, the lift is canonical in the sense that  for all $(s,t)\in \Delta_T^2,$
$$
\lim_{|\clP([s,t])|\rightarrow 0}\bbE\left|\sum_{[t_i,t_{i+1}]\in \clP([s,t])}\delta Z_{st_i}\otimes \delta Z_{t_it_{i+1}} - \bbZ_{st}\right|^2=0,
$$
where $\clP([s,t])$ denotes a finite partition of the interval $[s,t]$ and $|\clP([s,t])|$ denotes its mesh size and the integral is understood in the sense of a limit of nets.

If $X$ is stationary and
\begin{equation}\label{eq:sigma_sq}
\sigma^2_k(\tau):=R^{t(t+\tau)}_{k,t(t+\tau)}
\end{equation}
is concave and non-decreasing as a function of $\tau$ on an interval $[0,h]$ for some $h>0$ and there is a constant $C>0$ such that for all $k\in \{1,\ldots, K\}$ and $\tau\in [0,h]$, $$|\sigma^2_k(\tau)|\le C |\tau|^{\frac{1}{q}},$$ then \eqref{ineq:Gauss_rough_cond} holds. We refer the reader to \cite{friz2014course}[Ch.\ 10] or \cite{friz2010differential}[Ch.\ 15] for a more thorough exposition.  

\begin{example}[Fractional Brownian motion]
The prototypical Gaussian process  satisfying these assumptions is a $K$-dimensional  fractional Brownian motion $B^H$, $H\in (\frac{1}{3},1]$, which has the covariance function  $$R^H(s,t)= \frac{1}{2}\left[s^{2H}+t^{2H}-|t-s|^{2H}\right]\times I_K \quad \Rightarrow \quad  \sigma^2_k(\tau)=\tau^{2H},$$ 
where $I_K$ is the $K\times K$-identity matrix and $\sigma^2_k(\tau)$ is defined in \ref{eq:sigma_sq}.
Thus, $B^H$ lifts to a geometric rough path $\bB^H(\omega)=(B^H(\omega),\bbB^H(\omega))\in \bclC^{\alpha}_{g,T}(\bbR^K)$, $\alpha\in (\frac{1}{3},\frac{1}{4H})$ for all $\omega$ in a set of probability one. In particular, for $H=\frac{1}{2}$, $B:=B^{\frac{1}{2}}$ is a standard Brownian motion, $\bB(\omega)=(B(\omega),\bbB(\omega))\in \bclC^{\alpha}_{g,T}(\bbR^K)$, $\alpha\in \left(\frac{1}{3},\frac{1}{2}\right)$, and
$$
\bbB_{st}(\omega)=\left(\int_s^t\delta B_{st_2}\otimes \circ  \textnormal{d}B_{t_1}\right)(\omega), \;\; (s,t)\in \Delta_T^2.
$$
We  note that 
$$
\bbB_{st}^{lk}(\omega)\ne \int_s^t\delta B^l_{st_2}(\omega) \circ \textnormal{d}B^k_{t_1}(\omega)
$$
because stochastic integrals are defined for non-simple processes via an $L^2(\Omega)$-closure and there is no pathwise way (i.e., in the sense that it is robust under smooth approximations of the path) to make sense of  the right-hand-side other than by simply defining via the left-hand-side.
\end{example}

\begin{example}[Volterra Gaussian processes]
A Volterra kernel $K: [0,T]^2\rightarrow \bbR$  is a square integrable function such that $K(s,t)=0$ for $s\ge t$. One can find conditions on the kernels $K:[0,T]^2\rightarrow \mathbb{R}$ such that the corresponding Volterra Gaussian processes 
$$
Z_t=\int_0^TK(t,s)\textnormal{d}B_s, \quad \quad R(s,t)=\int_0^{t\wedge s}K(t,r)K(s,r)\textnormal{d}r,
$$
can be lifted to a geometric rough path. We refer the reader to \cite{cass2019stratonovich} for a more in depth discussion of Volterra Gaussian processes and even how to extend the setup to more irregular paths. Fractional Brownian motion, Riemann-Liouville, and more simply, Ornstein-Uhlenbeck processes are all  examples of Volterra Gaussian rough paths.
\end{example}

\bibliographystyle{alpha}
\bibliography{bibliography}
\end{document}